\documentclass[letterpaper,11pt,oneside,reqno]{article}

\usepackage[pdftex,backref=page,colorlinks=true,linkcolor=blue,citecolor=red]{hyperref}
\usepackage[alphabetic,nobysame]{amsrefs}

\usepackage{amsmath,amssymb,amsthm,amsfonts,mathtools}
\usepackage{graphicx,color}
\usepackage{upgreek}
\usepackage[mathscr]{euscript}

\allowdisplaybreaks
\numberwithin{equation}{section}

\usepackage{tikz}

\usepackage{array}
\usepackage{cleveref}
\usepackage{enumerate}
\usepackage[font=small]{caption} %\captionsetup

\usepackage[DIV=12]{typearea}

\newcommand{\ssp}{\hspace{1pt}}

\newcommand{\WQ}{\mathbb{W}}

\newtheorem{proposition}{Proposition}[section]
\newtheorem{lemma}[proposition]{Lemma}
\newtheorem{corollary}[proposition]{Corollary}
\newtheorem{theorem}[proposition]{Theorem}
\newtheorem{conjecture}[proposition]{Conjecture}
\theoremstyle{definition}
\newtheorem{definition}[proposition]{Definition}
\newtheorem{remark}[proposition]{Remark}
\begin{document}
\title{Colored Interacting Particle Systems on the Ring:\\Stationary Measures from Yang--Baxter Equations}

\author{Amol Aggarwal, Matthew Nicoletti, Leonid Petrov}

\date{}

\setcounter{tocdepth}{4}

\maketitle

\begin{abstract}
    Recently, there has been much progress in understanding stationary measures for colored (also called multi-species or multi-type) interacting particle systems, motivated by asymptotic phenomena and rich underlying algebraic and combinatorial structures (such as nonsymmetric Macdonald polynomials).

    In this paper, we present a unified approach to constructing stationary measures for most of the known colored particle systems on the ring and the line, including (1)~the Asymmetric Simple Exclusion Process (multispecies ASEP, or mASEP); (2)~the $q$-deformed Totally Asymmetric Zero Range Process (TAZRP) also known as the $q$-Boson particle system; (3)~the $q$-deformed Pushing Totally Asymmetric Simple Exclusion Process ($q$-PushTASEP). Our method is based on integrable stochastic vertex models and the Yang--Baxter equation. We express the stationary measures as partition functions of new ``queue vertex models'' on the cylinder. The stationarity property is a direct consequence of the Yang--Baxter equation.

    For the mASEP on the ring, a particular case of our vertex model is equivalent to the multiline queues of Martin \cite{martin2020stationary}.
	For the colored $q$-Boson process and the $q$-PushTASEP on the ring, we recover and generalize known stationary measures constructed using multiline queues or other methods by Ayyer--Mandelshtam--Martin \cite{ayyer2022modified}, \cite{Ayyer_2023}, and Bukh--Cox~\cite{bukh2019periodic}. Our proofs of stationarity use the Yang--Baxter equation and bypass the Matrix Product Ansatz (used for the mASEP by Prolhac--Evans--Mallick \cite{Prolhac_2009}).

    On the line and in a quadrant, we use the Yang--Baxter equation to establish a general colored Burke's theorem, which implies that suitable specializations of our queue vertex models produce stationary measures for particle systems on the line. We also compute the colored particle currents in stationarity.
\end{abstract}

\setcounter{tocdepth}{2}
\tableofcontents
\setcounter{tocdepth}{3}

\section{Introduction}
\label{sec:intro}

\subsection{Background}

This work connects stationary measures for colored (also referred to as multi-species or multi-type) systems of interacting particles hopping on a one-dimensional lattice (the ring or the whole line) to solvable lattice models. One of the particle systems we consider is the multi-species Asymmetric Simple Exclusion Process (\emph{mASEP}). On the ring with $N$ sites, the $n$-colored mASEP is a continuous-time Markov process under which each pair of  neighboring particles of colors $0\le i\ne j\le n$ and locations $(k,k+1)$ (mod $N$) swap with rate $1$ if $i<j$, and rate $q\in[0,1)$ if $i>j$ (color $0$ represents holes). See \Cref{fig:mASEP_intro} for an illustration.\footnote{Throughout the paper, we say that an event occurs at rate~$\alpha>0$ if the random time $\zeta$ till the occurrence is exponentially distributed with parameter~$\alpha$, that is, $\mathbb{P}(\zeta>t)=e^{-\alpha t}$ for $t\ge0$.} We also consider the multi-species $q$-TAZRP ($q$-deformed Totally Asymmetric Zero Range Process, also known as the colored stochastic $q$-Boson particle system), and the colored $q$-PushTASEP
($q$-deformed Pushing Totally Asymmetric Simple Exclusion Process). We refer to \Cref{sub:qBoson_process,sec:qPush} for definitions of these models on the ring.
\begin{figure}[htpb]
	\centering
	\includegraphics[width=.6\textwidth]{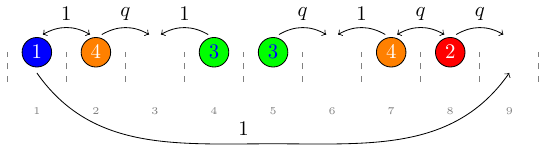}
	\caption{Rates of all possible particle hops in the mASEP on the ring with
		$N=9$ sites and $n=4$ colors (also indicated by numbers to 
		assist the printed version).
The mASEP preserves the number of particles of each type, and has a
unique stationary measure in each ``sector'' determined by
fixing the number $N_m$ of particles of each color $m=1,\ldots,n $.}
	\label{fig:mASEP_intro}
\end{figure}

The stationary measures of these particle systems have recently been the subject of a systematic investigation.
On one hand, their properties and asymptotic behavior highlight interesting physical and probabilistic phenomena, cf. Martin \cite[Section~6]{martin2020stationary}, Ayyer--Linusson \cite{ayyer2017correlations}, and Pahuja \cite{pahuja2023correlations}.
On the other hand, they have a rich underlying algebraic and combinatorial structure (in particular, deep connections to Macdonald symmetric functions and their nonsymmetric counterparts). Let us briefly review the known constructions of stationary measures, and their connections to nonsymmetric Macdonald polynomials.

\medskip

The stationary measure for the single-color mASEP (simply called \emph{ASEP}) on the ring with a given number of particles is uniform among all possible arrangements of these particles. On the line, translation invariant stationary measures are all product Bernoulli (that is, each site is independently a particle with probability $\rho$ or a hole otherwise); see Liggett \cite[Ch.~VIII]{Liggett1985}.

Prolhac--Evans--Mallick \cite{Prolhac_2009} constructed stationary measures for mASEP with an arbitrary number of colors using the \emph{Matrix Product Ansatz}, an algebraic formalism utilizing commutation relations of a family of matrices. For previous partial Matrix Product Ansatz results see the references in \cite{Prolhac_2009}. We discuss the Matrix Product Ansatz approach in the beginning of \Cref{sub:YBE_MPA} in the text. Methods for sampling from the mASEP stationary measures using combinatorial structures known as \emph{multiline diagrams} or \emph{multiline queues} were developed by Angel \cite{angel2006stationary}, Ferrari--Martin \cite{ Ferrari_2007}, and Martin \cite{martin2020stationary}.

\medskip

Connections to Macdonald symmetric and nonsymmetric polynomials served as another inspiration for studying stationary measures of interacting particle systems on the ring. Macdonald polynomials are a centerpiece of the symmetric functions theory and have wide applications to representation theory and geometry; see Macdonald \cite[Ch.~VI]{Macdonald1995}, \cite{marshall1999symmetric}. We do not focus on symmetric and nonsymmetric functions in this paper, but here we mention the necessary background.

A relationship between the mASEP stationary measures and Macdonald polynomials is first observed by Cantini--de Gier--Wheeler~\cite{cantini2015matrix}. Nonsymmetric Macdonald polynomials are constructed via multiline queues by Corteel--Mandelshtam--Williams \cite{corteel2018multiline}. An integrable vertex model for nonsymmetric Macdonald polynomials is given by Borodin--Wheeler \cite{borodin2019nonsymmetric}. One can say that the vertex model construction unifies the two points of view.

\medskip

The stochastic single-color $q$-Boson process (also referred to as the $q$-TAZRP) was introduced by Sasamoto--Wadati \cite{SasamotoWadati1998}. Its dual process, the $q$-TASEP, was extensively studied on the line by Borodin--Corwin \cite{BorodinCorwin2011Macdonald}, Borodin--Corwin--Sasamoto \cite{BorodinCorwinSasamoto2012}, Borodin--Corwin--Petrov--Sasamoto \cite{BorodinCorwinPetrovSasamoto2013}, Barraquand \cite{barraquand2015phase}, and others. On the ring, Wang--Waugh \cite{Wang2015inhomqTASEP} and Liu--Saenz--Wang \cite{liu2019integral} obtained integral formulas for transition probabilities and other observables of the $q$-Boson process.

The multi-species generalization of the $q$-Boson process is due to Takeyama \cite{Takeyama2015} and Kuniba--Maruyama--Okado \cite{kuniba2016multispecies}. Stationary measures for this process on the ring were recently connected to modified Macdonald polynomials by Ayyer--Mandelshtam--Martin \cite{ayyer2022modified}, \cite{Ayyer_2023} (see also the earlier work of Garbali--Wheeler \cite{garbali2020modified} connecting modified Macdonald polynoimals to vertex models). In contrast with the mASEP, the $q$-Boson process has spectral parameters (rapidities) attached to the sites on the ring, and these parameters enter the modified Macdonald polynomials. Ayyer--Mandelshtam--Martin revealed symmetries of the stationary measures in the parameters and utilized them to compute observables in the stationary model.

\medskip

The colored $q$-PushTASEP on the ring is less studied. On the line, it was introduced by Borodin--Wheeler \cite[Section~12.5]{borodin_wheeler2018coloured} as a degeneration of the colored stochastic higher spin six-vertex model. Like the $q$-Boson process, it contains spectral parameters, and also the capacity parameter $\mathsf{P}\in \mathbb{Z}_{\ge1}$ which is the maximum number of particles allowed at each site. A discrete time variant of the colored $q$-PushTASEP (with $q=0$ and $\mathsf{P}=1$) on the ring was introduced (under the name ``frog model'') by Bukh--Cox \cite{bukh2019periodic} in connection with the problem of the longest common subsequence of a random and a periodic word. In particular, Bukh--Cox \cite{bukh2019periodic} constructed and investigated stationary measures of the frog model.

\subsection{Main results and methods}

We unify and generalize existing constructions of stationary measures of the multi-species ASEP, the colored $q$-Boson process, and the colored $q$-PushTASEP on the ring and the line. We prove the stationarity property graphically in all cases, using the Yang--Baxter equation (discussed below in this subsection). We obtain the following results for the three particle systems:
\begin{enumerate}[$\bullet$]
    \item For mASEP, we present a vertex model interpretation of the multiline queues of Martin \cite{martin2020stationary}, which is very close to the Matrix Product Ansatz construction of Prolhac--Evans--Mallick \cite{Prolhac_2009}. Moreover, we recover the main ingredient of the Matrix Product Ansatz proof of the stationarity (the ``hat relation'') directly from the Yang--Baxter equation. On the ring, we add extra parameters to our vertex model that do not affect the stationary measures. We also connect our vertex models with extra parameters to other variants of multiline queues, including the one considered by Martin \cite[Section~7]{martin2020stationary}.

    \item For the colored $q$-Boson process, we present a vertex model construction of stationary measures on the ring and the line. On the ring, the vertex model has extra parameters that do not affect the stationary measures, but they need to be specialized to zero on the line. In a particular case of at most one particle of each color, our vertex model is equivalent to the multiline queues of Ayyer--Mandelshtam--Martin \cite{ayyer2022modified}. On the line, the parameters of the vertex model which survive are in one-to-one correspondence with colored particle densities, and we compute the colored currents in stationarity (as implicit functions of the densities of particles of each color).

    \item For the colored $q$-PushTASEP, we also construct vertex models for the stationary measures on the ring and the line. On the ring, the vertex model carries extra parameters that do not affect the stationary measures. On the line, some extra parameters are lost, and again the remaining ones exactly parameterize colored particle densities. A particular $\mathsf{P}=1$ case of the $q$-PushTASEP has the same stationary measures as the mASEP and the colored stochastic six-vertex model. On the line, we compute the colored currents for the stochastic six-vertex model and the $q$-PushTASEP in stationarity.
\end{enumerate}

To obtain all our results, we utilize \emph{integrable stochastic vertex models}, that is, whose vertex weights have a stochastic normalization and satisfy the Yang--Baxter equation. We start from stochastic $U_q(\widehat{\mathfrak{sl}}_{n+1})$ vertex weights whose explicit form was first obtained by Kuniba--Mangazeev--Maruyama--Okado
\cite{kunibaMangazeev2016stochastic} and which were studied in the stochastic context by
Bosnjak--Manga\-zeev \cite{BosnjakMangazeev2016}, Garbali--de Gier--Wheeler \cite{deGierWheeler2016}, Kuan \cite{Kuan2018}. The vertex model foundation of our work was developed by Borodin--Wheeler \cite{borodin_wheeler2018coloured} who listed various stochastic degenerations of the general $U_q(\widehat{\mathfrak{sl}}_{n+1})$ weights, leading to the mASEP, the colored $q$-Boson process, and the colored $q$-PushTASEP.

In a limit when the number of particles (corresponding to vertical arrows in vertex models) of a given color goes to infinity, the $U_q(\widehat{\mathfrak{sl}}_{n+1})$ stochastic weights degenerate into what we call the \emph{queue vertex weights}. Putting them on the cylinder, we obtain our main object --- the \emph{queue vertex model}. Its normalized partition functions serve as stationary measures for our colored stochastic particle systems on the ring. The stationarity of the normalized partition functions is a direct consequence of the Yang--Baxter equation.

For each particle system, we perform the following sequence of steps:
\begin{enumerate}[\bf1.\/]
\item We define a colored (higher spin) stochastic transfer matrix~$\mathfrak{T}$ on the cylinder. In a Poisson-type limit when the discrete time becomes continuous, large powers of $\mathfrak{T}$ converge to the Markov semigroup (also called the propagator) of the given stochastic process living on configurations of colored particles on the ring.
Each of our processes (the mASEP, the colored $q$-Boson process, and the colored $q$-PushTASEP) preserves the number of particles of each color, while~$\mathfrak{T}$ before the Poisson-type limit may not.

\item We construct a multiparameter family of transfer matrices~$\mathfrak{Q}$ on the cylinder using our queue vertex weights (see \Cref{fig:mult_queue_intro} for an illustration), such that
\begin{enumerate}[(a)\/]
    \item Their matrix elements $\langle \emptyset| \ssp\mathfrak{Q}\ssp |\mathbf{V} \rangle$ (which are, by definition, partition functions --- the sums of weights of all allowed arrow configurations) are nonzero for all terminal states~$\mathbf{V}$. Here~$\langle \emptyset|$ is the empty state (no particles present on the ring), and $\mathbf{V}$ encodes a state with prescribed particle locations and colors.
    \item The Yang--Baxter equation implies that $\mathfrak{Q}\ssp \mathfrak{T} = \mathfrak{T}\ssp \mathfrak{Q}$.
	In detail,
	we glue the cylinder with the output $\mathbf{V}$ to a sequence of stochastic~$R$ matrices,
	as shown in \Cref{fig:twisted_commutation_intro}.
	Their composition is the operator $\mathfrak{T}$.
	Together with the queue weights, these $R$ matrices satisfy the Yang--Baxter equation which allows to commute $\mathfrak{T}$ through the queue weights $\mathfrak{Q}$.
\end{enumerate}
Since $\langle \emptyset| \mathfrak{T}=\langle \emptyset|$, after a degeneration of the parameters of $\mathfrak{Q}$ (corresponding to the Poisson-type limit in $\mathfrak{T}$), the quantity $\langle \emptyset| \ssp\mathfrak{Q}$ viewed as a row vector (or, equivalently, an unnormalized probability distribution on the space of particle configurations on the ring) becomes stationary under the Markov semigroup. By restricting $\langle \emptyset| \ssp\mathfrak{Q}$ to particle configurations with fixed numbers of particles of each color (a \emph{sector}), we obtain a stationary measure for the corresponding particle system. Note that part~(a) above seems to violate the conservation of particles (the higher spin analog of the ice rule), which holds for models of six-vertex type. It is the limit to the queue vertex weights which allows this violation.

\item For each of the particle systems on the ring, we compute the queue weights explicitly (after the corresponding degeneration of the parameters of $\mathfrak{Q}$), and provide conditions on the remaining complex parameters guaranteeing the positivity of the unnormalized weights. Note that the normalized weights are automatically positive by the stationarity property and the Perron--Frobenius theorem.
\end{enumerate}

\begin{figure}[htp]
    \centering
    \includegraphics[width=\textwidth]{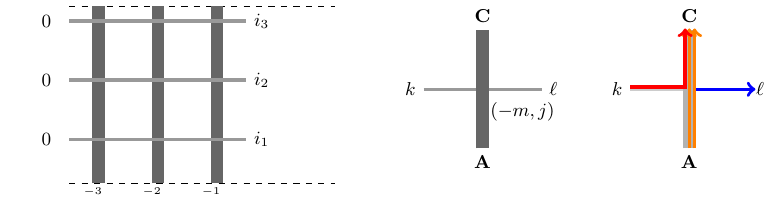}
		\caption{Left: The queue vertex model partition function $\langle \emptyset |\ssp \mathfrak{Q} \ssp | \mathbf{V} \rangle$ on the cylinder with $n=3$ colors and $N=3$ sites (the dotted lines are identified). There are no arrows incoming from the left, and the outgoing terminal state $\mathbf{V}=(i_1,i_2,i_3)$ has $i_1,i_2,i_3\in\{0,1,2,3\}$. Center: The queue vertex weight $\WQ^{(-m)}_{\mathrm{parameters}(m,j)}$ (see \Cref{def:queue_spec_defn} below) is in column $-m$, $1\le m\le N$ at position $j=1,\ldots,N$ in the transfer matrix $\mathfrak{Q}$. Right: An allowed path configuration at a queue vertex with $m=1$. Colors $1,2,3$ are, respectively, blue (densely dotted to assist the printed version), orange (dashed to assist the printed version), and red. We have $\mathbf{A} = (\infty, 0, 2)$, $k = 2$, $\mathbf{C} = (\infty, 1, 2)$, and $\ell = 1$. Infinitely many blue arrows (not depicted) pass through vertically, which allows a blue arrow to exit from the right even though none entered from the left.}
    \label{fig:mult_queue_intro}
\end{figure}
\begin{figure}[htb]
    \centering
    \includegraphics[width=\textwidth]{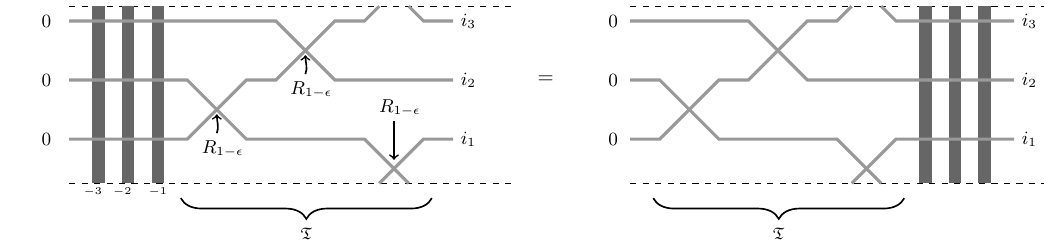}
    \caption{The main commutation relation
	$\mathfrak{Q}\ssp \mathfrak{T} = \mathfrak{T}\ssp \mathfrak{Q}$ leading to stationarity of the mASEP.
	As $\epsilon\to 0$ and under a Poisson-type rescaling of the horizontal direction to
	continuous time,
	the combination of the $R$ matrices with parameters $1-\epsilon$
	turns into the Markov semigroup of the mASEP on the ring.
For the other systems (the colored $q$-Boson process and the colored $q$-PushTASEP), we need a different configuration on the cylinder (given in \Cref{fig:straight,fig:straight_commuted} in the text) and a slightly different Poisson-type limit transition.}
    \label{fig:twisted_commutation_intro}
\end{figure}

We call the weights on the cylinder the \emph{queue vertex weights} because each column $(-m)$ of a cylinder resembles a queueing system for which the ``unused service times'' are assigned the new color $m$. Indeed, for the mASEP and (with certain restrictions) the $q$-Boson process, the output of our vertex model is the same (in distribution) as that of the multiline queues of Martin \cite{martin2020stationary} and Ayyer--Mandelshtam--Martin \cite{ayyer2022modified}, respectively.

For the mASEP, we explain how to set the parameters to exactly match a certain specialization of our stationary measure with the matrix product stationary measure constructed by Prolhac--Evans--Mallick \cite{Prolhac_2009}. We also show how the underlying algebraic apparatus for the Matrix Product Ansatz can be derived from the Yang--Baxter equation. Namely, we realize the product ansatz matrices satisfying the so-called ``hat relation'' (identity \eqref{eq:hat_matrices} in the text) as vertex model partition functions. This highlights the Yang--Baxter structure of the Matrix Product Ansatz for the multispecies ASEP which was previously less transparent, cf. Arita--Ayyer--Mallick--Prolhac \cite{arita2012generalized} (however, it may be derived from Cantini--de Gier--Wheeler \cite{cantini2015matrix}; see also Crampe--Ragoucy--Vanicat \cite{crampe2014integrable} for more connections between the two structures). Furthermore, we describe how more general solutions~$A, D, E$ of the quadratic algebra relations of the Matrix Product Ansatz (formula \eqref{eq:ADE_relations} in the text) can be derived from our queue vertex weights.

\medskip

In \Cref{sec:quarter_plane}, we apply similar vertex model
techniques to construct stationary distributions for the
mASEP, the colored $q$-Boson process, and the colored
$q$-PushTASEP on the line~$\mathbb{Z}$. We start in a
quarter-plane setup $\mathbb{Z}_{\ge0}^{2}$ (see \Cref{fig:qp_stationary_intro}), and use the
Yang--Baxter equation to justify that placing queue vertex weights
on the bottom and the left boundaries produces a random
configuration of colored paths in $\mathbb{Z}_{\ge0}^{2}$.
This random configuration is stationary with respect to
arbitrary shifts of the origin by vectors from
$\mathbb{Z}_{\ge0}^2$. This stationarity may be viewed as a
colored analog of Burke's theorem from queueing theory (in the sense
adopted in the study of percolation and random polymers, cf. \cite{Seppalainen2012}).
Looking at the random configuration of colored paths in the quarter plane
away from the origin, we obtain a translation invariant
stationary distribution on colored paths in the full plane $\mathbb{Z}^2$.
The parameters entering the queue vertex weights
along the boundary correspond to the densities of particles (paths) of each color.
We also compute the relation between the parameters along the bottom
and the left boundaries, which corresponds to slopes of the colored paths,
or currents in the terminology of particle systems.
In the single-color case, the quarter plane construction
involving
single-color queue vertex weights at the left boundary of the quadrant
appeared in Aggarwal \cite{Amol2016Stationary}.
To pass from quarter-plane systems to particle systems on $\mathbb{Z}$,
one has to take a continuous-time limit in which the vertical, horizontal, or diagonal
direction scales to the continuous time.

\medskip

Our method of producing stationary distributions as partition functions is conceptually reminiscent of the Bethe Ansatz. In Bethe Ansatz, eigenvectors of a transfer matrix of a quantum integrable model are constructed by applying certain other transfer matrices (with specially chosen parameters satisfying algebraic equations) to the vacuum vector $\langle \emptyset|$. For stationary measures of our Markov processes, we only need the leading (Perron--Frobenius) eigenvector of~$\mathfrak{T}$ having the eigenvalue $1$. We do not investigate further connections of our constructions to the Bethe Ansatz or the algebraic equations.

\medskip

% The rest of the paper is organized as follows. In Section~\ref{sec:colored_YBE} we outline the essential facts about the $sl_n$-related vertex model that we work with throughout, including the vertex weights and Yang--Baxter equations. In Section~\ref{sec:main_vertex_statement} we flesh out the stationarity arguments outlined above, for two types of pre-limit discrete time transition matrices. In Sections~\ref{sec:ASEP_matrix_products}, \ref{sec:qBoson}, and \ref{sec:qPush}, we take the appropriate degenerations to obtain stationary measures for mASEP, colored $q$-Boson, and colored $q$-PushTASEP, respectively. Finally, in Section \ref{sec:quarter_plane}, we define translation invariant stationary measures for each system on the line~$\mathbb{Z}$, by using a certain stationary quarter plane construction.

\subsection{Related work in progress}

While preparing the manuscript, we learned about two related works in progress. One by Corteel--Keating \cite{Corteel_Keating_2023_progress} defines a new particle system with zero-range interactions on the ring. Its stationary distribution is expressed as a queue-like vertex model on the cylinder with fermionic weights related to the algebra $U_q(\widehat{\mathfrak{s l}}(1 | n))$. These weights were recently investigated by Corteel--Gitlin--Keating--Meza \cite{corteel2022vertex} and Aggarwal--Borodin--Wheeler \cite{agg-bor-wh2020-sl1n}. In particular, instead of the Macdonald polynomials, they are related to the Lascoux--Leclerc--Thibon (LLT) polynomials \cite{lascoux1997ribbon}.

Another work in progress by Angel--Ayyer--Martin \cite{Angel_Ayyer_Martin_2023_progress} considers stationary measures for the colored $q$-PushTASEP on the ring (with capacity $\mathsf{P}=1$), and connects them to the multiline queues of Corteel--Mandelshtam--Williams \cite{corteel2018multiline}. The approach to this proof is different from ours, and it relies on properties of symmetric and nonsymmetric Macdonald polynomials and related functions.

\begin{figure}
    \centering
    \includegraphics[scale=.7]{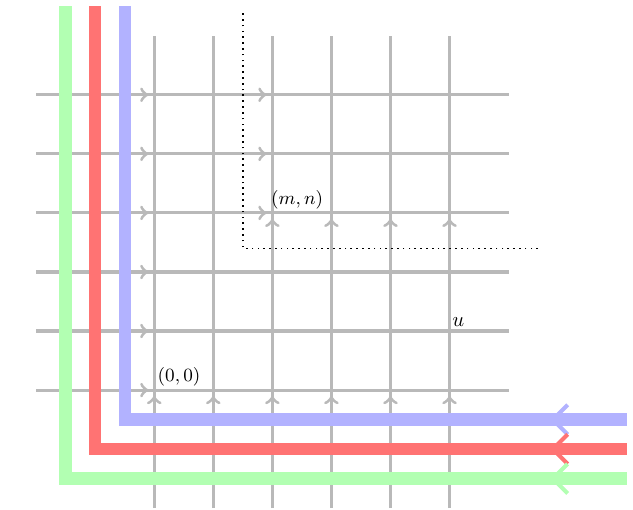}\qquad
    \raisebox{40pt}{\includegraphics[scale=.2]{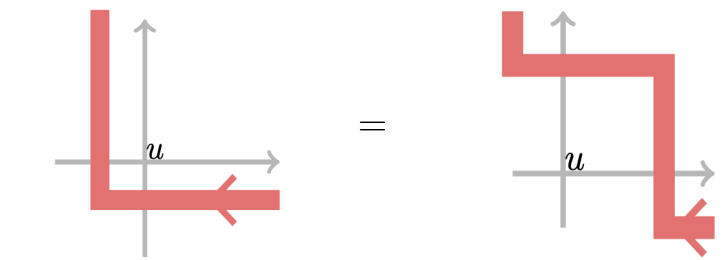}}
		\caption{The construction of the stationary quarter plane configuration for the colored stochastic six-vertex model. The arrow configuration entering the quadrant $\mathbb{Z}_{\geq 0}^2$ is generated by the stochastic vertex model with thick lines along the boundary (the vertex weights here are the \emph{queue vertex weights}), and has the same distribution as the configuration of arrows crossing the boundary of the $(m,n)$-shifted quadrant (the dotted line). Each of the thick lines carries its own color (type) of arrows.
		The configuration of arrows entering from the far right along the thick lines is in equilibrium and has a specific distribution (namely, the stationary distribution of a Markov chain describing the evolution of ``queues in tandem''). Stationarity can be proved by repeated applications of the Yang--Baxter equation, shown on the right. Details of this construction are in \Cref{sec:quarter_plane}.}
    \label{fig:qp_stationary_intro}
\end{figure}

\subsection{Outline}

In \Cref{sec:colored_YBE}, we recall the colored stochastic vertex weights (together with their fully fused version) and the Yang--Baxter equations for them. We define the queue vertex weights, which arise in the limit of the stochastic vertex weights when the number of vertical arrows of a given color goes to infinity. Putting the queue vertex weights on the cylinder, we obtain our main object --- the queue vertex model.

In \Cref{sec:main_vertex_statement}, we employ the Yang--Baxter equation to show that the measure on particle configurations on the ring coming from the queue vertex model is stationary under the twisted and the straight cylinder Markov transition operators. In full generality, these Markov operators are formal (may have negative matrix elements).

In \Cref{sec:ASEP_matrix_products}, we take a scaling limit under which the twisted cylinder Markov operator converges to the infinitesimal generator of the mASEP (an actual, not formal, Markov operator). The corresponding limit of the queue vertex model yields a stationary measure for the mASEP. Under a specialization of the parameters, we identify our queue vertex model with the multiline queue system of Martin \cite{martin2020stationary} and connect our constructions to the Matrix Product Ansatz of Prolhac--Evans--Mallick \cite{Prolhac_2009}. Moreover, a different parameter specialization presumably relates our queue vertex model to the alternative multiline queue system also considered by Martin \cite[Section~7]{martin2020stationary}.

In \Cref{sec:qBoson,sec:qPush}, we treat two other colored particle systems --- the $q$-Boson process (also known as the $q$-TAZRP) and the $q$-PushTASEP. For the $q$-Boson process, in a particular case of at most one particle of each color, we identify our queue vertex model with the multiline queueing system recently introduced by Ayyer--Mandelshtam--Martin \cite{ayyer2022modified}.

In \Cref{sec:quarter_plane}, we consider colored stochastic vertex models in the quadrant and use the Yang--Baxter equation to prove a new colored generalization of Burke's theorem. It implies that when put on the infinite vertical strip instead of the cylinder, our queue vertex models produce stationary distributions for the interacting particle systems (mASEP, colored $q$-Boson process, and $q$-PushTASEP) on the line.
In \Cref{sec:merging_of_colors}, we show that the stationary measures for our colored particle systems on the line respect the operations of color merging (when two or more colors are declared the same). The proof also relies on applying the Yang--Baxter equation to a stochastic vertex model in the quadrant.

\subsection{Acknowledgments}

We are grateful to Alexei Borodin, Kirone Mallick, James Martin, Ananth Sridhar, and Lauren Williams for helpful discussions. Amol Aggarwal was partially supported by a Packard Fellowship for Science and Engineering, a Clay Research Fellowship, by NSF grant DMS-1926686, and by the IAS School of Mathematics. The work of Matthew Nicoletti and Leonid Petrov was partially supported by the NSF grants DMS-1664617 and DMS-2153869, by the 4-VA Virginia Collaborative Research Initiative, and by the Simons Collaboration Grant for Mathematicians 709055.

%------------------------------------------------------------------------------------------
%------------------------------------------------------------------------------------------
%------------------------------------------------------------------------------------------
%------------------------------------------------------------------------------------------
%------------------------------------------------------------------------------------------

\section{Colored vertex weights and Yang--Baxter equation}
\label{sec:colored_YBE}

In this section, we collect the necessary formulas for the vertex weights of the colored stochastic vertex model from \cite{borodin_wheeler2018coloured}. Algebraically, this model is powered by the quantum affine Lie algebra $U_q(\widehat{\mathfrak{sl}}_{n+1})$, where $n$ is the number of colors. A stochastic normalization of the $U_q(\widehat{\mathfrak{sl}}_{n+1})$ vertex weights first appeared in \cite{kunibaMangazeev2016stochastic}; see also \cite{BosnjakMangazeev2016}, \cite{deGierWheeler2016}, \cite{Kuan2018}. In certain degenerations, which we recall in \Cref{sec:ASEP_matrix_products,sec:qBoson} below, the colored stochastic vertex model turns into multi-species interacting particle systems. We start with the vertex weights and the Yang--Baxter equation; then, we proceed to the fused weights and define their new \emph{queue specialization}.

\subsection{Vertex weights}
\label{sub:vertex_weights_text}

Fix the number of colors $n\ge1$.
The higher spin
$U_q(\widehat{\mathfrak{sl}}_{n+1})$
stochastic vertex weights
$L_{s,x}(\mathbf{A},k;\mathbf{B},\ell)$
are indexed by the
following data:
\begin{enumerate}[$\bullet$]
	\item The quantum parameter $q\in[0,1)$,
		which is fixed throughout this section;
	\item The spectral parameter $x$ and the spin parameter $s$, which may depend
		on the vertex;
	\item The configurations of incoming and outgoing arrows
		$(\mathbf{A},k;\mathbf{B},\ell)$, where $k,\ell\in\left\{ 0,1,\ldots,n  \right\}$, and
		$\mathbf{A}, \mathbf{B}$ are $n$-tuples
		$(A_1,\ldots,A_n ),(B_1,\ldots,B_n )$, where $A_i,B_j\in \mathbb{Z}_{\ge0}$.
		Here $k\ge1$ represents an arrow of color $k$ entering from the left, and
		$k=0$ corresponds to no arrows entering from the left; similarly, $\ell$ encodes the
		exiting arrows to the right.
		Each $A_i$ is the number of arrows of color $i$ entering from the bottom, and $B_j$ is the number of arrows of color~$j$ exiting from the top.
	\end{enumerate}
The arrow counts $(\mathbf{A},k;\mathbf{B},\ell)$ must satisfy the \emph{arrow conservation property}
(a higher spin analog of the \emph{ice rule}),
with the understanding that all arrows go in the up or right direction. Let $\mathbf{e}_1,\ldots,\mathbf{e}_n$ be the standard basis in $\mathbb{Z}^n$, then, the arrow conservation is
\begin{equation*}
	\mathbf{A}+\mathbf{e}_k\ssp \mathbf{1}_{k\ge1}=
	\mathbf{B}+\mathbf{e}_{\ell}\ssp \mathbf{1}_{\ell\ge1}.
\end{equation*}
Here and throughout the paper,
$\mathbf{1}_{E}$ denotes the indicator
of the event or condition $E$.
For $1\le k,\ell \le n$, define
\begin{equation*}
\mathbf{A}^{+}_{k}
\coloneqq
\mathbf{A} + \mathbf{e}_k,
\hspace{8pt}
\mathbf{A}^{-}_{k}
\coloneqq
\mathbf{A} - \mathbf{e}_k,
\hspace{8pt}
\mathbf{A}^{+-}_{k\ell}
\coloneqq
\mathbf{A} + \mathbf{e}_k - \mathbf{e}_\ell,
\hspace{8pt}
|\mathbf{A}|
\coloneqq
\sum\nolimits_{k=1}^{n} A_k,
\hspace{8pt}
A_{[k,\ell]}
\coloneqq
\sum\nolimits_{i=k}^{\ell} A_k.
\end{equation*}
The values of all the vertex weights $L_{s,x}(\mathbf{A},k;\mathbf{B},\ell)$
are listed in the table
in \Cref{fig:L_weights}.
In
\cite[Section~2]{borodin_wheeler2018coloured},
they are denoted by
$\tilde{L}_{s,x}(\mathbf{A},k;\mathbf{B},\ell)$,
and they differ from
\cite[(2.2.2)]{borodin_wheeler2018coloured}
by the factor $(-s)^{\mathbf{1}_{\ell >0}}$.
However, in this paper, we work with stochastic weights from the beginning, and remove the tilde from the notation.

\begin{figure}[htpb]
	\centering
	\includegraphics[width=.7\textwidth]{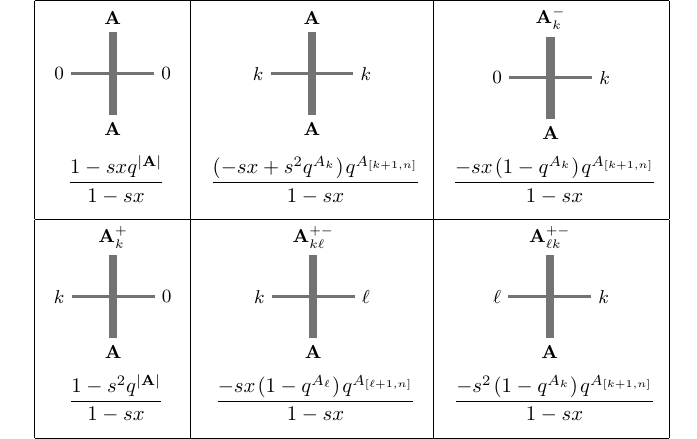}
	\caption{Colored stochastic higher spin
		vertex weights $L_{s,x}$. Here $1\le k<\ell \le n$,
	and all other values of $L_{s,x}$ are zero.}
	\label{fig:L_weights}
\end{figure}

\begin{proposition}
	\label{prop:stochastic}
	The vertex weights $L_{s,x}$ satisfy the sum-to-one property
	\begin{equation}\label{eq:sum_to_one}
		\sum_{\mathbf{B}\in \mathbb{Z}_{\ge0}^{n}}\sum_{\ell=0}^{n}
		L_{s,x}(\mathbf{A},k;\mathbf{B},\ell)=1
	\end{equation}
	for any fixed $\mathbf{A},k$.
	Moreover, if $q\in[0,1)$,
	\begin{equation}
		\label{eq:nonnegativity_of_vertex_weights_condition}
		-sx>0, \qquad -s^2\ge0,\qquad -sx+ s^2\ge0,
	\end{equation}
	then
	all the vertex weights
	$L_{s,x}(\mathbf{A},k;\mathbf{B},\ell)$ are nonnegative.
\end{proposition}
\begin{proof}
	The sum-to-one property is
	\cite[Proposition~2.5.1]{borodin_wheeler2018coloured}.
	The nonnegativity under conditions \eqref{eq:nonnegativity_of_vertex_weights_condition}
	is straightforward.
\end{proof}

Conditions \eqref{eq:nonnegativity_of_vertex_weights_condition} mean that $s$ and $x$ must be purely imaginary numbers. Observe that the weights $L_{s,x}$ contain $s^2$ and $-sx$ as two independent parameters, and they are more convenient to formulate the nonnegativity.

\medskip

A natural point of view is to interpret the vertex weights~$L_{s, x}(\mathbf{A}, k; \mathbf{B}, \ell)$ as matrix elements of a linear operator~$\mathbb{C}^{n+1} \otimes V \rightarrow \mathbb{C}^{n+1} \otimes V $, where $\mathbb{C}^{n+1}$ has the standard basis~$\{| j \rangle \}_{j=0}^n$, and~$V$ has the basis~$\{| \mathbf{A} \rangle \}_{\mathbf{A} \in \mathbb{Z}_{\geq 0}^n}$. For either space, we denote vectors of the dual basis by~$\langle \mathbf{v} |$, and for tensor products of vectors (or dual vectors), we use the notation~$|\mathbf{v}, \mathbf{A} \rangle \coloneqq | \mathbf{v} \rangle \otimes | \mathbf{A} \rangle$. For the tensor product~$\mathbb{C}^{n+1} \otimes V $ we use the basis~$\{|j, \mathbf{A} \rangle \}_{j \in\{ 0,\dots, n\}, \mathbf{A} \in \mathbb{Z}_{\geq 0}^n}$, and for its dual we use the basis~$\{\langle j, \mathbf{A} | \}_{j \in\{ 0,\dots, n\}, \mathbf{A} \in \mathbb{Z}_{\geq 0}^n}$. In these bases the operator~$\mathscr{L}_{s, x}$ corresponding to~$L_{s, x}$ acts as
\begin{align}
  \langle k, \mathbf{A} | \ssp \mathscr{L}_{s, x} \ssp  |\ell, \mathbf{B} \rangle = L_{s, x}(\mathbf{A}, k; \mathbf{B}, \ell). \label{eq:L_mat}
\end{align}
In this way, pairs of dual and primal basis vectors of~$\mathbb{C}^{n+1} \otimes V$ with nonzero~$\mathscr{L}_{s, x} $ matrix entries correspond precisely to the allowed local configurations of paths at a vertex, displayed in \Cref{fig:L_weights}.

\begin{remark}[Finite-spin reduction]
	\label{rmk:finite_spin}
	For generic~$s \in \mathbb{C}$, the operator $\mathscr{L}_{s, x} $ has infinitely many nonzero matrix entries, so any number of paths can occupy the vertical edges. If, on the other hand,~$s = q^{-\frac{\mathsf{M}}{2}}$ for some~$\mathsf{M} \in \mathbb{Z}_{\ge1}$, then let us set by definition~$L_{s, x}(\mathbf{A}, k; \mathbf{B}, \ell) = 0$ unless~$ |\mathbf{A}|, |\mathbf{B}| \leq \mathsf{M}$. Note that $L_{s,x}(\mathbf{A},k;\mathbf{A}_k^+,0)$ vanishes automatically if $|\mathbf{A}|=\mathsf{M}$, so vertices with $|\mathbf{A}|>\mathsf{M}$ or $|\mathbf{B}|>\mathsf{M}$ cannot be created from the stochastic evolution started from a configuration where all $|\mathbf{A}|\le \mathsf{M}$. We see that for $s=q^{-\mathsf{M}/2}$, at most~$\mathsf{M}$ vertical paths may occupy the vertical edges. In this case,~$\mathscr{L}_{s, x} $ acts in the finite-dimensional subspace~$\mathbb{C}^{n+1} \otimes V_\mathsf{M}$, with $V_\mathsf{M} \subset V$ spanned by $\{|\mathbf{A}\rangle\}_{|\mathbf{A}| \leq \mathsf{M}}$.
\end{remark}

\subsection{Yang--Baxter equation}
\label{sub:YBE_text}

Let us define the following cross vertex weights $R_{z}(i,j;k,\ell)$, originally due to~\cite{Jimbo:1985ua} and \cite{bazhanov1985trigonometric}:
\begin{equation}
	\label{eq:R_matrix_nonfused}
	\begin{split}
		&\hspace{6pt}
		\left.
		R_z(i,i;i,i)
		\coloneqq
		1,
		\quad
		i \in \{0,1,\dots,n \},
		\right.
		\\[5pt]
		&
		\left.
		\begin{array}{ll}
		R_z(j,i;j,i)
		\coloneqq
		\dfrac{q(1-z)}{1-qz},
		&
		\quad
		R_z(i,j;i,j)
		\coloneqq
		\dfrac{1-z}{1-qz}
		\\[12pt]
		R_z(j,i;i,j)
		\coloneqq
		\dfrac{1-q}{1-qz},
		&
		\quad
		R_z(i,j;j,i)
		\coloneqq
		\dfrac{(1-q)z}{1-qz}
		\end{array}
		\right\}
		\quad
		i,j \in \{0,1,\dots,n \},
		\quad  i<j.
	\end{split}
\end{equation}
These weights also satisfy the sum-to-one property: $\displaystyle\sum\nolimits_{k,\ell=0}^{n} R_{z}(i,j;k,\ell)=1$ for any $i,j$.
They are nonnegative if $0\le z\le 1$.

The vertex weights~$R_{z}(i,j;k,\ell)$ can also be regarded as matrix elements of an operator~$\mathscr{R}_z$ acting in $\mathbb{C}^{n+1} \otimes \mathbb{C}^{n+1}$, namely, $\langle j,i | \mathscr{R}_z | \ell,k \rangle = R_{z}(i,j;k,\ell)$.

One can check that $R_z$ is the spin-$\frac{1}{2}$ reduction of $L_{s,x}$ (as in \Cref{rmk:finite_spin}):
\begin{equation}
	\label{eq:R_matrix_spin_one_half_reduction}
	R_{z}(i,j;k,\ell)=
	L_{q^{-1/2},\ssp z^{-1}q^{-1/2}}(\mathbf{e}_i \mathbf{1}_{i\ge1},j;\mathbf{e}_k \mathbf{1}_{k\ge 1},\ell),\qquad
	i,j,k,\ell\in\left\{ 0,1,\ldots,n  \right\}.
\end{equation}
In the right-hand side, if~$i=0$, then $\mathbf{e}_i \mathbf{1}_{i\ge1}=( 0,\ldots,0  )$ ($n$ zeroes), which corresponds to no arrows at an edge.

\medskip

The weights $L_{s,x}$ and $R_{z}$ satisfy the following Yang--Baxter equation:
\begin{proposition}
	[{\cite[(2.3.1) and Corollary B.4.3]{borodin_wheeler2018coloured}}]
	\label{prop:YBE_nonfused_RLL}
	For any fixed $i_1,i_2,j_1,j_2\in \left\{ 0,1,\ldots,n  \right\}$
	and $\mathbf{A}, \mathbf{B}\in \mathbb{Z}_{\ge0}^{n}$, we have
	\begin{equation}
		\label{eq:YBE_nonfused_RLL}
		\begin{split}
			&\sum_{\mathbf{K}\in \mathbb{Z}_{\ge0}^{n}}
			\ssp
			\sum_{k_1,k_2=0}^{n}
			L_{s,y}(\mathbf{A},i_2;\mathbf{K},k_2)\ssp
			L_{s,x}(\mathbf{K},i_1;\mathbf{B},k_1)\ssp
			R_{y/x}(k_2,k_1;j_2,j_1)
			\\&\hspace{40pt}=
			\sum_{\mathbf{K}\in \mathbb{Z}_{\ge0}^{n}}
			\ssp
			\sum_{k_1,k_2=0}^{n}
			R_{y/x}(i_2,i_1;k_2,k_1)\ssp
			L_{s,y}(\mathbf{K},k_2;\mathbf{B},j_2)\ssp
			L_{s,x}(\mathbf{A},k_1;\mathbf{K},j_1).
		\end{split}
	\end{equation}
	See \Cref{fig:YBE_nonfused_RLL} for an illustration.  Note that the summations in both sides of \eqref{eq:YBE_nonfused_RLL} are actually finite.
\end{proposition}

\begin{figure}[htpb]
	\centering
	\includegraphics[width=.9\textwidth]{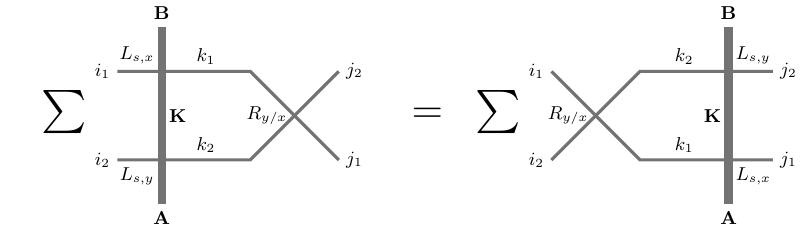}
	\caption{An illustration of the Yang--Baxter equation
	\eqref{eq:YBE_nonfused_RLL} in \Cref{prop:YBE_nonfused_RLL},
	where the sums in both sides
	are taken over all $k_1,k_2\in \left\{ 0,1,\ldots,n  \right\}$
	and $\mathbf{K}\in \mathbb{Z}_{\ge0}^n$.}
	\label{fig:YBE_nonfused_RLL}
\end{figure}

\subsection{Fused weights}
\label{sub:fusion_text}

The vertex weights $L_{s,x}(\mathbf{A},k;\mathbf{B},\ell)$
are higher spin (that is, they allow multiple arrows per edge) in
the vertical direction. For $s=q^{-1/2}$, they reduce to
the fundamental $R$-matrix $R_z(a,k;b,\ell)$, where
$a,k,b,\ell\in\left\{ 0,1,\ldots,n  \right\}$; see
\eqref{eq:R_matrix_spin_one_half_reduction}.  The inverse
procedure for constructing $L_{s,x}$ from $R_z$ is called
\emph{fusion} and dates back to
\cite{KulishReshSkl1981yang}.
In the uncolored case, it was put into
probabilistic context in \cite{CorwinPetrov2015arXiv},
\cite{BorodinPetrov2016inhom}.
The colored fusion is described in, e.g.,
\cite[Appendix~B]{borodin_wheeler2018coloured}.
The formula for the fully fused stochastic colored vertex weights
$W_{x,\mathsf{L},\mathsf{M}}(\mathbf{A},\mathbf{B};\mathbf{C},\mathbf{D})$
coming from
$U_q(\widehat{\mathfrak{sl}}_{n+1})$
is obtained in \cite{BosnjakMangazeev2016};
see also
\cite[Appendix~C]{borodin_wheeler2018coloured}.
Here we recall the stochastic vertex weights
fused in both the horizontal and vertical directions.

We need the standard $q$-Pochhammer symbols notation:
\begin{equation*}
	(a ; q)_k\coloneqq (1-a)(1-a q) \ldots(1-a q^{k-1}), \quad k \in \mathbb{Z}_{\geq 0},
\end{equation*}
and $(z ; q)_{\infty}\coloneqq \prod_{i=0}^{\infty}\left(1-z q^i\right)$ is a convergent infinite product because $q\in[0,1)$.

For $\mathbf{A},\mathbf{B}\in \mathbb{Z}_{\ge0}^{n}$
such that $A_i\le B_i$ for all $i$, define
\begin{equation}
	\label{eq:capital_Phi_coloreq_sqW}
	\Phi(\mathbf{A}, \mathbf{B} ; x, y)\coloneqq
	\frac{(x ; q)_{|\mathbf{A}|}(y / x ; q)_{|\mathbf{B}-\mathbf{A}|}}{(y ; q)_{|\mathbf{B}|}}\ssp
	(y / x)^{|\mathbf{A}|}
\ssp
q^{\sum_{1\le i<j\le n}\left(B_i-A_i\right) A_j}
\prod_{i=1}^n
\frac{(q;q)_{B_i}}{(q;q)_{A_i}(q;q)_{B_i-A_i}}.
\end{equation}
For any fixed $\mathbf{B}\in \mathbb{Z}_{\ge0}^{n}$, we have
\begin{equation}
	\label{eq:capital_Phi_coloreq_sum_to_one}
	\sum\nolimits_{\mathbf{A}\in \mathbb{Z}_{\ge0}^n
	\colon A_i\le B_i\textnormal{ for all $i$}}
	\Phi(\mathbf{A}, \mathbf{B} ; x, y)=1.
\end{equation}

With this notation,
if $\mathbf{A},\mathbf{B},\mathbf{C},\mathbf{D}
\in\mathbb{Z}_{\ge0}^{n}$, we define the vertex weights
\begin{equation}
	\label{eq:fully_fused_stochastic_weights}
	\begin{split}
		&W_{x,\mathsf{L},\mathsf{M}}(\mathbf{A},\mathbf{B};\mathbf{C},\mathbf{D})
		\coloneqq
		\mathbf{1}_{\mathbf{A}+\mathbf{B}=\mathbf{C}+\mathbf{D}}
		\cdot
		x^{|\mathbf{D}-\mathbf{B}|}
		\ssp
		(q^{\mathsf{L}})^{|\mathbf{A}|}
		\ssp
		(q^{\mathsf{M}})^{-|\mathbf{D}|}
		\\
		&\hspace{60pt}\times
		\sum\nolimits_{\mathbf{P}}
		\Phi(\mathbf{C}-\mathbf{P}, \mathbf{C}+\mathbf{D}-\mathbf{P} ; q^{\mathsf{L}-\mathsf{M}} x, q^{-\mathsf{M}} x )
		\ssp
		\Phi(
		\mathbf{P}, \mathbf{B} ; q^{-\mathsf{L}} / x, q^{-\mathsf{L}}
		).
	\end{split}
\end{equation}
The sum in \eqref{eq:fully_fused_stochastic_weights}
is finite and
is taken over all $\mathbf{P}\in \mathbb{Z}_{\ge0}^{n}$
such that $0\le P_i\le \min\left( B_i,C_i \right)$
for all $i$.

\begin{remark}
	\label{rmk:finite_spin_reduction_fully_fused}
	The parameters $\mathsf{L},\mathsf{M}$
	enter the vertex weights
	\eqref{eq:fully_fused_stochastic_weights}
	only through the powers $q^{\mathsf{L}},q^{\mathsf{M}}$.
	Moreover, the weights depend on
	$q^{\mathsf{L}}$ and $q^{\mathsf{M}}$ in a rational manner.
	Specializing $\mathsf{L},\mathsf{M}$, or both to
	positive integers leads to finite-spin reduction as in
	\Cref{rmk:finite_spin}.
	The integers $\mathsf{L}$ and
	$\mathsf{M}$ correspond to the horizontal and the
	vertical edge capacities, respectively.
	Outside of the
	finite-spin specializations, we may view
	$q^{\mathsf{L}},q^{\mathsf{M}}$
	as independent complex parameters of the weights.
\end{remark}

The weights
$W_{x,\mathsf{L},\mathsf{M}}$
\eqref{eq:fully_fused_stochastic_weights}
satisfy a version of the Yang--Baxter equation,
see formulas (C.1.2)--(C.1.3) in
\cite{borodin_wheeler2018coloured}.
Graphically, this equation is
similar to the one illustrated in
\Cref{fig:YBE_nonfused_RLL},
but the weights $L_{s,x},L_{s,y}$, and $R_{y/x}$
must be replaced with, respectively,
$W_{\frac{x}{z},\mathsf{L},\mathsf{N}}$,
$W_{\frac{y}{z},\mathsf{M},\mathsf{N}}$,
and
$W_{\frac{x}{y},\mathsf{L},\mathsf{M}}$.
The summation in the Yang--Baxter equation for the $W$'s
goes over triples of elements from $\mathbb{Z}_{\ge0}^n$.

The weights \eqref{eq:fully_fused_stochastic_weights} sum to one~\cite[(C.1.5)]{borodin_wheeler2018coloured},
\begin{equation}
	\label{eq:fully_fused_stochastic_weights_sum_to_one}
	\sum\nolimits_
	{\mathbf{C},\mathbf{D}\in \mathbb{Z}_{\ge 0}^n}
	W_{x,\mathsf{L},\mathsf{M}}(
	\mathbf{A},\mathbf{B};\mathbf{C},\mathbf{D}
	)=1,
	\qquad \mathbf{A},\mathbf{B}\in \mathbb{Z}_{\ge0}^n,
\end{equation}
and reduce to the weights $L_{s,x}$ (spin-$\frac{1}{2}$
in the horizontal direction and higher spin in the vertical direction)
and $R_z$
(spin-$\frac{1}{2}$ in both directions)
from \Cref{sub:vertex_weights_text,sub:YBE_text}
as follows~\cite[Proposition C.1.4 and formula (C.2.2)]{borodin_wheeler2018coloured}:
\begin{equation}
	\label{eq:fully_fused_reduction_to_L}
	\begin{split}
		R_{z}(i,j;k,\ell)
		&=
		W_{z^{-1},1,1}
		(\mathbf{e}_i\mathbf{1}_{i\ge1},
		\mathbf{e}_j\mathbf{1}_{j\ge1};
		\mathbf{e}_k\mathbf{1}_{k\ge1},
		\mathbf{e}_\ell\mathbf{1}_{\ell\ge1}
		),
		\\
		L_{s,x}(\mathbf{A},b;\mathbf{C},d)
		&=
		W_{x/s,1,\mathsf{N}}
		(
			\mathbf{A},\mathbf{e}_b\mathbf{1}_{b\ge1};
			\mathbf{C},\mathbf{e}_d\mathbf{1}_{d\ge1}
		)
		\big\vert_{q^{\mathsf{N}}=s^{-2}},
	\end{split}
\end{equation}
where
$i,j,k,\ell,b,d\in \left\{ 0,1,\ldots,n  \right\}$
and $\mathbf{A},\mathbf{C}\in \mathbb{Z}_{\ge0}^{n}$.
In the second line in \eqref{eq:fully_fused_reduction_to_L}, the parameter $s=q^{-\mathsf{N}/2}$ is real for $q\in[0,1)$ and $\mathsf{N}\in \mathbb{Z}_{\ge1}$. However, for the parameters $(s,x)$ to satisfy \eqref{eq:nonnegativity_of_vertex_weights_condition}, we let $\mathsf{N}$ not to be an integer, and treating $q^\mathsf{N}$ as a generic complex parameter, see \Cref{rmk:finite_spin_reduction_fully_fused}.

\subsection{Queue specialization}
\label{sub:mqueue_spec}

Here we define a procedure which we call the
\emph{queue specialization} of the
fully fused stochastic colored vertex weights
\eqref{eq:fully_fused_stochastic_weights}.
This specialization depends on an integer $1\le m\le n$ and on three parameters
$u,s_1,s_2$, and proceeds in the following manner:
\begin{enumerate}[$\bullet$]
	\item
		First, encode the parameters $\mathsf{L}$ and $\mathsf{M}$
		through $s_1,s_2\in \mathbb{C}$ as
		$q^{-\mathsf{L}}=s_1^2$, $q^{-\mathsf{M}}=s_2^2$,
		and let the spectral parameter be $x=s_1s_2^{-1}u$.
		This is just a change of variables (as in \Cref{rmk:finite_spin_reduction_fully_fused}),
		which will be useful in
		subsequent computations.
	\item After that, take the limit of
		$W_{s_1s_2^{-1}u,\mathsf{L},\mathsf{M}}
		(\mathbf{A},\mathbf{B};\mathbf{C},\mathbf{D})$
		as $A_m,C_m\to+\infty$ such that all other coordinates
		of $\mathbf{A},\mathbf{C}$, as well as the whole tuples
		$\mathbf{B},\mathbf{D}$ are fixed,
		and $A_m-C_m=D_m-B_m$ is also fixed.\footnote{Throughout
		this limit, we assume that $0<q<1$, and later we
		will specialize the limit to $q=0$ where needed.
		In other words, $q^0$ should be treated as $1$ before
		and after the limit.} The latter condition follows from the
		arrow conservation.
\end{enumerate}

\begin{lemma}
	\label{lemma:queue_spec_exists}
	The limit of
	the vertex weights
	$W_{s_1s_2^{-1}u,\mathsf{L},\mathsf{M}}
	(\mathbf{A},\mathbf{B};\mathbf{C},\mathbf{D})$
	described above exists. It is given by
 	\begin{equation}
	\label{eq:queue_spec_fully_fused}
	\begin{split}
		&
		\lim_{A_m,C_m\to +\infty}
		W_{s_1s_2^{-1}u,\mathsf{L},\mathsf{M}}
		(\mathbf{A},\mathbf{B};\mathbf{C},\mathbf{D})=
		\mathbf{1}_{\mathbf{A}+\mathbf{B}=\mathbf{C}+\mathbf{D}}
		\cdot
		\mathbf{1}_{D_1=\ldots=D_{m-1}=0}
		\cdot
		\frac{(s_1^{-1}s_2 u; q)_{\infty}}{(s_1s_2u ; q)_{\infty}}
		\\&\hspace{25pt}\times
		\sum_{\mathbf{P}}
		\frac{(s_1s_2/u ; q)_{|\mathbf{P}|}
		(s_1u/s_2 ; q)_{|\mathbf{B}-\mathbf{P}|}}
		{(s_1^2 ; q)_{|\mathbf{B}|}}\ssp
		q^{\sum_{1\le i<j\le n}\left(B_i-P_i\right) P_j}
		\prod_{i=1}^n
		\frac{(q;q)_{B_i}}{(q;q)_{P_i}(q;q)_{B_i-P_i}}
		\\&\hspace{25pt}\times
		\Bigl(\frac{s_1s_2}{u} \Bigr)^{|\mathbf{B}|-|\mathbf{P}|}
		\Bigl(\frac{us_2}{s_1} \Bigr)^{|\mathbf{D}|}
		% s_1^{|\mathbf{B}|-|\mathbf{D}|-|\mathbf{P}|}
		% s_2^{|\mathbf{B}|+|\mathbf{D}|-|\mathbf{P}|}
		% u^{|\mathbf{D}|-|\mathbf{B}|+|\mathbf{P}|}
		\ssp
		q^{\sum_{m\le i<j\le n} D_i (C_j-P_j)}
		\ssp
		\frac{
		(s_1^2 ; q)_{|\mathbf{D}|}}
		{(q;q)_{D_m}}
		\prod_{i=m+1}^n
		\frac{(q;q)_{C_i-P_i+D_i}}{(q;q)_{C_i-P_i}(q;q)_{D_i}},
	\end{split}
\end{equation}
where the sum is over $\mathbf{P}\in \mathbb{Z}_{\ge0}^n$
with $0\le P_i\le \min (B_i,C_i)$ for all $i$.
\end{lemma}
\begin{proof}
Set~$x = s_1 s_2^{-1} u$.
 Since $\mathbf{B}$ stays fixed,
	the second factor
	\begin{equation*}
		\Phi(
		\mathbf{P}, \mathbf{B} ; q^{-\mathsf{L}} / x, q^{-\mathsf{L}}
		)
		=\Phi(
		\mathbf{P}, \mathbf{B} ; s_1s_2 / u, s_1^2
		)
	\end{equation*}
	in the sum in
	\eqref{eq:fully_fused_stochastic_weights}
	is also fixed. Moreover, the
	summation multi-index $\mathbf{P}$
	belongs to a fixed finite set where $P_i\le B_i$ for all $i$.
	Thus, it remains to consider the limit of
	\begin{equation}
		\label{eq:queue_spec_proof}
		\begin{split}
			&(q^{\mathsf{L}})^{|\mathbf{A}|}\ssp
			\Phi(\mathbf{C}-\mathbf{P}, \mathbf{C}+\mathbf{D}-\mathbf{P} ; q^{\mathsf{L}-\mathsf{M}} x, q^{-\mathsf{M}} x )
			\\&\hspace{20pt}=
			(q^{\mathsf{L}})^{|\mathbf{A}|}\ssp
			\frac{(q^{\mathsf{L}-\mathsf{M}} x ; q)_{|\mathbf{C}-\mathbf{P}|}(q^{-\mathsf{L}} ; q)_{|\mathbf{D}|}}
			{(q^{-\mathsf{M}} x ; q)_{|\mathbf{C}-\mathbf{P}+\mathbf{D}|}}\ssp
			(q^{-\mathsf{L}})^{|\mathbf{C}-\mathbf{P}|}
			q^{\sum_{1\le i<j\le n} D_i (C_j-P_j)}
			\prod_{i=1}^n
			\frac{(q;q)_{C_i-P_i+D_i}}{(q;q)_{C_i-P_i}(q;q)_{D_i}}.
			\\&\hspace{20pt}=
			s_1^{-2|\mathbf{D}-\mathbf{B}+\mathbf{P}|}\ssp
			\frac{(s_1^{-1}s_2 u; q)_{|\mathbf{C}-\mathbf{P}|}
			(s_1^2 ; q)_{|\mathbf{D}|}}
			{(s_1s_2u ; q)_{|\mathbf{C}-\mathbf{P}+\mathbf{D}|}}\ssp
			q^{\sum_{1\le i<j\le n} D_i (C_j-P_j)}
			\prod_{i=1}^n
			\frac{(q;q)_{C_i-P_i+D_i}}{(q;q)_{C_i-P_i}(q;q)_{D_i}}.
		\end{split}
	\end{equation}
	We have $C_m-P_m\to +\infty$, so \eqref{eq:queue_spec_proof}
	converges to
	zero unless $D_i=0$ for all $i<m$ (as~$q \in [0, 1)$).
	This leads to the indicator $\mathbf{1}_{D_1=\ldots=D_{m-1}=0}$ in \eqref{eq:queue_spec_fully_fused}.
	Next, if $D_i=0$ for all $i<m$,
	then all other factors in \eqref{eq:queue_spec_proof}
	behave well, and the desired limit of the
	whole vertex weight exists.
	Taking the limit as~$C_m \rightarrow +\infty$,
    we immediately obtain \eqref{eq:queue_spec_fully_fused}.
	This completes the proof.
\end{proof}

\begin{definition}[Queue specialization of the vertex weights]
	\label{def:queue_spec_defn}
	Let $1\le m\le n$ and $u,s_1,s_2\in \mathbb{C}$.
	We denote the limiting vertex weights \eqref{eq:queue_spec_fully_fused} in
	\Cref{lemma:queue_spec_exists} by
	$\WQ_{s_1,s_2,u}^{(-m)}(\mathbf{A},\mathbf{B};\mathbf{C},\mathbf{D})$,
	and call them
	the \emph{queue specialization} of the
	fully fused stochastic colored vertex weights.

	The term ``queue specialization''
	comes from connections with multiline queues described (in two different degenerations) in
	\Cref{sub:mqueues_martin_matching,sub:qBoson_multiline} below.
	The label $(-m)$ will be
	useful when we later place the vertices
	$\WQ_{s_1,s_2,u}^{(-m)}$ on a lattice.
\end{definition}

\begin{remark}
	\label{rmk:queue_spec_fully_fused_abuse_notation}
	In the queue vertex weights
	$\WQ_{s_1,s_2,u}^{(-m)}(\mathbf{A},\mathbf{B};\mathbf{C},\mathbf{D})$,
	we abuse the notation of the tuples $\mathbf{A},\mathbf{C}\in \mathbb{Z}_{\ge0}^n$ by setting $A_m,C_m=+\infty$.
	That is,
	the tuples with infinitely many arrows of color $m$ are
	not elements of $\mathbb{Z}_{\ge0}^n$.
	However,
	for the uniformity of notation, we will still sometimes
	treat $\mathbf{A},\mathbf{C}$ as elements of $\mathbb{Z}_{\ge0}^n$,
	while explicitly stating that $A_m,C_m=+\infty$.
\end{remark}

\begin{remark}
	\label{rmk:queue_weights_independence_of_1_to_m_1}
	From \eqref{eq:queue_spec_fully_fused}
	we see that for fixed $B_1,\ldots,B_{m-1}\ge0 $,
	the weights
	$\WQ_{s_1,s_2,u}^{(-m)}(\mathbf{A},\mathbf{B};\mathbf{C},\mathbf{D})$
	are \emph{independent} of $A_1,\ldots,A_{m-1},C_1,\ldots, C_{m-1}$
	provided that $C_i=A_i+B_i$ for all $i<m$.
	For example, we can set
	$A_i=0$ and $C_i=B_i$ for all $i<m$.
	Note also that since
	$A_m,C_m=+\infty$, we may have $D_m>0$ even if $B_m=0$. The latter
	property is essential for our constructions.
\end{remark}

The next lemma
states the independence of
the queue vertex weights under
$B_m$, too, provided that no lower colors are present:

\begin{lemma}
	\label{lemma:queue_indep_of_Bm}
	Let $B_1=\ldots=B_{m-1}=0$. Then the queue vertex weight $\WQ_{s_1,s_2,u}^{(-m)}(\mathbf{A},\mathbf{B};\mathbf{C},\mathbf{D})$ is independent of $B_m$.
\end{lemma}
\begin{proof}
	By \Cref{rmk:queue_weights_independence_of_1_to_m_1},
	we may set $A_i=C_i=D_i=0$ for all $i<m$.
	Set, for simplicity, $B_m=b$, $B_{m+1}+\ldots+B_n=b' $, $P_m=p$, $P_{m+1}+\ldots+P_n =p'$, and $D_m=d$.
	Then the part of
	\eqref{eq:queue_spec_fully_fused}
	which depends only on $b$ and $p$ has the form
	\begin{equation*}
		\begin{split}
			&\sum_{p=0}^{b}
			\frac{(s_1s_2/u;q)_{p+p'}(s_1u/s_2;q)_{b+b'-p-p'}}
			{(s_1^2;q)_{b+b'}}
			\ssp
			\frac{(q;q)_{b}}{(q;q)_p(q;q)_{b-p}}
			\ssp
			(s_1s_2/u)^{b-p} q^{(b-p)p'}
			\\&=
			\frac{(s_1s_2/u;q)_{p'}
			(s_1u/s_2;q)_{b'-p'}
		}{(s_1^2;q)_{b'}}
			\sum_{p=0}^{b}
			\frac{(q^{p'} s_1s_2/u;q)_{p}
			(q^{b'-p'} s_1u/s_2; q)_{b-p}
			}
			{(q^{b'} s_1^2;q)_{b}}
			\ssp
			\frac{(q;q)_{b}}{(q;q)_p(q;q)_{b-p}}
			\ssp
			(q^{p'} s_1s_2/u )^{b-p}
			\\&=
			\frac{(s_1s_2/u;q)_{p'}
			(s_1u/s_2;q)_{b'-p'}
		}{(s_1^2;q)_{b'}}.
		\end{split}
	\end{equation*}
	In the last equality we used the
	sum-to-one property
	\eqref{eq:capital_Phi_coloreq_sum_to_one}
	for the single-color case $n=1$ (with~$x = q^{b'-p'} s_1u/s_2$, $y = q^{b'} s_1^2$, and~$\mathbf{B} = b$).
	We see that the resulting expression does not depend on $b$,
	as desired.
\end{proof}

\begin{proposition}
	\label{prop:queue_YBE}
	For any $m\in\{1,\ldots,n\}$, the queue vertex weights
	$\WQ_{s_1,s_0,\frac{u_1}{u_0}}^{(-m)}$,
	$\WQ_{s_2,s_0,\frac{u_2}{u_0}}^{(-m)}$
	\eqref{eq:queue_spec_fully_fused},
	and the fused cross vertex weight
	$W_{\frac{s_1u_1}{s_2u_2},\mathsf{L},\mathsf{M}}$
	\eqref{eq:fully_fused_stochastic_weights},
	where $q^{-\mathsf{L}}=s_1^2$ and $q^{-\mathsf{M}}=s_2^2$,
	satisfy the Yang--Baxter equation given in
	\Cref{fig:YBE_fused_queue}. In symbols, for all fixed~$\mathbf{A},\mathbf{I}_1, \mathbf{I}_2,\mathbf{B}, \mathbf{J}_1,\mathbf{J}_2$
	with $A_m,B_m=+\infty$, we have
 	\begin{equation}
	\begin{split}
		&
		\sum_{\mathbf{K}_1,\mathbf{K}_2,\mathbf{K}_3}
\WQ_{s_2,s_0,\frac{u_2}{u_0}}^{(-m)}(\mathbf{A}, \mathbf{I}_2 ; \mathbf{K}_3, \mathbf{K}_2)
    \WQ_{s_1,s_0,\frac{u_1}{u_0}}^{(-m)}(\mathbf{K}_3, \mathbf{I}_1; \mathbf{B}, \mathbf{K}_1)
     W_{\frac{s_1u_1}{s_2u_2},\mathsf{L},\mathsf{M}}
     (\mathbf{K}_2, \mathbf{K}_1; \mathbf{J}_2, \mathbf{J}_1) \\
  		&\hspace{5pt}
		=
     \sum_{\mathbf{K}_1,\mathbf{K}_2,\mathbf{K}_3}
     W_{\frac{s_1u_1}{s_2u_2},\mathsf{L},\mathsf{M}}
     (\mathbf{I}_2, \mathbf{I}_1; \mathbf{K}_2, \mathbf{K}_1)
     \WQ_{s_1,s_0,\frac{u_1}{u_0}}^{(-m)}(\mathbf{A}, \mathbf{K}_1; \mathbf{K}_3, \mathbf{J}_1)
\WQ_{s_2,s_0,\frac{u_2}{u_0}}^{(-m)}(\mathbf{K}_3, \mathbf{K}_2 ; \mathbf{B}, \mathbf{J}_2).
\end{split}
 	\end{equation}
\end{proposition}
\begin{proof}
	This is the queue specialization
	of the Yang--Baxter equation
	\cite[(C.1.2)]{borodin_wheeler2018coloured}
	for the fully fused stochastic weights.
	The queue specialization is taken
	along
	the vertical line carrying the parameters
	$z=s_0u_0$ and $\mathsf{N}$ with $q^{-\mathsf{N}}=s_0^2$.
\end{proof}

\begin{figure}[htpb]
	\centering
	\includegraphics[width=.9\textwidth]{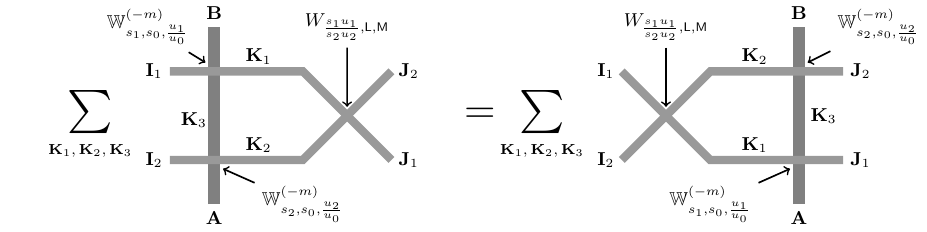}
	\caption{Yang--Baxter equation
		for the queue specialization. The
		sums in both sides
		are taken over all
		$\mathbf{K}_1,\mathbf{K}_2,\mathbf{K}_3\in \mathbb{Z}_{\ge0}^n$,
		and the inputs and outputs
	$\mathbf{I}_1,\mathbf{I}_2,\mathbf{J}_1,\mathbf{J}_2,\mathbf{A},\mathbf{B}\in \mathbb{Z}_{\ge0}^n$ are fixed.}
	\label{fig:YBE_fused_queue}
\end{figure}

Recall that $0\le q<1$.
Let us define the following two subsets of the
parameters $(s_1,s_2,u)$ in the queue vertex weights:
\begin{equation}
	\label{eq:queue_general_nonnegative_conditions}
	\parbox{.88\textwidth}{\begin{enumerate}[$\bullet$]
			\item
			(higher horizontal spin)
			$s_1, s_2 \in [-1,1]$ such that $0 \leq s_1 s_2 < u < \min(\frac{s_1}{s_2}, \frac{s_2}{s_1}, \frac{1}{s_1 s_2})$.
			The number of paths on a horizontal edge is not restricted.

			\item
			(finite horizontal spin)
			$s_1 = q^{-\frac{\mathsf{L}}{2}}$ for some $\mathsf{L}
			\in\mathbb{Z}_{\ge1}$ and $u = q^{\frac{\mathsf{L}}{2}} u'$,
			with purely imaginary
			$u',s_2$ satisfying
			$s_2 u'  \leq  s_2^2 \leq 0$
			and
			$s_2/u'\ge q^{\mathsf{L}}$.
			The number of paths on a horizontal edge is at most $\mathsf{L}$.
		\end{enumerate}}
\end{equation}
These subsets present convenient sufficient nonnegativity conditions:

\begin{proposition}\label{prop:queue_positivity}
	Under \eqref{eq:queue_general_nonnegative_conditions},
	the queue vertex weights
	$\WQ^{(-m)}_{s_1, s_2, u}$
	\eqref{eq:queue_spec_fully_fused}
	are nonnegative, and
	\begin{equation}
		\label{eq:queue_general_sum_to_one}
		\sum_{C_{m+1},C_{m+2},\ldots,C_n,
		D_{m},D_{m+1},\ldots,D_n=0}^{\infty}
		\WQ^{(-m)}_{s_1, s_2, u}(\mathbf{A}, \mathbf{B}; \mathbf{C}, \mathbf{D})=1,
	\end{equation}
	where
	$\mathbf{A},\mathbf{B}$ are
	fixed, $A_m,C_m=+\infty$,
	and $C_i=A_i+B_i$, $D_i=0$ for all $i<m$.
\end{proposition}

\begin{proof}
	Let us first consider the higher horizontal
	spin case of \eqref{eq:queue_general_nonnegative_conditions}.
	One can check that all arguments of the $q$-Pochhammer symbols in
	\eqref{eq:queue_spec_fully_fused}
	are in $[0,1)$. Moreover, the total sign coming from the powers
	$
	s_1^{|\mathbf{B}|-|\mathbf{D}|-|\mathbf{P}|}
	s_2^{|\mathbf{B}|+|\mathbf{D}|-|\mathbf{P}|}
	u^{|\mathbf{D}|-|\mathbf{B}|+|\mathbf{P}|}
	$
	is always nonnegative. Thus, we get the nonnegativity
	of the queue vertex weights.

	To see that they sum to one, we need to take the limit
	of the sum-to-one identity
	\eqref{eq:fully_fused_stochastic_weights_sum_to_one}
	for the fully fused weights
	$W_{x,\mathsf{L},\mathsf{M}}$
	as $A_m,C_m\to+\infty$.
	One can check that under our conditions
	\eqref{eq:queue_general_nonnegative_conditions}, the weight
	$W_{x,\mathsf{L},\mathsf{M}}
	(\mathbf{A},\mathbf{B};\mathbf{C},\mathbf{D})$
	(with $x=s_1u/s_2$,
	$q^{-\mathsf{L}}=s_1^2$,
	$q^{-\mathsf{M}}=s_2^2$)
	decays exponentially fast when $|\mathbf{D}|\to+\infty$
	due to the presence of the power $(us_2/s_1)^{|\mathbf{D}|}$.
	Thus, since our queue specialization
	requires $|\mathbf{D}|$ to stay fixed,
	this decay ensures that
	identity \eqref{eq:fully_fused_stochastic_weights_sum_to_one}
	yields \eqref{eq:queue_general_sum_to_one}.

	Let us now consider the finite horizontal spin
	case of
	\eqref{eq:queue_general_nonnegative_conditions}.
	In this case, $|\mathbf{D}|$ must stay finite
	(cf. \Cref{rmk:finite_spin,rmk:finite_spin_reduction_fully_fused}), so the sum-to-one property \eqref{eq:queue_general_sum_to_one}
	is automatic from~\eqref{eq:fully_fused_stochastic_weights_sum_to_one}.
	To see the nonnegativity, observe that since~$s_1 = q^{-\frac{\mathsf{L}}{2}} \geq 1$, the term $(s_1^2;q)_{|\mathbf{D}|}$ produces the sign $(-1)^{|\mathbf{D}|}$, which is compensated by $(us_2/s_1)^{|\mathbf{D}|}$. The other combined powers
	$
	(s_1s_2/u)^{|\mathbf{B}|-|\mathbf{P}|}
	$
	are always nonnegative.
	Next, we have
	\begin{equation*}
		\frac{(s_1^{-1}s_2 u; q)_{\infty}}{(s_1s_2u ; q)_{\infty}}
		=
		\frac{1}{(s_2u';q)_{\mathsf{L}}}\ge 0.
	\end{equation*}
	In the remaining $q$-Pochhammer symbols, we have
	\begin{equation*}
		s_1s_2/u=q^{-\mathsf{L}}s_2/u'\ge1,\qquad
		s_1u/s_2=u'/s_2\ge1,\qquad
		s_1^2=q^{-\mathsf{L}}\ge1.
	\end{equation*}
	Thus, the
	quantities
	$(s_1s_2/u ; q)_{|\mathbf{P}|}
	(s_1u/s_2 ; q)_{|\mathbf{B}-\mathbf{P}|}
	(s_1^2 ; q)_{|\mathbf{B}|}^{-1}$
	are nonnegative for all $\mathbf{P}$
	provided that $|\mathbf{B}|\le \mathsf{L}$.
	This completes the proof.
\end{proof}

% \begin{remark}
% 	\label{rmk:finite_spin_nonnegativity_ASEP}
% 	In the proof of
% 	\Cref{prop:queue_positivity},
% 	we showed that under
% 	conditions \eqref{eq:queue_general_nonnegative_conditions},
% 	all terms of the sum over $\mathbf{P}$
% 	are separately nonnegative.
% 	In \Cref{sec:ASEP_matrix_products} below,
% 	we consider the special spin-$\frac12$ case.
% 	Then the sum over $\mathbf{P}$ can be readily computed, which will allow us to extend the second range of parameters in \eqref{eq:queue_general_nonnegative_conditions} and still get nonnegative weights.
% \end{remark}

\subsection{Queue vertex model on the cylinder}
\label{sub:mqueue_states}

Let us fix the number $n$ of colors, and another integer $N\ge1$ which is the size of the cylinder.
In this subsection we define a linear operator $\mathfrak{Q}$
whose matrix elements
$\langle \emptyset | \mathfrak{Q} | \mathbf{V} \rangle$
are partition functions coming
from the queue vertex weights on the
cylinder $\{-n, \dots, -1\} \times \mathbb{Z}/N \mathbb{Z}$.
The vertices on the cylinder are indexed by $(-m,j)$, $m=1,\ldots,n $, $j=1,\ldots,N $, see
\Cref{fig:queue_state} for an illustration. No paths enter from the left.
The paths exiting horizontally from the right are encoded by the integer tuples $\mathbf{V}(1),\dots, \mathbf{V}(N) \in \mathbb{Z}_{\ge0}^n$.
Define the vector space for the states on the cylinder,
\begin{equation}
	\label{eq:queue_state_space_higher_spin}
	V^{\otimes N} \coloneqq
	\mathop{\mathrm{Span}}
	\left( \{ \big| \mathbf{V}\big\rangle =\big| (\mathbf{V}(1),\dots, \mathbf{V}(N)) \big\rangle : \mathbf{V}(j) \in \mathbb{Z}_{\geq 0}^n \  \forall j  =1,\dots, N\}  \right),
\end{equation}
and similarly let $\langle \mathbf{V} |$ denote the dual basis in $V^{\otimes N}$.
We also need the subspace
$V^{\otimes N}_{\mathrm{full}}$ of $V^{\otimes N}$ spanned by the vectors
$| \mathbf{V} \rangle$ satisfying
\begin{equation*}
	\sum\nolimits_{j=1}^{N}\mathbf{V}(j)_k > 0
	\quad
	\textnormal{for all $k=1,\ldots,n $}
\end{equation*}
(above,~$\mathbf{V}(j)_k$ denotes the~$k$-th coordinate of~$\mathbf{V}(j)$).
In words, each state
$| \mathbf{V} \rangle\in V^{\otimes N}_{\mathrm{full}}$
must contain at least one arrow of each of the $n$ colors.
\begin{definition}\label{def:queue_state_weights}
Fix complex parameters
\begin{equation}
	\label{eq:queue_state_weights_parameters}
	\mathbf{u} =(u_1,\dots, u_N) \in \mathbb{C}^N,
	\quad
	\mathbf{s}^{(h)} = (s^{(h)}_1,\dots, s^{(h)}_N),
	\quad
	\mathbf{v} = (v_1,\dots,v_n),
	\quad
	\mathbf{s}^{(v)} = (s^{(v)}_1,\dots, s^{(v)}_n),
\end{equation}
The linear operator
$\mathfrak{Q}=\mathfrak{Q}
( \mathbf{u}; \mathbf{s}^{(h)}; \mathbf{v}; \mathbf{s}^{(v)})$
on $V^{\otimes N}$,
called the \emph{queue transfer matrix},
is defined via its matrix elements
$\langle \mathbf{V}' | \mathfrak{Q} | \mathbf{V} \rangle$
as follows.
First, if $| \mathbf{V} \rangle\notin V^{\otimes N}_{\mathrm{full}}$,
we set
$\langle \mathbf{V}' | \mathfrak{Q} | \mathbf{V} \rangle=0$.
Otherwise,
$\langle \mathbf{V}' | \mathfrak{Q} | \mathbf{V} \rangle$
is the partition function of the queue vertex weights on the cylinder
$\{-n,\ldots,-1 \}\times \mathbb{Z}/N\mathbb{Z}$
with the following data:
\begin{enumerate}[$\bullet$]
	\item
		The entering arrow configurations
		$\mathbf{V}'(j)$ along the left horizontal edges $(-n-1,j)\to(-n,j)$,
		$j=1,\ldots,N $.
	\item
		The terminal arrow configurations
		$\mathbf{V}(j)$ along the right horizontal edges $(-1,j)\to(0,j)$,
		$j=1,\ldots,N $.
	\item
		Queue vertex weights
		$\WQ^{(-m)}_{s^{(h)}_j, s^{(v)}_m, u_j / v_{m}}$
		at each vertex $(-m,j)$ in the cylinder,
		where $j=1,\ldots,N $, $m=1,\ldots,n $, and
		$u_j,s_j^{(h)}$ and $v_m,s^{(v)}_m$ are the horizontal and the
		vertical parameters, respectively.
\end{enumerate}
In detail, the partition function
$\langle \mathbf{V}' | \mathfrak{Q} | \mathbf{V} \rangle$
is equal to the sum
\begin{equation*}
	\sum_{
		\mathbf{M}
	}
	\ssp
	\sum_{\mathscr{C}\in \mathcal{P}_{\mathbf{M},\mathbf{V}',\mathbf{M},\mathbf{V}}}
	\ssp
	\prod_{m=1}^{n}\prod_{j=1}^{N}
	\WQ^{(-m)}_{s^{(h)}_j, s^{(v)}_m, u_j / v_{m}}(\mathbf{A}(m,j),\mathbf{B}(m,j);\mathbf{C}(m,j),\mathbf{D}(m,j)),
\end{equation*}
where
$\mathbf{M}=(\mathbf{M}(-n),\dots,\mathbf{M}(-1))$
encodes the paths winding around the cylinder. The sum over~$\mathbf{M}$, by definition, has
\begin{equation}
	\label{eq:queue_state_trace_conditions}
	\mathbf{M}(-m)_m=+\infty
	\quad\textnormal{and}\quad
	\mathbf{M}(-m)_i=0,\ i<m,
	\quad \textnormal{for all}\ m=1,\ldots,n.
\end{equation}
For each $\mathbf{M}$, the sum over $\mathscr{C}$ runs over all
path configurations in the rectangle with the boundary
conditions
${\mathbf{M},\mathbf{V}',\mathbf{M},\mathbf{V}}$
at the bottom, left, top, and right, respectively.
The tuples
$\mathbf{A}(m,j)$, $\mathbf{B}(m,j)$, $\mathbf{C}(m,j)$, and $\mathbf{D}(m,j)$
encode the arrow configurations at each vertex $(-m,j)$ of the rectangle.

See \Cref{fig:queue_state} for an example when~$\mathbf{V}' = (\mathbf{0}, \dots, \mathbf{0})$.
\end{definition}

\begin{figure}[htpb]
	\centering
	\includegraphics[width=.5\textwidth]{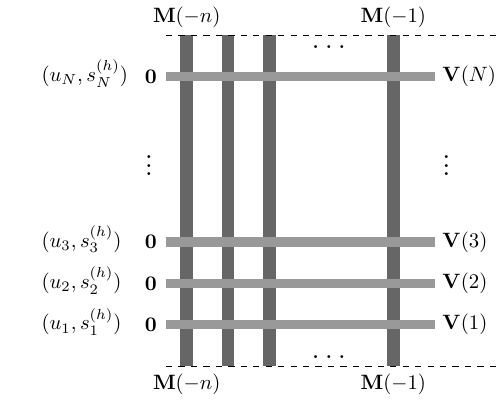}
	\caption{A vertex model on the cylinder whose partition function (indexed by the tuples $\mathbf{V}(1),\ldots,\mathbf{V}(N) $) is equal to $\langle \emptyset | \mathfrak{Q} | \mathbf{V} \rangle$. We identify the top and the bottom boundaries (dashed lines), and
		sum over all possible tuples $\mathbf{M}(-n),\dots,\mathbf{M}(-1)\in \mathbb{Z}^{n}_{\ge0}$
	encoding the paths which wind around the cylinder an arbitrary number of times. The left vector $\langle \emptyset |$ is empty, which corresponds to no paths entering from the left.}
	\label{fig:queue_state}
\end{figure}

Observe that the partition function $\langle \mathbf{V}' | \mathfrak{Q} | \mathbf{V} \rangle$ in \Cref{def:queue_state_weights} involving the sum over $\mathbf{M}(-n),\ldots,\mathbf{M}(-1)$  cannot be interpreted as a probability in a stochastic vertex model because we are summing over input path configurations at vertices, as well as over the output ones. Moreover, this sum may even be divergent. Therefore, we need to make sure that the quantities $\langle \mathbf{V}' | \mathfrak{Q} | \mathbf{V} \rangle$ are well-defined:

\begin{lemma}
	\label{lemma:queue_vertex_well_defined}
	For any
	$| \mathbf{V} \rangle\in V^{\otimes N}_{\mathrm{full}}$,
	the sum over
	$\mathbf{M}(-n),\ldots,\mathbf{M}(-1)$ in
	\Cref{def:queue_state_weights} is convergent for any $|q|<1$.
\end{lemma}
\begin{proof}
	By \Cref{lemma:queue_spec_exists}, paths of colors $i<m$ cannot
	leave the column $(-m)$.
	The fact that the configuration
	$| \mathbf{V} \rangle\in V^{\otimes N}_{\mathrm{full}}$
	contains at least one arrow of each color implies that
	each column $(-m)$ must horizontally emit at least one path of color $m$.
    Recall again that after we fix the entering horizontal arrows
    $\mathbf{B}$ at a lattice site, the
	summation multi-index $\mathbf{P}$ in each vertex weight
	belongs to a fixed finite set where $P_i\le B_i$ for all $i$.
    Furthermore, with~$\mathbf{V}$ fixed, the size of tuples
    $\mathbf{B},\mathbf{D}$ at each vertex is bounded from above. Thus, the factor
    $q^{-D_i P_j}$ in each vertex weight is bounded from above.
	Therefore, the presence of the factors
	$q^{\sum_{m\le i<j\le n} D_i C_j}$ in
	the vertex weights~\eqref{eq:queue_spec_fully_fused}
	implies that
	the weight of the whole column $(-m)$ with fixed
	winding path counts $\mathbf{M}(-m)_i$, $i>m$, contains
	the factor $q^{d\sum_{i=m+1}^n \mathbf{M}(-m)_i}$
	for some $d\ge 1$. This implies that for any fixed $m$,
	the
	sum over the quantities
	$\mathbf{M}(-m)_i$, $i>m$, is finite. This completes the proof.
\end{proof}

\begin{remark}
	\label{rmk:queue_vertex_not_well_defined_for_not_full}
	The condition that
	$| \mathbf{V} \rangle\in V^{\otimes N}_{\mathrm{full}}$
	is essential for
	\Cref{lemma:queue_vertex_well_defined}.
	Indeed, for $n=3$, we have
	\begin{equation*}
		\WQ^{(-2)}_{s_1,s_2,u}
		(\mathbf{A},\mathbf{B};\mathbf{C},\mathbf{D})=
		\frac{(s_2u/s_1;q)_\infty}{(s_1s_2u;q)_{\infty}},
	\end{equation*}
	where $\mathbf{A}=\mathbf{C}=(0,\infty,k)$ and $\mathbf{B}=\mathbf{D}=(0,0,0)$.
	This expression is independent of $k$. For $N=1$, we need to sum it over all $k$, which leads to divergence.
	However, if $\mathbf{D}=(0,d,0)$, the weight
	$\WQ^{(-2)}_{s_1,s_2,u}
	(\mathbf{A},\mathbf{B};\mathbf{C},\mathbf{D})$
	contains the power $q^{kd}$, eliminating this problem of divergence.
\end{remark}

The partition functions
$\langle \mathbf{V}' | \mathfrak{Q} | \mathbf{V} \rangle$
are essentially independent of the entrance
state $\langle\mathbf{V}'|$:
\begin{proposition}
	\label{prop:independence_of_v_prime}
	Let $| \mathbf{V} \rangle\in V^{\otimes N}_{\mathrm{full}}$. If the entering configuration $\mathbf{V}'$ contains at least one path of color strictly less than $n$, then $\langle \mathbf{V}' | \mathfrak{Q} | \mathbf{V} \rangle=0$. Otherwise, if $\mathbf{V}'$ contains only paths of color $n$, then we have $\langle \mathbf{V}' | \mathfrak{Q} | \mathbf{V} \rangle= \langle \emptyset | \mathfrak{Q} | \mathbf{V} \rangle$.
\end{proposition}
\begin{proof}
	\Cref{lemma:queue_spec_exists}
	guarantees that no paths of color strictly less than $n$
	leave column $(-n)$. Since we cannot have paths of color
    not equal to
	$n$ infinitely wind in the column $(-n)$,
	the partition function
	$\langle \mathbf{V}' | \mathfrak{Q} | \mathbf{V} \rangle$
	must vanish if
	$\mathbf{V}'$ contains at least one path of color
    strictly less than $n$.
	This establishes the first claim.
	The second claim immediately follows from
	\Cref{lemma:queue_indep_of_Bm}.
\end{proof}

\begin{remark}[Trace formula for queue partition functions]
	\label{rmk:product_trace_formula}
	The queue vertex model partition function can be interpreted through a product of the following $N$ operators, where $N$ is the size of the ring in the cross-section of the cylinder:
	\begin{equation}
		\label{eq:product_trace_formula}
		\langle \emptyset | \mathfrak{Q} | \mathbf{V} \rangle=
		\mathop{\mathrm{Trace}}\nolimits^\bullet
		\left( \mathscr{X}_{\mathbf{V}(1)}(u_1,s_1^{(h)})\cdots \mathscr{X}_{\mathbf{V}(N)}(u_N,s_N^{(h)})
		\right).
	\end{equation}
	Here each $\mathscr{X}_{\mathbf{V}}(u,s)$, $\mathbf{V}\in \mathbb{Z}_{\ge0}^n$,
	acts in the $n$-fold tensor product
	$V_{-n}\otimes \ldots \otimes V_{-1} $, where $V_{-m}$ has basis
	$|\mathbf{M}(-m)\rangle$,
	$\mathbf{M}(-m)\in \mathbb{Z}_{\ge0}^{n}$,
	$m=1,\ldots,n $.
	The matrix elements of $\mathscr{X}_{\mathbf{V}}(u,s)$ are partition functions
	of the queue weights
	on the lattice $\{1\}\times \{-n,\ldots,-1 \}$.
	See \Cref{fig:product_trace_formula} for an illustration.
	The operation $\mathop{\mathrm{Trace}}\nolimits^\bullet$ in \eqref{eq:product_trace_formula} means that we restrict the summation to the tuples~$\mathbf{M}$ satisfying \eqref{eq:queue_state_trace_conditions}. One can turn $\mathop{\mathrm{Trace}}\nolimits^\bullet$ into a genuine trace by suitably modifying the definition of the spaces $V_{-m}$.
\end{remark}

\begin{figure}[htpb]
	\centering
	\includegraphics[width=.45\textwidth]{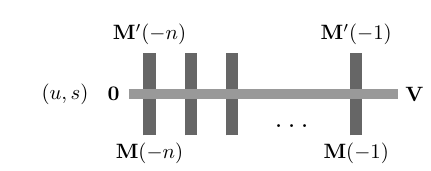}
	\caption{The matrix element
		$\large\langle \mathbf{M}(-n),\ldots,\mathbf{M}(-1)  \large| \ssp \mathscr{X}_{\mathbf{V}}(u,s)\ssp
		\large|  \mathbf{M}'(-n),\ldots,\mathbf{M}'(-1) \large\rangle$
		of one of the operators $\mathscr{X}_{\mathbf{V}}(u,s)$
	entering the trace formula \eqref{eq:product_trace_formula}.
	At the vertex indexed by $-m$, $m=1,\ldots,n $, we have the queue vertex weight
	$\WQ^{(-m)}_{s, s^{(v)}_m, u / v_{m}}$.}
	\label{fig:product_trace_formula}
\end{figure}

\section{Stationarity of the queue vertex model}
\label{sec:main_vertex_statement}

In this section, we establish two general stationarity
properties of the queue vertex model on the cylinder.
Specifically, we construct two Markov operators which, when
applied to the empty state $\langle\emptyset|$ on the
cylinder, commute with the queue transfer matrix
$\mathfrak{Q}$ from \Cref{def:queue_state_weights}. Both of
these commutation relations follow directly from the
Yang--Baxter equation.

In this section we work with \emph{formal} Markov
operators, that is, we do not assume that their matrix elements
are nonnegative. The matrix elements only need to satisfy the corresponding sum-to-one properties.
In the future sections we specify the ranges of parameters making the matrix elements nonnegative.

\subsection{Twisted cylinder Markov operator}
\label{sub:twisted}

Recall that $n$ is the number of colors, and $N$
is the size of the ring in the vertical cross-section of
the cylinder $\{-n,\ldots,-1 \}\times \mathbb{Z}/N\mathbb{Z}$
carrying the queue transfer matrix
$\mathfrak{Q}=\mathfrak{Q}
( \mathbf{u}; \mathbf{s}^{(h)}; \mathbf{v}; \mathbf{s}^{(v)})$ of \Cref{def:queue_state_weights}.
For our first relation, we take the spin-$\frac{1}{2}$ specialization
in the horizontal direction.
That is,
the horizontal spin parameters
in \eqref{eq:queue_state_weights_parameters}
are all equal to $q^{-1/2}$:
\begin{equation*}
	\mathbf{s}^{(h)} = (s^{(h)}_1,\dots, s^{(h)}_N)
	=
	\mathbf{s}^{(h)}_{\frac12}
	\coloneqq (q^{-1/2},\ldots,q^{-1/2} ).
\end{equation*}
We can take each tensor component in the space $V^{\otimes N}$ \eqref{eq:queue_state_space_higher_spin} to be $V=\mathbb{C}^{n+1}$.

We need some extra notation. Let
$R_{z}$ be the~$R$-matrix \eqref{eq:R_matrix_nonfused},
and denote
by $\check{\mathscr{R}}_z$ the operator in $V\otimes V$ with matrix elements
(see the end of \Cref{sub:vertex_weights_text}
for basis vector notations)
\begin{equation}
	\label{eq:swapping_R_check}
	\langle
		i,j
		| \ssp
		\check{\mathscr{R}}_z
		\ssp
		|
		\ell,k
		\rangle
		\coloneqq
		R_{z}(i,j;k,\ell).
\end{equation}
When $\check{\mathscr{R}}_z$ acts on the $k$-th and the $\ell$-th tensor components of $V^{\otimes N}$, we denote it by
$\check{\mathscr{R}}_z^{(k\ell)}$.

\begin{definition}\label{def:twisted_cyl}
	Fix two spectral parameters $u,u_1$.
	Let the \emph{twisted cylinder Markov operator}
	$\mathfrak{T}(u,u_1)$ be
	the linear operator on
	$V^{\otimes N}$
	defined as
	\begin{equation}
		\label{eq:twisted_cylinder_Markov_operator}
		\mathfrak{T}(u, u_1)
		\coloneqq
		\check{\mathscr{R}}_{u_1 u^{-1}}^{(1 2)}
		\ssp
		\check{\mathscr{R}}_{u_1 u^{-1}}^{(2 3)}
		\cdots
		\check{\mathscr{R}}_{u_1 u^{-1}}^{(N1)},
	\end{equation}
	Pictorially,~$\mathfrak{T}(u, u_1)$ is given in
	\Cref{fig:twisted}.
	The product in \eqref{eq:twisted_cylinder_Markov_operator} is interpreted as a product of Markov operators,
	that is, we first apply
	$\check{\mathscr{R}}_{u_1 u^{-1}}^{(1 2)}$ to a fixed configuration of arrows on the cylinder, get a random
	configuration, then apply $\check{\mathscr{R}}_{u_1 u^{-1}}^{(2 3)}$ to the new configuration, and so on.
\end{definition}

\begin{remark}
	\label{rmk:twisted_cylinder_Markov_operator_R_check}
	The swapping of the indices in the operator $\check{\mathscr{R}}_z$ compared to the $R$-matrix $R_z$ (see \eqref{eq:swapping_R_check}) is \emph{purely notational}. We employ it for the following convenience. When \Cref{fig:twisted} is read from left to right, the space $V^{\otimes N}$ encoding configurations on the ring stays the same after every single crossing of the strands. In particular, passing to $\check{\mathscr{R}}_z$ does not affect the random mechanism: the crosses act in \Cref{fig:twisted} in the same way as in the diagram of the Yang--Baxter equation in \Cref{fig:YBE_fused_queue}.
\end{remark}

\begin{figure}[tpb]
	\centering
	\includegraphics[width=.64\textwidth]{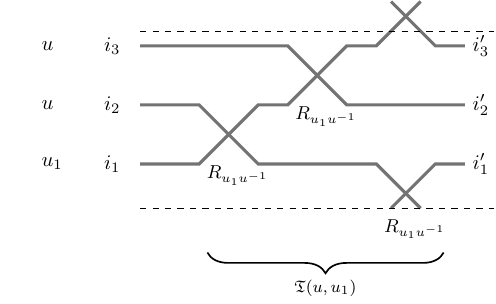}
	\caption{The configuration of vertices whose partition function is the
	matrix element of the twisted cylinder Markov operator $\langle \mathbf{V} | \mathfrak{T}(u, u_1) | \mathbf{V}' \rangle$,
	where $\mathbf{V}=(i_1,i_2,i_3)$, $\mathbf{V}'=(i_1',i_2',i_3')$, and $i_k,i_k'\in \{0,\ldots,n \}$. The size of the ring is $N=3$. After all the $N$ crossings, the spectral parameters attached to the strands on the right are the same as on the left. Note that at each crossing, we use the vertex weights $R_{u_1u^{-1}}$ in the same way as in the Yang--Baxter equation in \Cref{fig:YBE_nonfused_RLL}; see also \Cref{rmk:twisted_cylinder_Markov_operator_R_check} for a discussion of the notation.}
	\label{fig:twisted}
\end{figure}

\begin{figure}[tpb]
	\includegraphics[height=.19\textwidth]{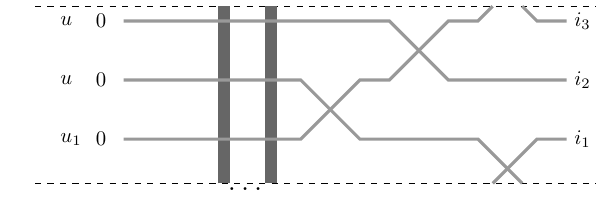}
	\\[15pt]
	\includegraphics[height=.19\textwidth]{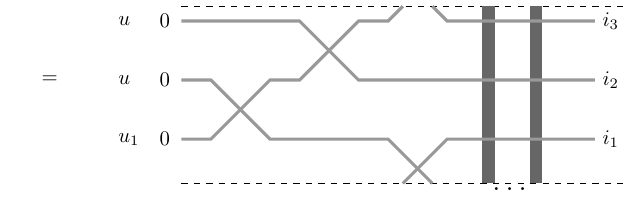}
	\includegraphics[height=.19\textwidth]{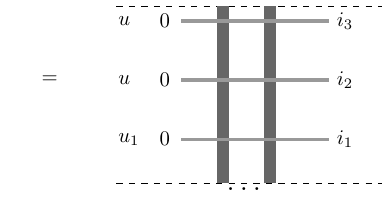}
	\caption{The commutation in the proof of \Cref{thm:twisted_st} for~$N=3$. Here~$i_1,i_2,i_3 \in \{0,1,\dots, n\}$ are arbitrary. }
	\label{fig:twisted_commutation}
\end{figure}

Iterating the sum-to-one property of the $R$-matrix \eqref{eq:R_matrix_nonfused}, we see that for any initial configuration
$\mathbf{V}' = (i_1,\ldots,i_N)$, where $i_k \in \{0,\ldots,n\}$, we have
\begin{equation}
	\label{eq:twisted_cylinder_Markov_operator_sum_to_one}
	\sum_{\mathbf{V}=(i_1',\ldots,i_N')\in \{0,1,\ldots,n \}^N}
	\langle \mathbf{V}' | \mathfrak{T}(u, u_1) | \mathbf{V} \rangle=1.
\end{equation}
If the matrix elements of $\mathfrak{T}(u, u_1)$ are nonnegative, then it is a Markov operator. However, considering $\mathfrak{T}(u, u_1)$ as a formal Markov operator with the sum-to-one property suffices for this section.

\medskip
We have the following stationarity of the queue vertex model under the action of the twisted cylinder Markov operator:

\begin{theorem}\label{thm:twisted_st}
	Let the parameters
	$u,u_1$ of the twisted cylinder Markov operator~$\mathfrak{T}$ \eqref{eq:twisted_cylinder_Markov_operator},
	as well as the parameters
	$\mathbf{v} = (v_1,\dots, v_n)$
	and $\mathbf{s}^{(v)} = (s^{(v)}_1,\dots, s^{(v)}_n)$
	in the queue transfer matrix~$\mathfrak{Q}$ of \Cref{def:queue_state_weights}
	be arbitrary.
	Denote
	$\mathbf{u} \coloneqq (u_1,u,\dots, u)$.
	Then we have
	\begin{equation}\label{eq:twisted_st}
		\big\langle \emptyset \big|
		\ssp
		\mathfrak{Q}
		(\mathbf{u}; \mathbf{s}_{\frac{1}{2}}^{(h)}; \mathbf{v}; \mathbf{s}^{(v)})
		\ssp
		\mathfrak{T}(u, u_1) = \big\langle \emptyset \big| \ssp \mathfrak{Q}
		(\mathbf{u}; \mathbf{s}_{\frac{1}{2}}^{(h)}; \mathbf{v}; \mathbf{s}^{(v)})
		.
	\end{equation}
\end{theorem}
\begin{proof}
	We apply the Yang--Baxter equation specialized to the queue vertex weights; see
	\Cref{prop:queue_YBE}. In this equation, the cross vertex weights have the form
	$
	W_{\frac{u}{u_1},1,1}
	$,
	which are encoded by the operator $\check{\mathscr{R}}_{u_1u^{-1}}$; see
	\eqref{eq:fully_fused_reduction_to_L}.
	We illustrate the argument diagrammatically for~$N = 3$ in \Cref{fig:twisted_commutation}.

	Applying the Yang--Baxter equation
	successively at columns~$-n, -n+1, \dots,-2, -1$
	of the queue vertex model, we get the intertwining
	relation
	\begin{equation}\label{eq:comm1}
		\mathfrak{Q}((u_1, u, \dots,
		u);\mathbf{s}_{\frac{1}{2}}^{(h)}; \mathbf{v}; \mathbf{s}^{(v)})\ssp
		\check{\mathscr{R}}_{u_{1} u^{-1}}^{(1 2)} = \check{\mathscr{R}}_{u_1 u^{-1}}^{(1
		2)}\ssp\mathfrak{Q}((u, u_1, \dots, u);
		\mathbf{s}_{\frac{1}{2}}^{(h)}; \mathbf{v} ; \mathbf{s}^{(v)} )
	\end{equation}
	Continuing inductively with
	$\check{\mathscr{R}}_{u_{1} u^{-1}}^{(2 3)}, \ldots, \check{\mathscr{R}}_{u_{1} u^{-1}}^{(N-1,N)}$
	turns the horizontal spectral parameter sequence in $\mathfrak{Q}$ into $(u, \dots, u, u_1)$. Finally, after applying the last operator $\check{\mathscr{R}}_{u_{1} u^{-1}}^{(N,1)}$, we get back to the original sequence. Here we used the periodicity in the vertical direction of the vertex model defining~$\mathfrak{Q}$, which is crucial to complete the commutation.
	Therefore,
	\begin{equation*}
		\mathfrak{Q}
		(\mathbf{u}; \mathbf{s}_{\frac{1}{2}}^{(h)}; \mathbf{v}; \mathbf{s}^{(v)})
		\ssp
		\mathfrak{T}(u, u_1)
		=
		\mathfrak{T}(u, u_1)
		\ssp
		\mathfrak{Q}
		(\mathbf{u}; \mathbf{s}_{\frac{1}{2}}^{(h)}; \mathbf{v}; \mathbf{s}^{(v)}).
	\end{equation*}
 	It remains to observe that~$ \langle \emptyset | \check{\mathscr{R}}_{u_1 u^{-1}}^{(i,j)} =
	\langle \emptyset | $ for all~$i,j$. This completes the proof.
\end{proof}

% \begin{remark}
% \textcolor{red}{Possibility for remark comparing this construction to the construction of leading eigenvector (=ground state) of transfer matrix via algebraic Bethe Ansatz.}
% \end{remark}

\subsection{Straight cylinder Markov operator}
\label{sub:straight}

Here we define another Markov operator on the cylinder of size $N+1$ which commutes with the queue transfer matrix on this cylinder
when applied to the empty configuration $\emptyset$ at the left boundary of the cylinder.
It acts in the space $\mathbb{C}^{n+1} \otimes V^{\otimes N}$,
where the first factor~$\mathbb{C}^{n+1}$ is an auxiliary space which we identify with the index~$0$ in the superscripts.
The tensor components of the space $V^{\otimes N}$ are indexed by~$j=1,2,\dots, N$, and have basis vectors $|\mathbf{V}(j)\rangle$, where $\mathbf{V}(j) \in \mathbb{Z}_{\geq 0}^n$
encodes the arrow configuration on site $j$. That is, in contrast with \Cref{sub:twisted},
we return to the higher spin setting.
Later in \Cref{sec:qBoson,sec:qPush} we will take limits in which
the marginal process corresponding to the factor~$V^{\otimes N}$ will lead to the colored $q$-Boson or the colored $q$-PushTASEP
on the ring of size $N$.

Recall the Markov operator $\mathscr{L}_{s,z}$
acting in $\mathbb{C}^{n+1}\otimes V$; see
\eqref{eq:L_mat}.
\begin{definition}\label{def:straight_cyl}
	Fix spectral parameters~$(x,\mathbf{u})=(x,u_1,u_2,\dots, u_N)\in \mathbb{C}^{N+1}$
	and horizontal spin parameters $\mathbf{s}^{(h)}=(s_1^{(h)},\dots, s_N^{(h)})\in \mathbb{C}^N$.
	The
	\emph{straight cylinder Markov operator}
	denoted by $\mathfrak{S}=\mathfrak{S}(x,\mathbf{u};\mathbf{s}^{(h)})$
	acts in $\mathbb{C}^{n+1}\otimes V^{\otimes N}$
	as follows:
	\begin{equation}
		\label{eq:straight_cylinder_Markov_operator}
		\mathfrak{S}(x,\mathbf{u};\mathbf{s}^{(h)}) \coloneqq
		\mathscr{L}^{(0N)}_{s^{(h)}_N,xu_{N}^{-1}}\ssp
		\mathscr{L}^{(0,N-1)}_{s^{(h)}_{N-1},xu_{N-1}^{-1}}
		\ldots
		\mathscr{L}^{(01)}_{s^{(h)}_1,xu_{1}^{-1}}.
	\end{equation}
	Here by
	$\mathscr{L}^{(0,j)}_{s,z}$
	we denote the operator
	$\mathscr{L}_{s,z}$
	acting on $\mathbb{C}^{n+1}\otimes V^{\otimes N}$
	in the auxiliary space $\mathbb{C}^{n+1}$ and the $j$-th tensor component of
	$V^{\otimes N}$.
	The operator $\mathfrak{S}$ is illustrated in
	\Cref{fig:straight}.
\end{definition}

\begin{remark}
	\label{rmk:labels_of_axes}
	In \Cref{def:straight_cyl} and throughout this subsection,
	we index the tensor factors of the space
	$\mathbb{C}^{n+1} \otimes V^{\otimes N}$
	by $0,1,\ldots,N$, but put the horizontal strand corresponding to the $0$-th factor
	on top  of the cylinder in \Cref{fig:straight,fig:straight_commuted}.
	This way of indexing is consistent with the action of the operators
	$\mathscr{L}_{s,z}$
	in $\mathbb{C}^{n+1}\otimes V$ in
	\eqref{eq:straight_cylinder_Markov_operator}, while the
	location of the $0$-th strand on top just below the cylinder's cut is convenient for illustrations.
\end{remark}

\begin{figure}
	\centering
	\includegraphics[width=.5\textwidth]{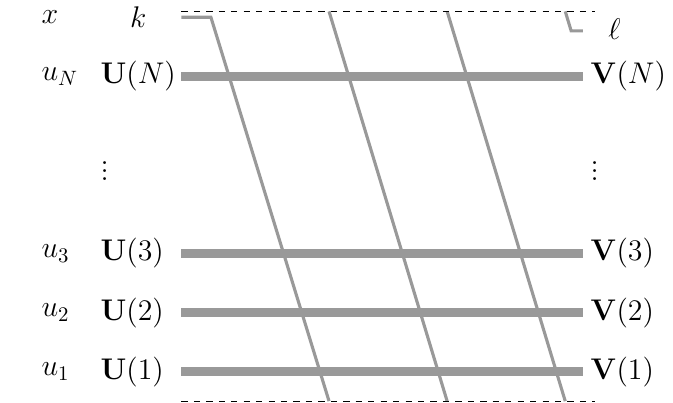}
	\caption{The straight cylinder Markov operator~$\mathfrak{S}=\mathfrak{S}(x,\mathbf{u};\mathbf{s}^{(h)})$,
	applied three times. The line with spectral parameter~$x$ has spin parameter~$q^{-\frac{1}{2}}$, so edges along this line can only be occupied by at most one path
	(hence $k,\ell\in \{0,1,\ldots, n \}$ ).
	The partition function of the displayed configuration is the matrix element
	$\big\langle k, \mathbf{U} \big|\ssp \mathfrak{S}^3 \ssp \big|\ell, \mathbf{V} \big\rangle$.}
	\label{fig:straight}
\end{figure}

\begin{figure}[htb]
	\centering
	\includegraphics[scale=.75]{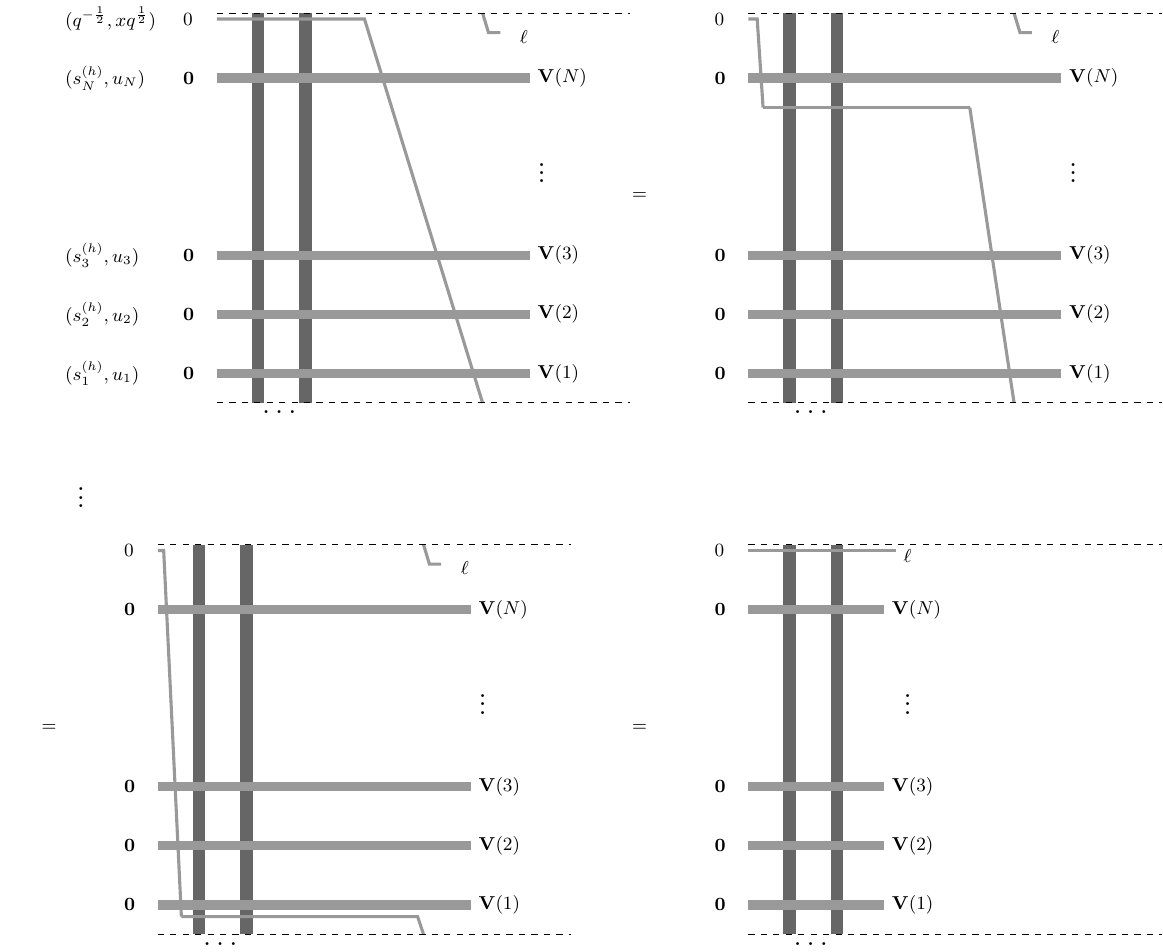}
	\caption{An illustration of the proof of \Cref{thm:straight_st}.
	Each move is justified by the Yang--Baxter equation (\Cref{prop:queue_YBE}).
	As usual, the equality of pictures means the equality of the corresponding partition functions with the fixed boundary conditions.}
	\label{fig:straight_commuted}
\end{figure}

The operator $\mathfrak{S}(x,\mathbf{u};\mathbf{s}^{(h)})$ satisfies the sum-to-one property similarly to \eqref{eq:twisted_cylinder_Markov_operator_sum_to_one}, and thus is a formal
Markov operator.
We have the following stationarity of the queue vertex model under $\mathfrak{S}$:
\begin{theorem}\label{thm:straight_st}
	Let us take the parameters
	$\mathbf{u},\mathbf{s}^{(h)},\mathbf{v},\mathbf{s}^{(v)}$ as in
	\eqref{eq:queue_state_weights_parameters}, and consider the
	queue vertex transfer matrix on the ring of size $N+1$ with the parameters
	\begin{equation*}
		\mathfrak{Q}=\mathfrak{Q}\left( (xq^{\frac12},\mathbf{u});(q^{-\frac{1}{2}},\mathbf{s}^{(h)})
		; \mathbf{v};\mathbf{s}^{(v)}\right).
	\end{equation*}
	Then we have
	\begin{equation*}
		\big\langle \emptyset \big|\ssp \mathfrak{Q}\ssp \mathfrak{S}
		=
		\big\langle \emptyset \big|\ssp\mathfrak{Q},
	\end{equation*}
	where $\mathfrak{S}=\mathfrak{S}(x,\mathbf{u};\mathbf{s}^{(h)})$
	has parameters compatible with those in $\mathfrak{Q}$.
\end{theorem}
\begin{proof}
	We employ the Yang--Baxter equation of
	\Cref{prop:queue_YBE}. Let us match our parameters to this equation.
	The queue vertex model at rows $0$ and $j$ has vertex weights
	\begin{equation*}
		\WQ^{(-m)}_{q^{-1/2},s_m^{(v)},\frac{xq^{1/2}}{v_m}}
		,\qquad
		\WQ^{(-m)}_{s_j^{(h)},s_m^{(v)},\frac{u_j}{v_m}},
	\end{equation*}
	respectively. The
	operator $\mathfrak{S}$
	has the vertex weight
	$L_{s_j^{(h)},xu_j^{-1}}$,
	which, by \eqref{eq:fully_fused_reduction_to_L},
	is the same as the fused weight
	$W_{x u_j^{-1}/s_j^{(h)},1,\mathsf{M}}$, where $q^{-\mathsf{M}/2}=s_j^{(h)}$.
	We see that these three weights indeed satisfy the Yang--Baxter
	equation illustrated in \Cref{fig:YBE_fused_queue}.

	Applying this
	Yang--Baxter equation
	with parameters $(v_1,s_1^{(v)}), \ldots,(v_n,s_n^{(v)}) $
	(that is,
	at the columns $-1,\ldots,-n$ of the queue vertex model),
	we obtain the following relation
	between
	operators acting
	in the space~$\mathbb{C}^{n+1} \otimes V^{\otimes N}$,
	applied
	to the empty configuration $\emptyset$ at the left boundary of the cylinder:
	\begin{equation}
		\label{eq:straight_stat_proof_1}
		\begin{split}
			&\big\langle \emptyset \big|\ssp
			\mathfrak{Q}\left( (xq^{\frac12},u_1,\ldots,u_N );(q^{-\frac{1}{2}},s_1^{(h)},\ldots,s_N^{(h)} )
			; \mathbf{v};\mathbf{s}^{(v)}\right)
			\mathscr{L}^{(0N)}_{s_N^{(h)},xu_{N}^{-1}}
			\\&\hspace{10pt}
			=
			\big \langle \emptyset \big|\ssp
			P^{(0,N)}\ssp
			\mathfrak{Q}\left( (u_N,u_1,\ldots,u_{N-1},xq^{\frac12} );(s_N^{(h)},s_1^{(h)},\ldots,s_{N-1}^{(h)},
				q^{-\frac{1}{2}})
			; \mathbf{v};\mathbf{s}^{(v)}\right)
			P^{(0,N)}
			.
		\end{split}
	\end{equation}
	Here the permutation operator $P^{(0,N)}$
	swaps the $0$-th and the $N$-th
	tensor components in the space
	$\mathbb{C}^{n+1}\otimes V\otimes \ldots\otimes V$,
	and is required
	since the operator $\mathfrak{Q}$ in the right-hand side
	acts in $V\otimes \ldots\otimes V\otimes \mathbb{C}^{n+1} $.
	Identity \eqref{eq:straight_stat_proof_1} represents
	the first step
	of the transformations illustrated in
	\Cref{fig:straight_commuted}.
	In the next step,
	the action of
	$\mathscr{L}^{(0,N-1)}_{s^{(h)}_{N-1},xu_{N-1}^{-1}}$
	results in the following identity:
	\begin{equation*}
		\begin{split}
			&
			\big \langle \emptyset \big|\ssp
			P^{(0,N)}\ssp
			\mathfrak{Q}\left( (u_N,u_1,\ldots,u_{N-1},xq^{\frac12} );(s_N^{(h)},s_1^{(h)},\ldots,s_{N-1}^{(h)},
				q^{-\frac{1}{2}})
			; \mathbf{v};\mathbf{s}^{(v)}\right)
			P^{(0,N)}
			\mathscr{L}^{(0,N-1)}_{s^{(h)}_{N-1},xu_{N-1}^{-1}}
			\\&\hspace{5pt}
			=
			\big \langle \emptyset \big|\ssp
			P^{(N,N-1)}P^{(0,N)}
			\mathfrak{Q}\left( (u_N,u_1,\ldots,u_{N-2},xq^{\frac12},u_{N-1} );(s_N^{(h)},s_1^{(h)},\ldots,s_{N-2}^{(h)},
				q^{-\frac{1}{2}},s_{N-1}^{(h)})
			; \mathbf{v};\mathbf{s}^{(v)}\right)
			\\
			&\hspace{.79\textwidth}
			\times
			P^{(0,N)}P^{(N,N-1)}
			.
		\end{split}
	\end{equation*}
	Iterating the action of the other operators $\mathscr{L}$,
	after $N$
	total steps the horizontal parameter
	sequences $\mathbf{u}$ and $\mathbf{s}^{(h)}$
	return back to their original states
	$(xq^{\frac12},\mathbf{u});(q^{-\frac{1}{2}},\mathbf{s}^{(h)})$.
    Here we employed the periodicity of the vertex model defining~$\mathfrak{Q}$
	to complete the commutation.
	This establishes the desired stationarity relation.
\end{proof}

\section{Multi-species ASEP from twisted cylinder}
\label{sec:ASEP_matrix_products}

In this section, we take a continuous-time limit of the twisted cylinder Markov operator
$\mathfrak{T}$ defined in \Cref{sub:twisted},
and recover the known descriptions of the stationary distribution
of the multi-species TASEP and ASEP on the ring
from \cite{FerrariMartin2005}, \cite{Prolhac_2009}, and \cite{martin2020stationary}.

\subsection{Multi-species ASEP on the ring}
\label{sub:mASEP}

Recall that $n$ is the number of particle species (also called ``types'' or ``colors''), and
$N$ is the size of the ring $\mathbb{Z}/N\mathbb{Z}$.
The state space of the multi-species ASEP
consists of particle configurations
$\eta=(\eta_1,\ldots,\eta_N )$
on the ring, where $\eta_k\in \left\{ 0,1,\ldots,n  \right\}$ encodes the type of the particle at
site $k$.
The type $0$ corresponds to the empty site.

For a configuration $\eta$ and each pair of neighboring sites
$(k,k+1)$ (including $(N,1)$ for $k=N$),
denote by $\eta^{k,k+1}$ the configuration
$(\eta_1,\ldots,\eta_{k+1},\eta_k ,\ldots,\eta_N )$.
That is, $\eta^{k,k+1}$ is obtained from $\eta$ by swapping the types
at sites $k$ and $k+1$.

\begin{definition}
	\label{def:mASEP}
	The \emph{multi-species ASEP} (\emph{mASEP}) is a continuous
	time Markov chain on the space of particle configurations on the ring, with the
	following transition rates:
	\begin{equation}
		\label{eq:ASEP_rates}
		\text{Rate}(\eta \rightarrow \eta^{k,k+1}) =
		\begin{cases}
			q, &  \eta_k > \eta_{k+1}; \\
			1, &  \eta_k < \eta_{k+1},
		\end{cases}
	\end{equation}
	where $k$ runs over $1,2,\ldots,N$, and in \eqref{eq:ASEP_rates}
	we assume that $\eta_k\ne \eta_{k+1}$.
	The multi-species ASEP depends on a single parameter $q\in[0,1)$.
\end{definition}

\begin{remark}
	\label{rmk:types_vs_colors}
	We use the ordering of colors in which color $n$
	has the highest priority (to move from site $j+1$ to site $j$
	on the ring), and color $1$ has the lowest priority.
	Often in the literature on multi-type interacting
	particle systems, e.g., in
	\cite{martin2020stationary}, a
	reverse convention is used, in which type $1$ has the
	highest priority. In \Cref{sub:mqueues_martin_matching} below, when this distinction
	becomes relevant, we recast all the necessary definitions
	from the existing literature using our color ordering conventions.
\end{remark}

Observe that mASEP preserves the number of particles of each type. We denote these type counts by
\begin{equation}
	\label{eq:mASEP_type_counts}
	N_m\coloneqq \sum\nolimits_{j=1}^N
	\mathbf{1}_{\eta_j = m},
	\qquad
	m=1,\ldots,n.
\end{equation}
We have $N_1+\ldots+N_n\le N $. In
addition, throughout this section, we assume that $N_m\ge1$
for all $m=1,\ldots,n $. This assumption is very natural
since if there are no particles of a given type, then the
evolution of the $n$-type mASEP is the same as that of a
$(n-1)$-type mASEP, where the missing type is removed
entirely. Note also that at the level of queue vertex
models, a violation of the assumption that  $N_m\ge1$ for
all $m$ leads to convergence issues when summing over
windings around the cylinder; see
\Cref{lemma:queue_vertex_well_defined} and
\Cref{rmk:queue_vertex_not_well_defined_for_not_full}.

\begin{definition}[mASEP stationary distribution]
	\label{def:mASEP_stationary_distribution}
	When restricted to a \emph{sector} (namely, the subset of the state space) with fixed type counts $(N_1,\ldots,N_n )$, mASEP becomes an irreducible continuous-time Markov chain on a finite state space. Therefore, it admits a \emph{unique stationary distribution} in each sector.
	We denote this distribution by
	$\mathop{\mathrm{Prob}}^{\mathrm{mASEP}}\nolimits_{N_1,\ldots,N_n }(\eta)$.
\end{definition}

It is natural to encode the states $\eta$ as basis vectors
in $(\mathbb{C}^{n+1})^{\otimes N}$:
\begin{equation*}
	|\eta\rangle =
	|\eta_1,\ldots,\eta_N \rangle =
	|\eta_1 \rangle\otimes \ldots \otimes |\eta_N \rangle.
\end{equation*}

Since all jumps under mASEP are nearest neighbor,
the infinitesimal generator of mASEP can be
written as a sum of local rate matrices as follows.
Consider two possible configurations of particles~$i i'$,~$j
j'$ at adjacent lattice sites, say,~$k$ and $k+1$. Define an
operator~$\mathscr{M}_{loc}$ in~$(\mathbb{C}^{n+1})^{\otimes 2}$
such that
\begin{multline}
	\label{eq:Mloc}
	\langle i,i'
	| \mathscr{M}_{loc} | j,j' \rangle =  \big( \mathscr{M}_{loc} \big)_{i i', j j'}  \coloneqq  \mathbf{1}_{(i, i') = (j', j)}  \left( \text{jump rate $i i' \rightarrow j j'$} \right)
  \\
-  \mathbf{1}_{(i, i') = (j, j')}  \left( \text{jump rate $i i' \rightarrow j' j$} \right).
\end{multline}
The matrix element \eqref{eq:Mloc}
is nonzero if and only if $(i,i')=(j,j')$ or $(i,i')=(j',j)$.
The infinitesimal Markov generator of mASEP then has the form
\begin{equation}
	\label{eq:mASEP_generator}
	\mathfrak{M}_{\mathrm{mASEP}}=
	\sum_{l=1}^{N}
	\big( \mathscr{M}_{loc}\big)^{l,l+1},
\end{equation}
where	$\big( \mathscr{M}_{loc}\big)^{l,l+1}$ acts
as~$\mathscr{M}_{loc}$ on tensor factors of sites~$l, l+1$, and
as the identity on all other factors.
Denote by $\{\mathfrak{P}_{\mathrm{mASEP}}(t)\}_{t\in \mathbb{R}_{\ge0}}$, the continuous-time Markov semigroup generated by \eqref{eq:mASEP_generator}. The passage from the infinitesimal generator to this semigroup is straightforward, as the process lives on a finite state space.
See the left side of \Cref{fig:ASEP_dynamics} for an illustration of the process.

\medskip

The next statement identifies mASEP as a Poisson-type
limit of the twisted cylinder Markov operators.
\begin{proposition}
	\label{prop:mASEP_from_R}
	For any $u\in \mathbb{R}$, we have
	the convergence of Markov operators in $(\mathbb{C}^{n+1})^{\otimes N}$:
	\begin{equation}
		\label{eq:limit_twisted_to_mASEP}
		\lim_{\epsilon\to0}
		\mathfrak{T}(u,\ssp u(1-\epsilon))^{\lfloor (1-q) t/\epsilon \rfloor }
		=
		\mathfrak{P}_{\mathrm{mASEP}}(t),
		\qquad t\in \mathbb{R}_{\ge0}.
	\end{equation}
	Here
	$\mathfrak{T}$ is the twisted cylinder Markov
	operator \eqref{eq:twisted_cylinder_Markov_operator}.
	See the right side of \Cref{fig:ASEP_dynamics} for an illustration
	of the limiting jump rates.
\end{proposition}
Recall that $q\in[0,1)$.
One readily sees that for any $u\in \mathbb{R}$ and
$\epsilon\in(0,1)$, both sides of the limiting relation
\eqref{eq:limit_twisted_to_mASEP}
are Markov operators with nonnegative
matrix elements.
\begin{proof}[Proof outline of \Cref{prop:mASEP_from_R}]
	This is a standard limit of
	the stochastic six-vertex model leading to the ASEP,
	see
	\cite{BCG6V},
	\cite{Aggarwal2017convergence}
	and also \cite[Section~12.3]{borodin_wheeler2018coloured} for its colored version.
	In short, in the regime $\epsilon\to0$, all paths want to follow a ``staircase'' motion, with occasional deviations that occur in continuous time according to independent exponential clocks. Because of how the cross vertices are organized to form the twisted cylinder Markov operator (\Cref{fig:twisted}), the staircase motion corresponds to particles staying in place.
	For convenience, let us reproduce the main computation.

	First, consider the limit of the local operators
	$\check{\mathscr{R}}_{z}$ \eqref{eq:swapping_R_check},
	where $z=u_1u^{-1}=1-\epsilon$ because $u_1=u(1-\epsilon)$.
	We have for the matrix elements \eqref{eq:R_matrix_nonfused}:
	\begin{equation*}
		\frac{1-z}{1-qz}=\frac{\epsilon}{1-q}+O(\epsilon^2),
		\qquad
		1-\frac{1-z}{1-qz}=1- \frac{\epsilon}{1-q} + O(\epsilon^2),
	\end{equation*}
	and similarly for $q\frac{1-z}{1-qz}$ and
	$1-q\frac{1-z}{1-qz}$. Therefore, the local
	infinitesimal generator \eqref{eq:Mloc} of mASEP has the following form:
	\begin{equation}
		\label{eq:Mloc_as_derivative}
		\mathscr{M}_{loc}
		=
		(1-q)
		\left.\frac{\partial}{\partial\epsilon}\right|_{\epsilon=0}
		\check{\mathscr{R}}_{1-\epsilon}.
	\end{equation}
	See \Cref{fig:ASEP_dynamics}, right, for an illustration of
	how we interpret the operators
	$\check{\mathscr{R}}_{1-\epsilon}$ in terms of the hopping rates.

	In the $\epsilon\to0$ limit, the product of the operators $\check{\mathscr{R}}_{1-\epsilon}$ over all pairs of neighboring lattice sites (which is equal to the twisted cylinder operator \eqref{eq:twisted_cylinder_Markov_operator}) behaves as~$Id + \frac{1}{1-q} \epsilon \; \mathfrak{M}_{\mathrm{mASEP}}$, where~$Id$ is the identity matrix, and $\mathfrak{M}_{\mathrm{mASEP}}$ is defined in \eqref{eq:mASEP_generator}. This leads to the desired statement about the convergence to the mASEP Markov semigroup.
\end{proof}

\begin{figure}[htb]
    \centering
	\includegraphics[width=.8\textwidth]{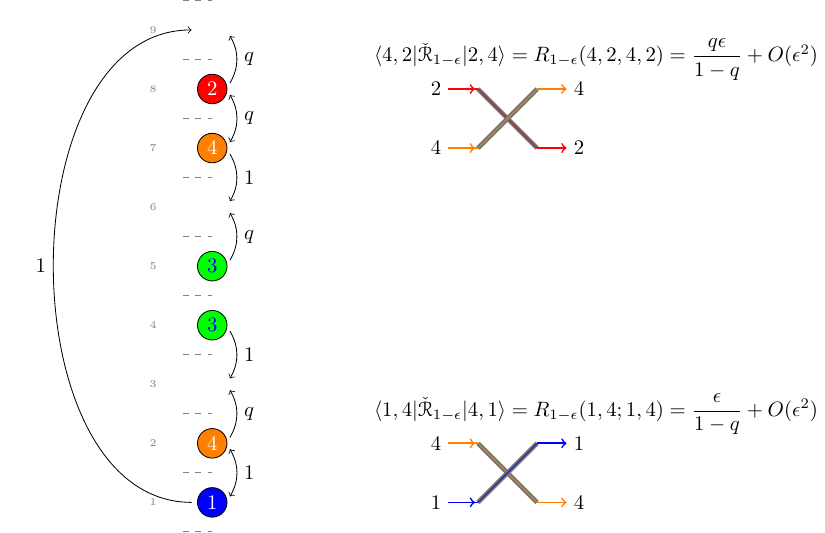}
	\caption{Left: An illustration of a state in mASEP on a ring of size~$N = 9$, with all possible jump rates indicated.
		Right: The interpretation of the mASEP hopping rates as limits of the
		operators $\check{\mathscr{R}}_{1-\epsilon}$.
		Throughout the figure, the colors (types) of the particles are indicated by the numbers to assist the printed version.}
    \label{fig:ASEP_dynamics}
\end{figure}

\subsection{Vertex model for the mASEP stationary distribution}
\label{sub:vertex_models_mASEP}

Let us now apply \Cref{thm:twisted_st}
to mASEP. Recall that by
$\langle \emptyset |
\ssp
\mathfrak{Q}
\ssp
|
\mathbf{V}
\rangle
$
we denote the partition function of the queue
vertex model on the cylinder introduced in
\Cref{def:queue_state_weights}.
For mASEP, the queue transfer matrix
$\mathfrak{Q} =
\mathfrak{Q}(\mathbf{u};\mathbf{s}^{(h)}_{\frac12}
;\mathbf{v};\mathbf{s}^{(v)})$
has the parameters
\begin{equation*}
	\mathbf{u}=(u,\ldots,u ),
	\qquad
	\mathbf{s}^{(h)}_{\frac12}=(q^{-1/2},\ldots,q^{-1/2} ),
\end{equation*}
where $\mathbf{v},\mathbf{s}^{(v)}\in \mathbb{C}^n$ are arbitrary.
The $N$-tuple
$
\mathbf{V}=\mathbf{V}_\eta\coloneqq
(\mathbf{e}_{\eta_1}
	\mathbf{1}_{\eta_1 \ge 1},\dots,
	\mathbf{e}_{\eta_N}
\mathbf{1}_{\eta_N \ge 1})
$
encodes the same information as
the mASEP state $\eta=(\eta_1,\ldots,\eta_N )$. Indeed, this is because each of the
subsets $\mathbf{V}_\eta(j)\subset \mathbb{Z}_{\ge0}^n$, $j=1,\ldots,N$, must have at
most one element thanks to the spin-$\frac12$ reduction coming from
$\mathbf{s}^{(h)}_{\frac12}$.

We split this subsection into two parts. First, we show in \Cref{prop:queue_mASEP_stationary} that the normalized partition functions of the queue vertex model produce the mASEP stationary distribution.
Then we present in \Cref{prop:positivity_parameter_dependence} a slightly modified vertex model for which all partition functions on the cylinder with right boundary $\eta$ are positive without normalization.

\begin{proposition}
	\label{prop:queue_mASEP_stationary}
	With the above notation and for generic complex parameters $u,\mathbf{v},\mathbf{s}^{(v)}$,
	the queue vertex model partition function with boundary $\eta$
	and vertex weights
	$\WQ_{q^{-1/2},s_m^{(v)},u/v_m}^{(-m)}$
	is proportional to the stationary probability for the multi-species ASEP on the ring:
	\begin{equation}
		\label{eq:queue_mASEP_stationary_proport}
		\mathop{\mathrm{Prob}}^{\mathrm{mASEP}}\nolimits_{N_1,\ldots,N_n }(\eta)
		=
		\frac{1}{Z^{\mathrm{mASEP}}_{N_1,\ldots,N_n }(u ;\mathbf{v};\mathbf{s}^{(v)})}\ssp
		\langle \emptyset | \ssp \mathfrak{Q}(\mathbf{u};\mathbf{s}^{(h)}_{\frac12} ;\mathbf{v};\mathbf{s}^{(v)}) \ssp | \mathbf{V}_\eta \rangle.
	\end{equation}
	The normalizing constant
	$Z^{\mathrm{mASEP}}_{N_1,\ldots,N_n }(u ;\mathbf{v};\mathbf{s}^{(v)})$
	may depend on the type counts $(N_1,\ldots,N_n)$ \eqref{eq:mASEP_type_counts}, but not on the state $\eta$ within the sector determined by $(N_1,\ldots,N_n)$.
\end{proposition}
In \Cref{prop:queue_mASEP_stationary}, by ``generic'' we mean that the parameters must ensure that $\langle \emptyset | \ssp \mathfrak{Q} \ssp | \mathbf{V}_\eta \rangle $ is finite and nonzero for all $\eta$ with type counts $N_m\ge 1$, $m=1,\ldots,n$.
For fixed $N$ and $n$, genericity is ensured by excluding zero sets of finitely many polynomials in
$u,\mathbf{v}$, and $\mathbf{s}^{(v)}$ from the parameter space. Later in \Cref{prop:positivity_parameter_dependence}, we present concrete conditions on the parameters producing stationary measures for all $N$ and $n$.
\begin{proof}[Proof of \Cref{prop:queue_mASEP_stationary}]
	Iterating \Cref{thm:twisted_st}, we have
	\begin{equation*}
		\big\langle \emptyset \big|
		\ssp
		\mathfrak{Q}
		\bigl( (u_1,u,\ldots,u ); \mathbf{s}_{\frac{1}{2}}^{(h)}; \mathbf{v}; \mathbf{s}^{(v)}\bigr)
		\ssp
		\mathfrak{T}(u, u_1)^{\lfloor (1-q)t / \epsilon \rfloor } = \big\langle \emptyset \big| \ssp
		\mathfrak{Q}
		\bigl( (u_1,u,\ldots,u ); \mathbf{s}_{\frac{1}{2}}^{(h)}; \mathbf{v}; \mathbf{s}^{(v)}\bigr).
	\end{equation*}
	Setting $u_1=u(1-\epsilon)$ and sending $\epsilon\to0$
	turns the power of the twisted cylinder Markov operator
	in the left-hand side into
	$\mathfrak{P}_{\mathrm{mASEP}}(t)$;
	see \Cref{prop:mASEP_from_R}. The vertex weights in the queue vertex model with horizontal spin $\frac12$ are given in \Cref{fig:ASEP_weights} (recall that $\mathfrak{Q}$ involves the weights
	$\WQ_{q^{-1/2},s_m^{(v)},u/v_m}^{(-m)}$).
	These vertex weights are continuous in the spectral parameter $u$ for generic parameters.
	Therefore, we can simply take the limit $u_1\to u$ in the queue transfer matrix, which implies that
	$\big\langle \emptyset \big|
	\ssp
	\mathfrak{Q}
	\bigl( (u,u,\ldots,u ); \mathbf{s}_{\frac{1}{2}}^{(h)}; \mathbf{v}; \mathbf{s}^{(v)}\bigr)$
	is the left (row) eigenvector of the Markov semigroup $\mathfrak{P}_{\mathrm{mASEP}}(t)$ with eigenvalue $1$ (where here the exchange of the limit and the implicit summation in both sides of the above equation is justified by the fact that the summation
    defining each partition function is uniformly bounded by
    a geometric series, and the sum over $\mathbf{V}_{\eta}$ on the left hand side
 has a bounded number of terms, independently of $\epsilon$).
	Thus, it is proportional to the row vector representing the stationary distribution of mASEP, as desired.
\end{proof}

\begin{figure}[htbp]
	\centering
	\includegraphics[width=.8\textwidth]{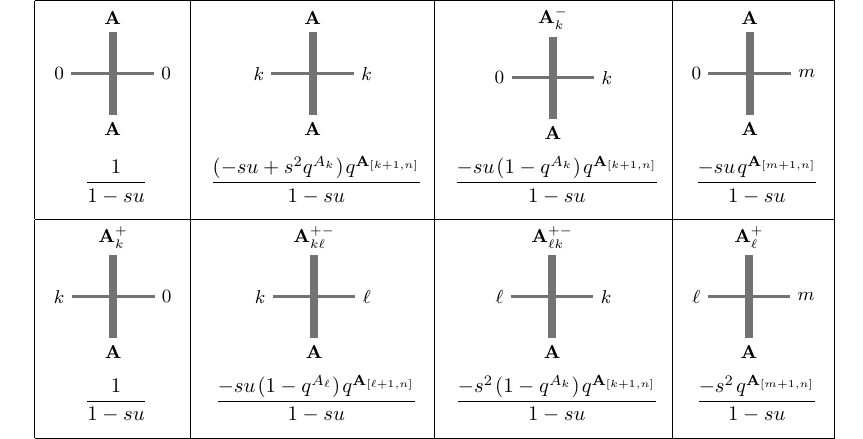}
	\caption{Weights
	$\WQ_{q^{-1/2},s,uq^{1/2}}^{(-m)}(\mathbf{A},k;\mathbf{C},\ell)$ entering
	the spin-$\frac12$
	queue vertex model
	which represents the
	stationary distribution of mASEP. Here
	$m<k<\ell\le n$, and it suffices to consider
	only vertices with no colors $\le m$ entering from the left.
	Recall that $A_m=+\infty$, so the vertices in the last column
	also satisfy the arrow conservation property.}
	\label{fig:ASEP_weights}
\end{figure}

\begin{remark}
	\label{rmk:queue_mASEP_stationary_remark_independent_of_parameters}
	\Cref{prop:queue_mASEP_stationary} expresses the stationary probabilities of mASEP, a system depending on a single parameter $q\in[0,1)$, as normalized partition functions
	$\langle \emptyset | \ssp \mathfrak{Q} \ssp | \mathbf{V}_\eta \rangle/
	Z^{\mathrm{mASEP}}_{N_1,\ldots,N_n }$
	of the queue vertex model. The latter in addition depends on the parameters $u,\mathbf{v}$, and $\mathbf{s}^{(v)}$. However, by \eqref{eq:queue_mASEP_stationary_proport}, we see that while $u,\mathbf{v},\mathbf{s}^{(v)}$ enter the weights of the queue vertex model, these parameters do not affect the normalized partition functions.
	This observation is not obvious, and we can claim the cancellation
	of the additional parameters
	only because we have the connections to mASEP which depends only on one parameter $q$.
\end{remark}

\begin{definition}\label{def:mASEP_gauge_weights}
For $s\ne 0$ let us define the \emph{mASEP gauge transformation} of the queue vertex weights:
\begin{equation}
	\label{eq:ASEP_gauge}
	\WQ_{q^{-1/2},s,uq^{1/2}}^{(-m),\mathrm{mASEP}_+}(\mathbf{A},k;\mathbf{C},\ell)
	\coloneqq
	(-1/s)^{\mathbf{1}_{\ell\ge1}}\ssp
	\WQ_{q^{-1/2},s,uq^{1/2}}^{(-m)}(\mathbf{A},k;\mathbf{C},\ell).
\end{equation}
That is, we
remove the factor $(-s)$
from the weights in the right three columns
in \Cref{fig:ASEP_weights}.
The resulting weights make sense for $s=0$, too.
The notation ``$+$'' in~\eqref{eq:ASEP_gauge} indicates that
we will impose conditions on the parameters under which
these weights are nonnegative.

Also, denote by
\begin{equation}
	\label{eq:ASEP_gauge_normalized}
	\frac{\langle \emptyset | \ssp \mathfrak{Q}^{\mathrm{mASEP}_+} \ssp | \mathbf{V}_\eta \rangle}{
Z^{\mathrm{mASEP}_+}_{N_1,\ldots,N_n }}
\end{equation}
the corresponding normalized partition function of the queue vertex model on the cylinder with the right boundary $\eta$.
\end{definition}

\begin{remark}
	The weights
	\eqref{eq:ASEP_gauge}
	are the queue limits
	(as in \Cref{sub:mqueue_spec}) of
	the non-stochastic higher spin colored vertex weights
	defined in
	\cite[(2.2.2)]{borodin_wheeler2018coloured}.
\end{remark}

\begin{proposition}
	\label{prop:positivity_parameter_dependence}
	Fix the type counts $(N_1,\ldots,N_n)$ with $N_m\ge1$ for all $m$. Let the parameters satisfy
	% \begin{equation}
	% 	\label{eq:positivity_parameter_dependence}
	% 	0 \le s_m^{(v)} < \frac{u}{v_m q^{1/2}}\le 1,
	% 	\qquad
	% 	m = 1,\ldots,n.
	% \end{equation}
        \begin{equation}
		\label{eq:positivity_parameter_dependence}
		0 \le s_m^{(v)} < \frac{u}{v_m q^{1/2}}, \qquad s_m^{(v)} \frac{u}{v_m q^{1/2}} < 1,
		\qquad
		m = 1,\ldots,n.
	\end{equation}
	Then
	\begin{equation*}
		\mathop{\mathrm{Prob}}^{\mathrm{mASEP}}\nolimits_{N_1,\ldots,N_n }(\eta)
		=
		\frac{
		\langle \emptyset | \ssp \mathfrak{Q}^{\mathrm{mASEP}_+} \ssp | \mathbf{V}_\eta \rangle}{
		Z^{\mathrm{mASEP}_+}_{N_1,\ldots,N_n }},
	\end{equation*}
	is the mASEP stationary distribution,
	and
	$\langle \emptyset | \ssp \mathfrak{Q}^{\mathrm{mASEP}_+} \ssp | \mathbf{V}_\eta \rangle>0$ for all $\eta$.
\end{proposition}
Note that conditions
\eqref{eq:positivity_parameter_dependence}
are written for~$u/v_m$ entering the vertex weights as
$\WQ_{q^{-1/2},s_m^{(v)},u/v_m}^{(-m),\mathrm{mASEP}_+}$,
while for displaying the weights in \Cref{fig:ASEP_weights}
it was convenient to
mulitply the spectral parameter by $q^{1/2}$.

\begin{proof}[Proof of \Cref{prop:positivity_parameter_dependence}]
	Assume first that $s_m^{(v)}>0$ for all $m$.
	Replacing the queue
	vertex weights
	with their gauge transformed versions
	\eqref{eq:ASEP_gauge} multiplies the partition
	function
	$\langle \emptyset | \ssp \mathfrak{Q} \ssp | \mathbf{V}_\eta \rangle$
	by
	\begin{equation*}
		\prod\nolimits_{m=1}^{n}(-1/s_m^{(v)})^{N_m+\ldots+N_n },
	\end{equation*}
	which depends only on the sector, but not on the configuration
	$\eta$. Therefore, the gauge transformation may be incorporated into the
	normalizing constant $Z^{\mathrm{mASEP}_+}_{N_1,\ldots,N_n }$.
	One readily sees that
	under conditions \eqref{eq:positivity_parameter_dependence}
	and when $s_m^{(v)}>0$ for all $m$,
	all vertex weights \eqref{eq:ASEP_gauge} are positive
	(see \Cref{fig:ASEP_weights}).
	This completes the proof in the case
	when all the $s_m^{(v)}$'s are positive.

	Setting
	some (or all) of the $s_m^{(v)}$'s
	to zero in the weights
	\eqref{eq:ASEP_gauge}
	is allowed, and leads to a well-defined partition function
	$\langle \emptyset | \ssp \mathfrak{Q}^{\mathrm{mASEP}_+} \ssp | \mathbf{V}_\eta \rangle$.
	In this partition function, some of the vertex weights
	in \Cref{fig:ASEP_weights} vanish.
	To show that $\langle \emptyset | \ssp \mathfrak{Q}^{\mathrm{mASEP}_+} \ssp | \mathbf{V}_\eta \rangle$ is still positive and not merely nonnegative for all $\eta$,
	first notice that there exists $\eta$ for which
	$\langle \emptyset | \ssp \mathfrak{Q}^{\mathrm{mASEP}_+} \ssp | \mathbf{V}_\eta \rangle\ne 0$
	(this verification is straightforward, and we omit it).
	Next, observe that
	\begin{equation}
		\label{eq:ASEP_gauge_positive_Perron_Frobenius}
		\sum\nolimits_{\eta} \langle \emptyset | \ssp \mathfrak{Q}^{\mathrm{mASEP}_+} \ssp | \mathbf{V}_\eta \rangle \langle \mathbf{V}_\eta |
	\end{equation}
	is a nonzero
	left (row) eigenvector with eigenvalue~$1$
	of the
	mASEP semigroup
	$\mathfrak{P}_{\mathrm{mASEP}}(t)$ corresponding to an
	irreducible continuous-time Markov process on a finite state space.
	Therefore,
	\eqref{eq:ASEP_gauge_positive_Perron_Frobenius}
	is proportional to the
	Perron--Frobenius eigenvector of
	$\mathfrak{P}_{\mathrm{mASEP}}(t)$, which has all components positive.
	This completes the proof.
\end{proof}

\subsection{Matching to multiline queues}
\label{sub:mqueues_martin_matching}

\subsubsection{Original multiline queues}
\label{subsub:mqueue_matching}

First, we connect the queue vertex model
$\mathfrak{Q}^{\mathrm{mASEP}_+}$
with special parameters with \emph{multiline queues}
introduced by Martin \cite{martin2020stationary}.
He showed that the output of the latter produces the
stationary distribution of mASEP on the ring.

Setting $u=q^{1/2}$ and
$s_m^{(v)}=0$, $v_m=1$ for all $m=1,\ldots,N $ in the vertex weights
\eqref{eq:ASEP_gauge} leads to the weights given in
\Cref{fig:mqueue_vertex_weights} which we
denote by
\begin{equation}
	\label{eq:WQ_mqueue}
	\WQ^{(-m),\mathrm{mq}}
	\coloneqq
	\WQ_{q^{-1/2},0,q^{1/2}}^{(-m),\mathrm{mASEP}_+}.
\end{equation}
By \Cref{prop:positivity_parameter_dependence}, these weights
produce positive partition functions of the queue vertex
model on the cylinder.

Let us connect
the queue vertex model $\mathfrak{Q}^{\mathrm{mASEP}_+}$
with these particular parameters
to multiline queue
diagrams. These diagrams were defined in \cite[Sections~1.1
and~3.6]{martin2020stationary}, and the vertex model
interpretation follows from formula (3.9) in
\cite{martin2020stationary}. For the reader's convenience,
in the rest of \Cref{subsub:mqueue_matching} we reproduce
the main definitions and the matching of queues to our
vertex models.

\begin{figure}[htbp]
	\centering
	\includegraphics[width=.75\textwidth]{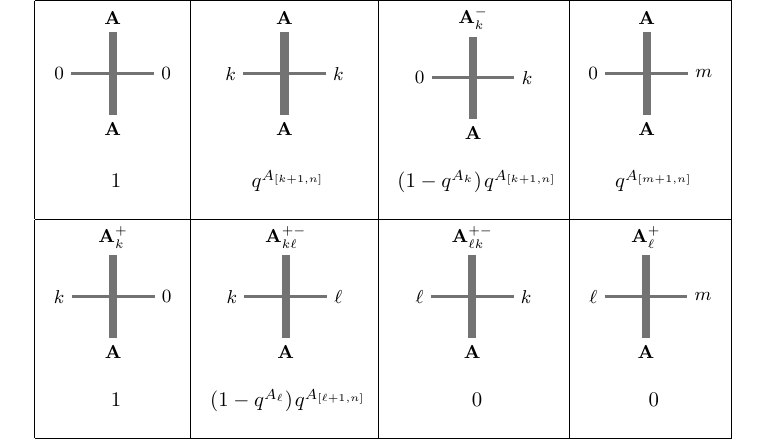}
	\caption{Weights
	$\WQ^{(-m),\mathrm{mq}}$
	\eqref{eq:WQ_mqueue}.
	In \Cref{prop:martin_matching} we match them to
	probabilities under
	multiline queues of \cite{martin2020stationary}.
	Here
	$m<k<\ell\le n$, and recall (\Cref{rmk:types_vs_colors}) that
	our ordering of colors is reversed
	compared to particle types in \cite{martin2020stationary}.}
	\label{fig:mqueue_vertex_weights}
\end{figure}

Let us recast \cite[Algorithm~2]{martin2020stationary}
(called \emph{the Martin algorithm} in what follows)
in the language
of one column of a vertex model on the cylinder.
Fix $m=1,\ldots,n $,
type counts $N_m,N_{m+1},\ldots,N_n $,
and
assume that
we have an arbitrary fixed configuration
$\eta$ of paths of colors strictly larger than $m$
entering the column $(-m)$ from the left.
The configuration $\eta$ has $N_i$ paths of color
$i$, $i>m$.
The Martin algorithm (for color $m$)
samples
a random new configuration $\eta'$ of paths
exiting the column $(-m)$ to the right.
The configuration $\eta'$ has $N_i$ paths of color $i \ge m$; it is constructed sequentially, through the following three steps.
\begin{enumerate}[\bf1.\/]
	\item
		Start with the empty configuration $\eta'=\{0,\ldots,0 \}$
		($N$ zeros).
		In addition, sample a uniformly random subset $\mathcal{J}\subset
		\{1,\ldots,N \}$ of sites on the cylinder
		of cardinality $N_m+\ldots+N_n $.
		Next, we randomly
		update $\eta'$
		such that in the end
		$\eta_j'>0$ if and only if $j \in \mathcal{J}$.
	\item
		For each color $i=n,n-1,\ldots,m+1 $ (in this order),
		let $a_1^i<\ldots<a^i_{N_i}$ be the locations of paths of color $i$ in $\eta$.
		\begin{enumerate}
	\item[\bf2a.\/] For $j=1,\ldots,N_i$,
			if $\eta'_{a_j^i}=0$ and~$a_j^i \in \mathcal{J}$, set $\eta'_{a_j^i}=i$
			(if a path of color $i$ can come straight through, it does so).
		\item[\bf2b.\/] Otherwise, the $j$-th path of color $i$ starts from $\eta_{a_j^i}$ and randomly chooses an exit site among yet unoccupied sites in $\mathcal{J}$ as follows. Let $a_j^i<p_1<p_2<\ldots<p_l$ (here,~$p < p'$ means that as we read upwards starting from~$p$, possibly               wrapping around in the vertical direction, we observe~$p'$ before getting back to~$a_{j}^i$), where $(p_1,\ldots,p_l )$ are all sites in $\mathcal{J}$ for which at this point we have $\eta'_{p_t}=0$, $t=1,\ldots,l $. Then, set $\eta'_{p_t}=i$ with probability proportional to $q^{t-1}$.
		\item[\bf2b'.\/]
			Equivalently, instead of step \textbf{2b}, one can think that the color $i$ path starting from site $a_j^i$ goes up the cylinder and sequentially with probability $1-q$ picks an unoccupied site from $\mathcal{J}$ to exit, or with probability $q$ skips this site (\emph{accepts} or \emph{declines the service}, in queueing terminology). The path continues the motion up the cylinder until its exit, and can go around the cylinder an arbitrary number of times. If the path exits at $p_t$, then after normalization, this produces the same probability proportional to $q^{t-1}$.
		\end{enumerate}
		\item Once all paths of all colors strictly larger than $m$ are processed, we have $N_m$ unoccupied sites in $\mathcal{J}$ left. We set $\eta'_j=m$ for all these remaining sites.
\end{enumerate}

To obtain the mASEP stationary distribution
$\mathop{\mathrm{Prob}}^{\mathrm{mASEP}}\nolimits_{N_1,\ldots,N_n }$
(\Cref{def:mASEP_stationary_distribution}),
one needs to apply the Martin algorithm
for color $n$ with input $\eta(0)=\emptyset$,
and get a random output $\eta(1)$.
Then apply the algorithm for color $n-1$
with input $\eta(1)$, get an output $\eta(2)$, and so on.
The final output $\eta(n)$ of the algorithm for color $1$
is the random configuration distributed according to the mASEP
stationary distribution.

\begin{figure}[htpb]
	\centering
	\includegraphics[width=.4\textwidth]{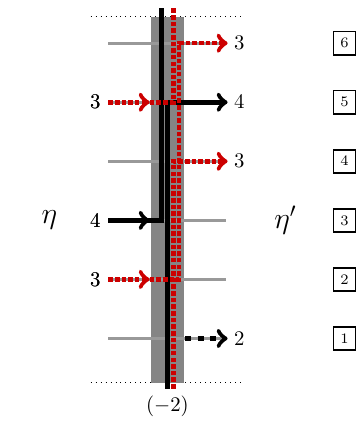}
	\caption{The Martin algorithm from \cite{martin2020stationary}
		with $N=6$, $n=4$, in the column $(-m)$, where $m=2$.
		The set $\mathcal{J}$ is $\left\{ 1,4,5,6 \right\}$.
		Given this $\mathcal{J}$, the conditional probability of the configuration in the figure (according to the description with \textbf{2b'}) is
		proportional to~$q^5$ (the color~$4$ path skips five possibilities)
		times~$q^2$ (the top color~$3$ path skips two possibilities)
		times~$1$ (the bottom color~$3$ path selects the first available possibility).
		Colors of the arrows are indicated by the numbers
		and by different dashing to assist the printed version.
		The boxed numbers on the right indicate the ring sites.}
	\label{fig:martin_algorithm}
\end{figure}

The following statement matches the output of the Martin algorithm to vertex models and essentially coincides with \cite[Theorem~3.4]{martin2020stationary}. For convenience, we reproduce it here.

\begin{proposition}
	\label{prop:martin_matching}
	In each sector determined by the fixed type counts $(N_1,\ldots,N_n)$,
	the output~$\eta$ of the Martin algorithm
	has the same distribution as the output of our
	queue vertex model on the cylinder with the
	weights $\WQ^{(-m),\mathrm{mq}}$ given in
	\Cref{fig:mqueue_vertex_weights}.
\end{proposition}
\begin{proof}[Idea of proof]
	This follows by matching the vertex weights in \Cref{fig:mqueue_vertex_weights} to the weights $w_i(Q | A, S)$ given in \cite[(3.9)]{martin2020stationary}. The translation from the queueing language to vertex models is straightforward and we omit it.
\end{proof}

In \cite{martin2020stationary}, the stationarity of the output $\eta$ of the Martin algorithm
under mASEP follows from the Matrix Product Ansatz. The connection between the algorithm and the Matrix Product Ansatz is essentially equivalent to \Cref{prop:martin_matching}. We link queue vertex models to Matrix Product Ansatz in \Cref{sub:YBE_MPA} below.

\begin{remark}
	\label{rmk:queue_martin_matching_subtle}
	While the Martin algorithm and the queue vertex model produce the same output~$\eta$ (in distribution, in each sector), it remains unclear whether one can define appropriate ``states'' of the queueing system under the Martin algorithm such that these states are in a weight-preserving bijection (possibly up to a common proportionality constant) with states of the queue vertex model. Indeed, tracking each particle's choices as in step \textbf{2b'} of the Martin algorithm with the input as in \Cref{fig:martin_algorithm} involves
	the following information:
	\begin{enumerate}[$\bullet$]
		\item Track how many times the path of color $4$ wraps around the cylinder.
		\item Pick a bijection between the color $3$ inputs and outputs (there are $2!$ choices in \Cref{fig:martin_algorithm}).
		\item For each of the two bijections, track how many times each of the two color $3$ paths wraps around the cylinder.
	\end{enumerate}
	Under the queue vertex model, we do not choose a bijection,
	and the wrapping arrows of colors $m+1,\ldots,n$ are encoded by the tuple~$\mathbf{M}(-m)$. Here
	$\mathbf{M}(-m)_k$ is the total number of arrows of color $k$ that wrap around the ring (that is, go between the sites $N$ and $1$).

	To show that summing over all data in the queueing system produces the desired distribution of the output $\eta$, one seems to require the intricate argument with compatible queue-length processes; see \cite[Section~4.2]{martin2020stationary} for details.
 \end{remark}

\subsubsection{Alternative multiline queues and an interpolation}
\label{subsub:mqueue_matching_alternative}

Let us consider a different specialization of the
queue vertex model
$\mathfrak{Q}^{\mathrm{mASEP}_+}$
introduced in
\Cref{sub:vertex_models_mASEP}.
Namely, set
$u=1$ and
$s = s_m^{(v)}=q$, $v_m=1$ for all $m=1,\ldots,N $
in~\eqref{eq:ASEP_gauge},
and clear the common denominators
$1-q$. The
resulting weights
$
\WQ^{(-m),\mathrm{mq\,alt}}
\coloneqq
(1-q)\ssp
\WQ_{q^{-1/2},q,q^{1/2}}^{(-m),\mathrm{mASEP}_+}
$
are given in
\Cref{fig:alt_mqueue_weights}.
By \Cref{prop:positivity_parameter_dependence}, they
lead to positive partition functions on the cylinder.

\begin{figure}[htpb]
	\centering
	\includegraphics[width=.75\textwidth]{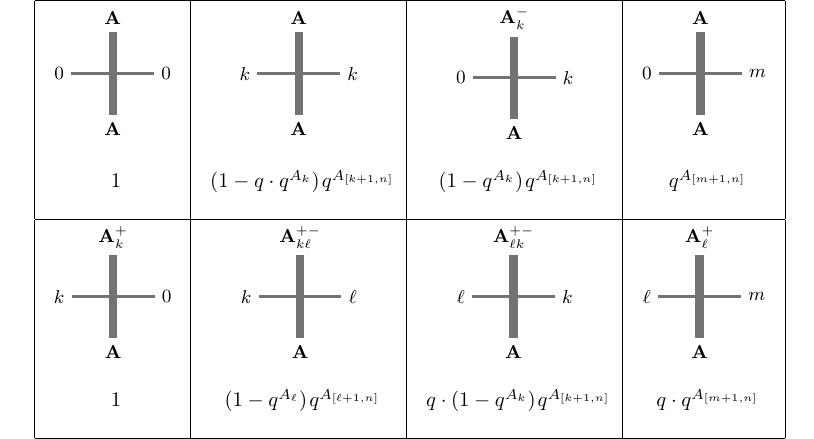}
	\caption{Weights $\WQ^{(-m),\mathrm{mq\,alt}}$
	for the alternative multiline queue model.}
	\label{fig:alt_mqueue_weights}
\end{figure}

% Arguing exactly as in the proof of
% \Cref{prop:martin_matching}, we see that configurations of the queue vertex model
% on the cylinder
% with the weights
% $\WQ^{(-m),\mathrm{mq\,alt}}$
% are in a weight-preserving bijection
% (up to a common proportionality constant)
% with queue diagrams of the

Just as by \Cref{prop:martin_matching},
the weights $\WQ^{(-m),\mathrm{mq}}$ from \Cref{fig:mqueue_vertex_weights}
produce the same output $\eta$ (in distribution) as the Martin algorithm,
the new weights $\WQ^{(-m),\mathrm{mq\,alt}}$
should be related to the
\emph{alternative multiline
queue model}
introduced in \cite[Section~7]{martin2020stationary}.

By definition, the alternative
algorithm for color $m$
consists of the same steps as the algorithm described in
\Cref{subsub:mqueue_matching} above, except step~\textbf{2a}.
Instead, if $\eta'_{a^i_j}=0$
and~$a^i_j\in \mathcal{J}$, then the entering
color $i$ path ($i>m$) from site $a^i_j$
still has probability $q$
to go up the cylinder and not exit immediately through
this site.

\begin{proposition}
	\label{prop:Martin_conjecture}
	In each sector determined by the fixed type counts $(N_1,\ldots,N_n)$, the output~$\eta$ of the alternative multiline queue algorithm has the same distribution as the output of the queue vertex model on the cylinder with the weights $\WQ^{(-m),\mathrm{mq\,alt}}$ given in \Cref{fig:alt_mqueue_weights}.
\end{proposition}
\begin{proof}[Idea of proof]
	This is established very
	similarly to
	\cite[Section~4.2]{martin2020stationary},
	by using the notion of compatibility of a queue-length process
	with ``marks''. Here marks encode geometrically distributed choices of how many times
	a service is rejected at a site before being accepted.
	These choices may be
	done beforehand. The queue-length
	process corresponds to the vertical paths in the queue vertex model.
	If one interprets the randomness in the alternative model in terms of
	marks, then one can see the vertex weights $\WQ^{(-m),\mathrm{mq\,alt}}$.
	We omit the technical details of this proof.
\end{proof}
\begin{corollary}[{Conjecture from \cite[Section~7]{martin2020stationary}}]
	The output $\eta$
	of the alternative multiline queue algorithm
	is stationary under the mASEP dynamics on the ring.
\end{corollary}
\begin{proof}
	We already know that
	the output of the queue vertex model with the weights $\WQ^{(-m),\mathrm{mq\,alt}}$ is  stationary under the mASEP dynamics on the ring, thanks to our general \Cref{prop:queue_mASEP_stationary,prop:positivity_parameter_dependence} which ultimately rely on the Yang--Baxter equation for the twisted cylinder (\Cref{thm:twisted_st}).
	Applying \Cref{prop:Martin_conjecture}, we arrive at the desired statement.
\end{proof}

\begin{figure}[htpb]
	\centering
	\includegraphics[width=.75\textwidth]{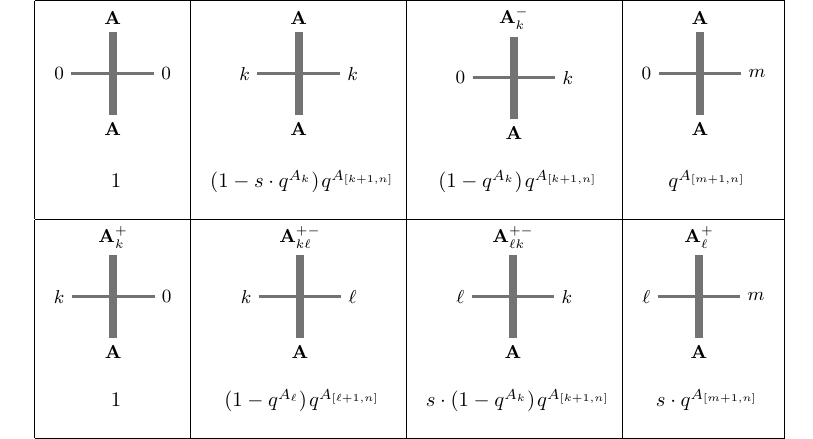}
	\caption{Weights $(1-s)\ssp\WQ_{q^{-1/2},s,q^{1/2}}^{(-m),\mathrm{mASEP}_+}$
	interpolating between the vertex weights related to the original and the alternative multiline queues of \cite{martin2020stationary}.}
	\label{fig:interpolating_mqueue_weights}
\end{figure}

\medskip

If we do not specialize the parameter $s$ to $0$ or $q$, we obtain a family of queue vertex weights depending on $q$ and $s$ (see \Cref{fig:interpolating_mqueue_weights}). The output of the vertex model with these weights produces the mASEP stationary distribution (this again follows from \Cref{prop:queue_mASEP_stationary,prop:positivity_parameter_dependence}). The $(q,s)$-dependent weights should be related to a new multiline queue model that interpolates between the original and the alternative multiline queues:

\begin{definition}[Interpolating multiline queues]
	\label{def:interpolating_mqueue}
	Let $s\in[0,1)$, and modify the Martin algorithm (for a given color $m$) by changing step \textbf{2a} as follows. If a color $i$ path ($i>m$) enters at $a^i_j$ and $\eta'_{a^i_j}=0$ (that is, a service is immediately available), then the path exits (accepts the service) with probability~$1-s$. With probability~$s$, the path turns up the cylinder and skips every successive available exit (service) with probability~$q$
	(as prescribed by step \textbf{2b'}).
	All other parts of the algorithm remain the same.

	For $q=0$ (when mASEP becomes the multi-species TASEP), the interpolating model produces a multiline queue model with \emph{random service assignment}. Note that for $q=0$, both the original and the alternative multiline queues become the same and are deterministic. For two colors, this deterministic model was
	constructed in \cite{angel2006stationary} to describe the stationary distribution of the two-color TASEP. It was generalized to $n$ colors in \cite{Ferrari_2007}.
\end{definition}

Similarly to the argument in \cite[Section~4.2]{martin2020stationary}, it should be possible to identify the output of the interpolating multiline queues with that of the $(q,s)$-dependent queue vertex model on the cylinder. In \Cref{sub:YBE_MPA} below, we outline a possible connection of the latter with the Matrix Product Ansatz.

\subsection{Connection to Matrix Product Ansatz}
\label{sub:YBE_MPA}

Prior to the multiline queue realization of the
mASEP stationary distribution
$\mathop{\mathrm{Prob}}^{\mathrm{mASEP}}\nolimits_{N_1,\ldots,N_n }$
in \cite{martin2020stationary},
Prolhac--Evans--Mallick
\cite{Prolhac_2009}
showed that
$\mathop{\mathrm{Prob}}^{\mathrm{mASEP}}\nolimits_{N_1,\ldots,N_n }$
can be expressed in a \emph{matrix product form}.
For processes on the ring, this expression has the same
format as our general
trace formula \eqref{eq:product_trace_formula},
see \eqref{eq:product_trace_formula_ASEP} below,
and includes matrices $\mathscr{X}^{\mathrm{MPA}}_{m}$, $m=0,1,\ldots,n $,
indexed by available colors.
Matrix product ansatz representations
for stationary probabilities of stochastic
interacting particle systems
date back to
\cite{Derrida1993solution}.
In the single-species case, the
stationary distribution on the ring is uniform,
and so the Matrix Product Ansatz becomes
nontrivial only for ASEP on an open interval,
in which particles can hop in and out at the endpoints.
This case was considered in
\cite{Derrida1993solution}.
For the two- and three-species ASEP on the ring, the
matrix product approach was employed, respectively, in
\cite{Derrida1993shock}
and
\cite{MallickMallickRajewsky1999}.
A full multi-species solution on the ring appeared about ten years
later in
\cite{Prolhac_2009}.
See also
\cite{BlytheEvansSolverGuide2007}
for an earlier survey of Matrix Product Ansatz
applications to particle systems.

In the multi-species case, the matrices
$\mathscr{X}^{\mathrm{MPA}}_{m}$ entering the product ansatz
are constructed by recursive tensoring from a few
single-species building blocks $A,D$, and $E$ satisfying
quadratic relations
\begin{equation}
	\label{eq:ADE_relations}
	AD-qDA=
	EA-qAE=(1-q)A,
	\qquad
	ED-qDE=(1-q)(E+D).
\end{equation}
The matrices $A,D,E$, as well as the $\mathscr{X}_m^{\mathrm{MPA}}$'s,
are infinite-dimensional, and
their products, as well
as the trace of $A$ times a finite product of $D$ and $E$
matrices must be well-defined.
The tensoring construction of
$\mathscr{X}^{\mathrm{MPA}}_{m}$ resembles the process of
horizontally stacking the vertices in columns
$-n,-n+1,\ldots,-1 $ as in \Cref{fig:product_trace_formula}.
We refer to \cite[(24)--(33)]{Prolhac_2009} or
\cite[Section~2]{martin2020stationary} for details on the tensoring
construction, and omit them here.

Once the matrix product probability distribution is defined
through the trace in an appropriate space as
\begin{equation}
	\label{eq:product_trace_formula_ASEP}
	\mathop{\mathrm{Prob}}^{\mathrm{mASEP}}\nolimits_{N_1,\ldots,N_n }
	(\eta)=
	\frac{\mathop{\mathrm{Trace}}
	\left( \mathscr{X}^{\mathrm{MPA}}_{\eta_1}\cdots \mathscr{X}^{\mathrm{MPA}}_{\eta_N}
	\right)}{Z^{\mathrm{MPA}}_{N_1,\ldots,N_n }}
	,
	\qquad \eta=(\eta_1,\ldots,\eta_N ),
\end{equation}
one must independently check that it is stationary
under the mASEP dynamics.
A key property in the argument is the existence of
the so-called \emph{hat matrices} $\widehat{\mathscr{X}}_m^{\mathrm{MPA}}$,
$m=0,1,\ldots,N$,
satisfying quadratic relations
\cite[(68)]{Prolhac_2009}:
\begin{equation}
	\label{eq:hat_matrices}
	\sum_{i,i'=0}^{n}
	\mathscr{X}_i^{\mathrm{MPA}}
	\mathscr{X}_{i'}^{\mathrm{MPA}}
	\big( \mathscr{M}_{loc} \big)_{i i', j j'}
	=
	\mathscr{X}_j^{\mathrm{MPA}}
	\widehat{\mathscr{X}}_{j'}^{\mathrm{MPA}}
	-
	\widehat{\mathscr{X}}_{j}^{\mathrm{MPA}}
	\mathscr{X}_{j'}^{\mathrm{MPA}}
	,
\end{equation}
where $\mathscr{M}_{loc}$ are the local infinitesimal rates
of the mASEP, see
\eqref{eq:Mloc}. These hat matrices are also constructed in \cite{Prolhac_2009} by recursive tensoring procedures.
Note that our notation differs from
\cite{Prolhac_2009}
by a transposition
(e.g. comparing~\eqref{eq:hat_matrices} with formula (66) in \cite{Prolhac_2009}). See also \cite{arita2012generalized} for a more general family of matrices
$\mathscr{X}_j^{\mathrm{MPA}}$,
$\widehat{\mathscr{X}}_{j}^{\mathrm{MPA}}$
satisfying \eqref{eq:hat_matrices}.

\medskip

Let us explain how the construction of the hat matrices,
and identity \eqref{eq:hat_matrices} for vertex model partition functions
directly
follow from the Yang--Baxter equation.
We take a concrete realization of the
Matrix Product Ansatz matrices
$\mathscr{X}_m^{\mathrm{MPA}}$
using our vertex models.

Namely, let $\mathscr{X}_j^{\mathrm{MPA}}(u)$, $j=0,1,\ldots,n$,
be operators in $V_{-n}\otimes\ldots\otimes V_{-1}$,
where $V_{-m}$ has basis
$|\mathbf{M}(-m)\rangle$,
$\mathbf{M}(-m)\in \mathbb{Z}_{\ge0}^{n}$,
$m=1,\ldots,n $. By definition, the matrix elements of
$\mathscr{X}^{\mathrm{MPA}}_j$ are partition functions
(with weights $\WQ_{q^{-1/2},s,u}$)
on the one-row lattice $\{1\}\times \{-n,\ldots,-1 \}$
with boundary conditions $0$ and $j$ on the left and right,
respectively.
See
\Cref{fig:product_trace_formula}
for an illustration.

Let $\epsilon>0$ be small.
By the Yang--Baxter equation (\Cref{prop:queue_YBE}),
the matrices $\mathscr{X}^{\mathrm{MPA}}_i(u)$
and
$\mathscr{X}^{\mathrm{MPA}}_{i'}(u(1-\epsilon))$
satisfy the following identity involving the
elements of $R_{1-\epsilon}$ \eqref{eq:R_matrix_nonfused}:
\begin{equation}
	\label{eq:hat_eqn_from_R}
	\sum_{i,i'=0}^{n}
	\mathscr{X}^{\mathrm{MPA}}_i(u)
	\ssp
	\mathscr{X}^{\mathrm{MPA}}_{i'}(u(1-\epsilon))
	\cdot
	R_{1-\epsilon}(i,i';j',j)=
	\mathscr{X}^{\mathrm{MPA}}_{j}(u(1-\epsilon))
	\ssp
	\mathscr{X}^{\mathrm{MPA}}_{j'}(u).
\end{equation}
As in the proof of
\Cref{prop:mASEP_from_R},
let us differentiate
\eqref{eq:hat_eqn_from_R} with respect to $\epsilon$
at $\epsilon=0$.
Denote
\begin{equation*}
	\widehat{\mathscr{X}}^{\mathrm{MPA}}_j(u)\coloneqq
	(1-q)\ssp u\ssp \frac{\partial}{\partial u}\ssp \mathscr{X}^{\mathrm{MPA}}_j(u),\qquad j=0,1,\ldots,n .
\end{equation*}
\begin{proposition}
	\label{prop:hat_identity}
	The matrices
	$\mathscr{X}_i^{\mathrm{MPA}}(u)$,
	$\widehat{\mathscr{X}}_j^{\mathrm{MPA}}(u)$
	defined above with the help of the queue vertex weights
	$\WQ_{q^{-1/2},s,u}$
	satisfy the hat matrix identity \eqref{eq:hat_matrices}.
\end{proposition}
Our realization of the
matrices $\mathscr{X}_j^{\mathrm{MPA}}$,
$\widehat{\mathscr{X}}_{j}^{\mathrm{MPA}}$
satisfying the hat relation \eqref{eq:hat_matrices}
as vertex model partition functions highlights the Yang--Baxter
structure of the Matrix Product Ansatz for the multispecies ASEP
which was previously less transparent (see, however, \cite{cantini2015matrix}
and \cite{crampe2014integrable} for earlier connections between the two
structures).
\begin{proof}[Proof of \Cref{prop:hat_identity}]
	The $\epsilon$-derivative at $\epsilon=0$ of
	the right-hand side of
	\eqref{eq:hat_eqn_from_R} is
	equal to $$- (1-q)^{-1}\widehat{\mathscr{X}}_j^{\mathrm{MPA}}(u)
	\mathscr{X}_{j'}^{\mathrm{MPA}}(u).$$
	In the left-hand side, differentiating $\mathscr{X}_{i'}^{\mathrm{MPA}}$
	and noticing that
	$R_{1}(i,i';j',j)=\mathbf{1}_{i=j}\mathbf{1}_{i'=j'}$,
	we obtain $- (1-q)^{-1} \mathscr{X}^{\mathrm{MPA}}_j(u)
	\widehat{\mathscr{X}}^{\mathrm{MPA}}_{j'}(u)$.
	This yields the second summand in the right-hand
	side of \eqref{eq:hat_matrices}.
	Finally, differentiating $R_{1-\epsilon}$ and using
	\eqref{eq:Mloc_as_derivative},
	we recover the local infinitesimal
	rates of mASEP, also multiplied by $(1-q)^{-1}$.
	This completes the proof.
\end{proof}

\begin{figure}
    \centering
		\includegraphics[scale=.8]{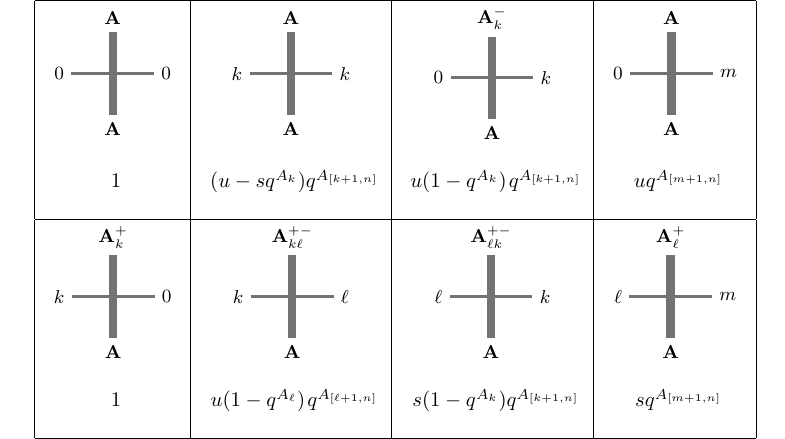}
    \caption{Table of weights $\WQ_{q^{-1/2},s,uq^{1/2}}^{(-m),\mathrm{mASEP}_+}$ \eqref{eq:ASEP_gauge} with common denominators $(1-su)$ cleared, which are used in the construction of the matrices $A, D, E$
	satisfying \eqref{eq:ADE_relations}.}
    \label{fig:Lm_weights_nodenom}
\end{figure}

Let us also
note that one can ``recognize'' the
matrix product building blocks $A,D,E$ in the vertex
weights of the queue vertex model on the cylinder.
This allows to insert two extra parameters into the matrices.
Namely,
in \cite{martin2020stationary}, two examples of $A,D,E$
are presented, for the original and for the alternative
multiline queue models
discussed in \Cref{sub:mqueues_martin_matching} above.
Both of these examples satisfy the relations
\eqref{eq:ADE_relations}.
In these two examples, the matrices $A, D , E$ are equal to the $n = 2$ species versions of the corresponding $\mathscr{X}^{\mathrm{MPA}}$ matrices. Informed by this, and now keeping track of the parameters $u$ and $s$ in the vertex weights,
let us define
\begin{equation}
	\label{eq:ADE_12}
	\begin{split}
		A&
		\coloneqq
		\left(
		\begin{array}{ccccc}
		1      & s      & 0        & \ldots \\
		0      & q      & q s      & \ldots \\
		0      & 0      & q^2      & \ldots \\
		\vdots & \vdots & \vdots &  \ddots \\
		\end{array}
		\right),
		\qquad
		D\coloneqq
		u^{-1}
		\left(
		\begin{array}{ccccc}
		u-s & 0 & 0           & \ldots \\
		1-q & u- sq & 0       & \ldots \\
		0 & 1-q^2 & u-sq^2    & \ldots \\
		\vdots & \vdots & \vdots  & \ddots \\
		\end{array}
		\right),
		\\
		E&\coloneqq
		\left(
		\begin{array}{ccccc}
		1 & u & 0 &  \ldots \\
		0 & 1 & u &  \ldots \\
		0 & 0 & 1 &  \ldots \\
		\vdots & \vdots & \vdots & \ddots \\
		\end{array}
		\right).
	\end{split}
\end{equation}
These matrices can be obtained as one-row partition functions of the two-color queue-specialized vertex model as shown in Figure~\ref{fig:ADE}. The relevant vertex weights $\WQ_{q^{-1/2},s,uq^{1/2}}^{(-m),\mathrm{mASEP}_+}$ are shown in Figure~\ref{fig:Lm_weights_nodenom}.
We use the same parameter $u$ in both columns.

 \begin{figure}[h]
    \centering
		\includegraphics[scale=.9]{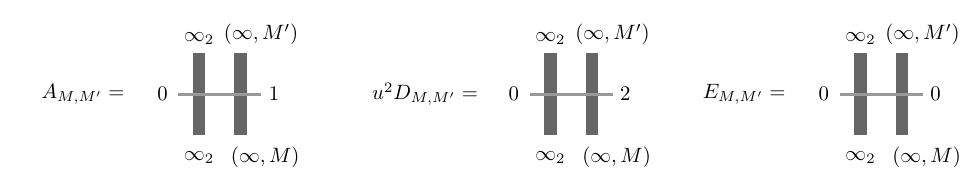}
    \caption{The $(M,M')$ martrix elements of the matrices $A, D, E$ are obtained from the queue-specialized vertex model. The symbol $\infty_2$ represents $(0, +\infty)$ (infinitely many color $2$ paths, and no color $1$ paths). Similarly, $(\infty, M)$ means a path configuration with $M$ color $2$ paths and $+\infty$ color $1$ paths.}
    \label{fig:ADE}
\end{figure}

One can directly check that the matrices \eqref{eq:ADE_12}
satisfy
\eqref{eq:ADE_relations} for all $u,s$. Alternatively, relations
\eqref{eq:ADE_relations}
can be obtained from the Yang--Baxter equation, satisfied by $A, u^2 D, E$ and the $R$ matrix, in the same way as described in the proof of Proposition \ref{prop:hat_identity}.
When $u=1$ and $s=0$, the matrices~\eqref{eq:ADE_12}
become the matrix product
building blocks of \cite{Prolhac_2009}, and lead to
the
original multiline queue model
\cite[(2.5)]{martin2020stationary}.
For $u=1$ and $s=q$, they
correspond to the alternative model from
\cite[Section~7]{martin2020stationary}.
However, for the alternative queue, it is not known whether the
recursive tensoring of the matrices $A,D,E$ (with $u=1$, $s=q$)
like in \cite{Prolhac_2009}
produces a matrix
product ansatz
for mASEP. We do not pursue this question here.

We remark that
matrices similar to
\eqref{eq:ADE_12}
have occurred in the product
ansatz for the open ASEP on a bounded interval
(for example, see \cite{Derrida1993solution}).
There, the extra parameters are tied to the
boundary rates.

\section{Colored stochastic $q$-Boson process from straight cylinder}
\label{sec:qBoson}

In this section we consider a
specialization of the straight cylinder
Markov transition operator leading to the
colored stochastic $q$-Boson process on the ring
\cite{Takeyama2015},
\cite[Section~12.4]{borodin_wheeler2018coloured}.
It is also called the
multi-species totally asymmetric zero range process
(mTAZRP)
in \cite{ayyer2022modified}.
The corresponding specialization
of the queue vertex model
will allow us to recover the stationary distribution of the $q$-Boson process.
In the case of at most one particle of each color,
we also match path configurations
in the vertex model representing this stationary distribution
to states of a multiline queue
considered in \cite[Section~8]{ayyer2022modified}.

\subsection{Colored $q$-Boson process on the ring}
\label{sub:qBoson_process}

Let us fix the size of the ring $N$ and the number of
colors $n$. Also let us fix
the type counts $(N_1,\ldots,N_n )$, where $N_i\ge1$
stands for the number of particles of color $i$
in the system.
The state space of the colored
stochastic $q$-Boson process consists of configurations
of particles at sites of the ring, where at each site
there can be an arbitrary number of particles.
The configurations are encoded by
\begin{equation*}
	\mathbf{V}=(\mathbf{V}(1),\ldots,\mathbf{V}(N) ),
	\qquad \mathbf{V}(j)\in \mathbb{Z}_{\ge0}^n.
\end{equation*}
Here $\mathbf{V}(j)_i$ denotes the number of particles of
color $i$ at site $j$, and
$N_i=\sum_{j=1}^{N}\mathbf{V}(j)_i$.

\begin{definition}
	\label{def:cqBoson}
	The \emph{stochastic colored $q$-Boson process}
	depends on parameters
	$q\in[0,1)$ and
	$u_1,\ldots,u_N>0 $, and evolves in continuous
	time as follows.
	A particle
	of color $i$ hops
	from site $k$ to site $k-1$
	(cyclically mod $N$)
	according to an independent exponential
	clock
	with rate
	\begin{equation*}
		u_k^{-1}(1-q^{\mathbf{V}(k)_i})\ssp q^{\mathbf{V}(k)_{[i+1,n]}}.
	\end{equation*}
	Here we used the usual notation
	$\mathbf{V}(k)_{[i+1,n]}=\sum_{r=i+1}^n\mathbf{V}(k)_r$.
	Denote by
	$\mathfrak{P}_{\mathrm{qBos}}(t)$,
	$t\in \mathbb{R}_{\ge0}$,
	the continuous-time Markov semigroup of this stochastic process.
\end{definition}
The colored $q$-Boson process evolution is of \emph{zero range} kind, that is,
the jump from site $k$ depends only on the
state of the system at site $k$.
In \cite{ayyer2022modified} it is referred to as the \emph{multi-species
totally asymmetric zero range process}, or \emph{mTAZRP}.

The $q$-Boson process
preserves the type counts
$(N_1,\ldots,N_n )$.
For a fixed vector of type counts, this continuous-time Markov chain
evolves on a finite state space and
is clearly irreducible. Thus, it has a
unique stationary distribution.
We denote it by
$\mathop{\mathrm{Prob}}^{\mathrm{qBos}}\nolimits_{N_1,\ldots,N_n}(\mathbf{V})$.

\medskip

Following
\cite[Section~12.4.3]{borodin_wheeler2018coloured},
we can identify
$\mathfrak{P}_{\mathrm{qBos}}(t)$
as a certain Poisson-type continuous-time
limit of the straight cylinder
formal Markov operator
$\mathfrak{S}(x,\mathbf{u};\mathbf{s}^{(h)})$
defined in \Cref{sub:straight}.
Recall that $\mathfrak{S}$
has $N+1$ sites on the ring.
However, in the degeneration
to the $q$-Boson process, the
distinguished
site
corresponding to spectral parameter $x$
(cf.~\Cref{fig:straight})
will be empty with probability $1$, and the dynamics can be restricted
to $N$ sites.

Fix small $\epsilon>0$, and set the horizontal spin parameters of $\mathfrak{S}$ to $s_j^{(h)}=\epsilon$. Also let $x=-1$. With this specialization, the matrix elements of the operators $\mathscr{L}_{s_j^{(h)},xu_j^{-1}}$ entering the definition of~$\mathfrak{S}$ (formula \eqref{eq:straight_cylinder_Markov_operator}, see also \Cref{fig:L_weights}) become as in \Cref{fig:qboson_specialization}. The operator $\mathfrak{S}$ is a formal Markov operator acting on states of the form~$(c, \mathbf{V})$, where~$\mathbf{V}\in \mathbb{Z}_{\ge0}^{n}$ is a state of the color $q$-Boson process, and $c \in\left\{ 0,1,\ldots,n  \right\}$ corresponds to the auxiliary line (i.e., the one with the spectral parameter $x$ and the spin parameter $q^{-1/2}$; see \Cref{fig:straight}).

\begin{figure}[htbp]
	\centering
	\includegraphics[width=.8\textwidth]{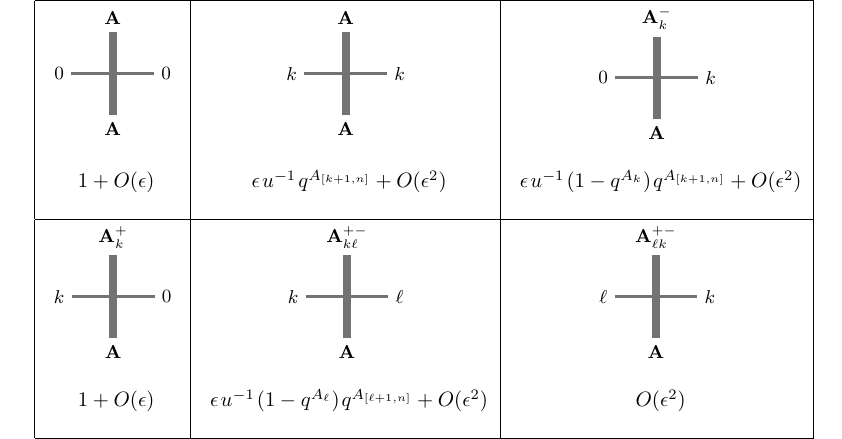}
	\caption{
	The weights
	$\mathscr{L}_{\epsilon,-u_j^{-1}}$ employed in the
	approximation of the colored $q$-Boson process,
	Taylor expanded to $O(\epsilon^2)$
	or $1+O(\epsilon)$ depending on whether they go to $0$
	or $1$ as $\epsilon\to0$.
	Here, as usual, $1\le k<\ell \le n$.}
 \label{fig:qboson_specialization}
\end{figure}

\begin{proposition}
	\label{prop:limit_to_cqboson}
	Fix $t\in \mathbb{R}_{\ge0}$.
	With the parameter specialization as above, have
	\begin{equation*}
		\lim
		_{\epsilon\to0}\,
		\langle 0, \mathbf{V}|
		\ssp
		\mathfrak{S}(-1,\mathbf{u};(\epsilon,
		\ldots,\epsilon ))^{\lfloor t/\epsilon \rfloor}
		\ssp
		|0, \mathbf{V}'\rangle
		=
		\langle \mathbf{V}|
		\ssp
		\mathfrak{P}_{\mathrm{qBos}}(t)
		\ssp
		|\mathbf{V}'\rangle,
		\qquad \mathbf{V},\mathbf{V}'\in \mathbb{Z}_{\ge0}^n.
	\end{equation*}
	Moreover, for any $c\ge 1$ we have
	$
	\lim
	_{\epsilon\to0}
	\langle 0, \mathbf{V}|
	\ssp
	\mathfrak{S}(-1,\mathbf{u};(\epsilon,
	\ldots,\epsilon ))^{\lfloor t/\epsilon \rfloor}
	\ssp
	|c, \mathbf{V}'\rangle
	=0$.
\end{proposition}
Not all matrix elements of $\mathfrak{S}$ are nonnegative before the $\epsilon\to0$ limit. This is not a problem because the limiting semigroup $\mathfrak{P}_{\mathrm{qBos}}(t)$ is a nonnegative Markov semigroup, and the stationarity result (which we prove in \Cref{prop:qBos_stationary} below) is a purely algebraic statement.
\begin{proof}[Proof of \Cref{prop:limit_to_cqboson}]
	Both statements follow from the expansions
	in \Cref{fig:qboson_specialization},
	after the identification of the vertices in the cylinder
	(in \Cref{fig:straight})
	with the stochastic $q$-Boson transitions via
	$$
	\begin{tikzpicture}[baseline=(current bounding box.center),scale=0.6]
		\draw[white!45!black,line width=1.5pt,->] (-5.2,1) --++ (.4,-2);
		\draw[white!45!black,line width=4pt,->] (-6,0) -- (-4,0);
		\node[left] at (-6,0) {\small $\mathbf{A}$};
		\node[right] at (-4,0) {\small $\mathbf{C}$};
		\node[above] at (-5.2,1) {\small $k$};
		\node[below] at (-4.8,-1) {\small $\ell$};
		\node[left] at (-2,0) {\small $=$};
		\draw[white!45!black,line width=1.5pt,->] (-1,0) -- (1,0);
		\draw[white!45!black,line width=4pt,->] (0,-1) -- (0,1);
		\node[left] at (-1,0) {\small $k$};
		\node[right] at (1,0) {\small $\ell$};
		\node[below] at (0,-1) {\small $\mathbf{A}$};
		\node[above] at (0,1) {\small $\mathbf{C}$};
		\node[above] at (2.2,0) { \phantom{.} };
	 \end{tikzpicture},
	$$
	where in the right-hand side the time is continuous and
	increases in the upward direction.

	We see that the auxiliary line may become occupied with probability
	$O(\epsilon)$, and then instantaneously becomes free again with
	probability $1+O(\epsilon)$.
	This means that the vertices of the type $(\mathbf{A},k;\mathbf{A},k)$
	are not present in the limit.
	All other probabilities of order $O(\epsilon)$ in
	\Cref{fig:qboson_specialization}
	give rise to the corresponding colored $q$-Boson transitions,
	which leads to the first claim.

	To get the second claim, observe that
	with probability going to $1$ in the limit as $\epsilon\to0$,
	the auxiliary line is not occupied.
	This completes the proof.
\end{proof}

\subsection{Vertex model for the $q$-Boson stationary distribution}
\label{sub:qBoson_stationary}

The convergence of \Cref{prop:limit_to_cqboson} together
with the general stationarity result
(\Cref{thm:straight_st}) allows us to express the stationary
distribution
$\mathop{\mathrm{Prob}}^{\mathrm{qBos}}\nolimits_{N_1,\ldots,N_n}(\mathbf{V})$
of the colored $q$-Boson process as a vertex model partition
function.

To get a queue vertex model on the cylinder
with nonnegative vertex weights, we
take a certain limit in the
vertical parameters $\mathbf{v},\mathbf{s}^{(v)}$.
As a first step, let us consider the following
degeneration of the queue vertex weights
$\WQ_{s_1,s_2,u}^{(-m)}$ \eqref{eq:queue_spec_fully_fused}:

\begin{lemma}
	\label{lemma:qboson_weights_degeneration}
	We have
	\begin{equation}
		\label{eq:q_boson_weights_degeneration_vertex_s_dependent}
		\begin{split}
			&
			( u; q)_{\infty}^{-1}\cdot
			\lim_{s_1\to 0}
			\WQ_{s_1,s,us_1/s}^{(-m)}(\mathbf{A},\mathbf{B};\mathbf{C},\mathbf{D})
			=
			\mathbf{1}_{\mathbf{A}+\mathbf{B}=\mathbf{C}+\mathbf{D}}
			\cdot
			\mathbf{1}_{D_1=\ldots=D_{m-1}=0}
			\\&\hspace{20pt}\times
			\sum_{\mathbf{P}}
			(s^2/u ; q)_{|\mathbf{P}|}
			(s^2)^{|\mathbf{B}|-|\mathbf{P}|}
			u^{|\mathbf{D}|-|\mathbf{B}|+|\mathbf{P}|}
			\ssp
			q^{\sum_{1\le i<j\le n}\left(B_i-P_i\right) P_j}
			\prod_{i=1}^n
			\frac{(q;q)_{B_i}}{(q;q)_{P_i}(q;q)_{B_i-P_i}}
			\\&\hspace{25pt}\times
			\ssp q^{\sum_{m\le i<j\le n} D_i (C_j-P_j)}
			\ssp
			\frac{
			1}
			{(q;q)_{D_m}}
			\prod_{i=m+1}^n
			\frac{(q;q)_{C_i-P_i+D_i}}{(q;q)_{C_i-P_i}(q;q)_{D_i}},
		\end{split}
	\end{equation}
	where
	the sum is over $\mathbf{P}\in \mathbb{Z}_{\ge0}^n$ with $0\le P_i\le \min (B_i,C_i)$ for all $i$,
	and
	\begin{equation}
		\label{eq:q_boson_weights_degeneration_vertex}
		\begin{split}
			&
			( u; q)_{\infty}^{-1}\cdot
			\lim_{s\to0}\Bigl(\ssp\lim_{s_1\to 0}
				\WQ_{s_1,s,us_1/s}^{(-m)}(\mathbf{A},\mathbf{B};\mathbf{C},\mathbf{D})
			\Bigr)
			=
			\mathbf{1}_{\mathbf{A}+\mathbf{B}=\mathbf{C}+\mathbf{D}}
			\cdot
			\mathbf{1}_{D_1=\ldots=D_{m-1}=0}\cdot
			\prod_{i=1}^n\mathbf{1}_{B_i\le C_i}
			\\&\hspace{130pt}\times
			u^{|\mathbf{D}|}\ssp
			q^{\sum_{m\le i<j\le n} D_i (C_j-B_j)}
			\ssp
			\frac{
			1}
			{(q;q)_{D_m}}
			\prod_{i=m+1}^n
			\frac{(q;q)_{C_i-B_i+D_i}}{(q;q)_{C_i-B_i}(q;q)_{D_i}}.
		\end{split}
	\end{equation}
\end{lemma}

\begin{proof}
	We have from \eqref{eq:queue_spec_fully_fused}:
	\begin{equation*}
		\begin{split}
			&
			\WQ_{s_1,s,us_1/s}^{(-m)}(\mathbf{A},\mathbf{B};\mathbf{C},\mathbf{D})
			=
			\mathbf{1}_{\mathbf{A}+\mathbf{B}=\mathbf{C}+\mathbf{D}}
			\cdot
			\mathbf{1}_{D_1=\ldots=D_{m-1}=0}
			\cdot
			\frac{( u; q)_{\infty}}{(s_1^2 u ; q)_{\infty}}
			\\&\hspace{20pt}\times
			\sum_{\mathbf{P}}
			\frac{(s^2/u ; q)_{|\mathbf{P}|}
			(s_1^2u/s^2 ; q)_{|\mathbf{B}-\mathbf{P}|}}
			{(s_1^2 ; q)_{|\mathbf{B}|}}\ssp
			\ssp
			q^{\sum_{1\le i<j\le n}\left(B_i-P_i\right) P_j}
			\prod_{i=1}^n
			\frac{(q;q)_{B_i}}{(q;q)_{P_i}(q;q)_{B_i-P_i}}
			\\&\hspace{25pt}\times
			(s^2)^{|\mathbf{B}|-|\mathbf{P}|}
			u^{|\mathbf{D}|-|\mathbf{B}|+|\mathbf{P}|}
			\ssp q^{\sum_{m\le i<j\le n} D_i (C_j-P_j)}
			\ssp
			\frac{
			(s_1^2 ; q)_{|\mathbf{D}|}}
			{(q;q)_{D_m}}
			\prod_{i=m+1}^n
			\frac{(q;q)_{C_i-P_i+D_i}}{(q;q)_{C_i-P_i}(q;q)_{D_i}}.
		\end{split}
	\end{equation*}
	Sending $s_1\to0$
	immediately leads to \eqref{eq:q_boson_weights_degeneration_vertex_s_dependent}.
	Further letting
	$s\to0$, we see that
	$\mathbf{P}=\mathbf{B}$, for otherwise the factor
	$(s^2)^{|\mathbf{B}|-|\mathbf{P}|}$ vanishes. This
	eliminates the summation over $\mathbf{P}$ and produces
	the desired
	expression~\eqref{eq:q_boson_weights_degeneration_vertex}
	together with the indicator that $B_i\le C_i$ for all $i$.
\end{proof}

We denote the right-hand side of
\eqref{eq:q_boson_weights_degeneration_vertex_s_dependent}
by
$\WQ_{s,u}^{(-m),\mathrm{qBos}}(\mathbf{A},\mathbf{B};\mathbf{C},\mathbf{D})$.
The right-hand side of
\eqref{eq:q_boson_weights_degeneration_vertex} is
the $s=0$ degeneration of \eqref{eq:q_boson_weights_degeneration_vertex_s_dependent}.
The weights
$\WQ_{s,u}^{(-m),\mathrm{qBos}}$
are
nonnegative
for $q\in[0,1)$ and $u>s^2\ge 0$.

\begin{definition}
	\label{def:qBoson_queue_vertex_model}
Fix parameters
$\mathbf{u}=(u_1,\ldots,u_N )$,
$\mathbf{y}=(y_1,\ldots,y_n)$,
and
$\mathbf{s}^{(v)}=(s_1^{(v)},\ldots,s_n^{(v)} )$
such that
\begin{equation}
	\label{eq:s_u_y_conditions_qBoson}
	0\le (s_m^{(v)})^2<u_iy_m,\qquad i=1,\ldots,N ,\quad m=1,\ldots,n .
\end{equation}
Let
$\mathfrak{Q}^{\mathrm{qBos}}(\mathbf{u};\mathbf{y};\mathbf{s}^{(v)})$
denote the queue transfer matrix
on the $n\times N$ cylinder
as in \Cref{fig:queue_state},
where the vertex weight at each site $(-m,j)$ is
$\WQ_{s_m^{(v)},u_jy_m}^{(-m),\mathrm{qBos}}$.
\end{definition}

The vertex model of \Cref{def:qBoson_queue_vertex_model} has nonnegative weights. Note that its partition functions
$\langle \emptyset |\ssp\mathfrak{Q}^{\mathrm{qBos}}(\mathbf{u};\mathbf{y};\mathbf{s}^{(v)})\ssp
|\mathbf{V}\rangle$
involve infinite sums over paths winding around the cylinder.
Similarly to \Cref{lemma:queue_vertex_well_defined},
we see that
these sums are
convergent when
$\mathbf{V}$
has all types appearing at least once.

\begin{proposition}\label{prop:qBos_stationary}
	For any type counts
	$(N_1,\ldots,N_n )$, $N_i\ge1$,
	and the parameters $\mathbf{y},\mathbf{s}^{(v)}$
	satisfying \eqref{eq:s_u_y_conditions_qBoson}, the
	stationary
	distribution
	of the colored $q$-Boson process
	with parameters $\mathbf{u}$
	has the form
	\begin{equation}
		\label{eq:qBoson_stationary_proposition}
		\mathop{\mathrm{Prob}}^{\mathrm{qBos}}\nolimits_{N_1,\ldots,N_n}(\mathbf{V})
		=
		\frac{\langle\emptyset | \ssp
			\mathfrak{Q}^{\mathrm{qBos}}(\mathbf{u};\mathbf{y};\mathbf{s}^{(v)})
		\ssp | \mathbf{V}\rangle}{Z_{N_1,\ldots,N_n }^{\mathrm{qBos}}}.
	\end{equation}
	The normalizing constant
	$Z_{N_1,\ldots,N_n }^{\mathrm{qBos}}$
	depends on the parameters and the type counts,
	but not on the state
	$\mathbf{V}$
	within the sector determined by $(N_1,\ldots,N_n)$.
\end{proposition}
\begin{proof}
	We use \Cref{thm:straight_st}
	(in particular, recall the queue vertex model on the cylinder interacting with the
	straight cylinder Markov operator as illustrated in
	\Cref{fig:straight_commuted}).
	Let us choose the parameters of the queue vertex model
	\begin{equation*}
		\mathfrak{Q}=\mathfrak{Q}\left( (xq^{\frac12},\mathbf{u});(q^{-\frac{1}{2}},\mathbf{s}^{(h)})
		; \mathbf{v};\mathbf{s}^{(v)}\right)
	\end{equation*}
	as
	\begin{equation*}
		x=-1,\qquad s_j^{(h)}=\epsilon\to 0,
		\qquad
		v_m=\frac{s_m^{(v)}}{\epsilon\ssp y_m}
	\end{equation*}
	for all $1\le j\le N$, $1\le m\le n$.
	By \Cref{lemma:qboson_weights_degeneration},
	sending $\epsilon\to0$
	turns the weight at each site
	$(-m,j)$
	of this queue vertex model on the cylinder
	into
	$(u_jy_m;q)_\infty\WQ_{s_m^{(v)},u_j y_m}^{(-m),\mathrm{qBos}}$.
	The overall factor $\prod_{j=1}^{N}\prod_{m=1}^n(u_jy_m;q)_{\infty}$
	is absorbed into the normalizing constant, and thus we can ignore it.

	At the sites $(-m,0)$, before the limit we have the weights
	$\WQ_{q^{-1/2},s_m^{(v)},-\epsilon q^{1/2}y_m/s_m^{(v)}}^{(-m)}$.
	Up to re\-para\-met\-rization, these are the same weights as in \Cref{fig:ASEP_weights}.
	Sending $\epsilon\to0$ (that is, $-su\to 0$ in the notation \Cref{fig:ASEP_weights}),
	we see that
	\begin{equation*}
		\WQ_{q^{-1/2},s_m^{(v)},-\epsilon q^{1/2}y_m/v_m^{(v)}}^{(-m)}
		(\mathbf{A},0;\mathbf{A},0)\to 1,
		\qquad
		\WQ_{q^{-1/2},s_m^{(v)},-\epsilon q^{1/2}y_m/v_m^{(v)}}^{(-m)}
		(\mathbf{A},0;\mathbf{A}_k^{-},k)\to 0.
	\end{equation*}
	Since the auxiliary line
	(i.e., the one with the spin parameter $q^{-1/2}$)
	begins as initially unoccupied, these convergences imply that this auxiliary line remains unoccupied in the cylindrical queue vertex model under this limit (more specifically, any configuration in which it is occupied has weight
	going to $0$ as $\epsilon\to0$).
	Therefore, we can remove this auxiliary line from the model on the cylinder
	as follows:
	\begin{equation}
		\label{eq:going_from_Nplus1_to_N_in_the_model}
		\lim_{\epsilon\to0}\ssp
		\big\langle 0,\emptyset \big|
		\ssp
		\mathfrak{Q}\left( (xq^{\frac12},\mathbf{u});(q^{-\frac{1}{2}},\mathbf{s}^{(h)})
		; \mathbf{v};\mathbf{s}^{(v)}\right)
		\ssp
		\big| 0, \mathbf{V} \big\rangle
		=
		\langle \emptyset |\ssp \mathfrak{Q}^{\mathrm{qBos}}
		(\mathbf{u};\mathbf{y};\mathbf{s}^{(v)})
		\ssp
		| \mathbf{V} \rangle,
	\end{equation}
	where
	$\mathfrak{Q}^{\mathrm{qBos}}(\mathbf{u};\mathbf{y})$ is defined before the proposition.

	Arguing as in the proof of
	\Cref{prop:queue_mASEP_stationary},
	we can take the limit as $\epsilon\to0$
	simultaneously in the queue vertex model and in the straight
	cylinder Markov operator.
	Before the limit, these operators satisfy the general stationarity
	relation
	of \Cref{thm:straight_st}.
	By
	\Cref{prop:limit_to_cqboson},
	the straight cylinder Markov operator
	converges as $\epsilon\to0$
	(in the Poisson-type continuous-time limit)
	to the Markov semigroup
	$\mathfrak{P}_{\mathrm{qBos}}(t)$.
	The limit of the general stationarity relation
	yields
 \begin{align*}
 \langle \emptyset |\ssp  \mathfrak{Q}^{\mathrm{qBos}}
		(\mathbf{u};\mathbf{y}; &  \mathbf{s}^{(v)})  \mathfrak{P}_{\mathrm{qBos}}(t)
		\ssp
		| \mathbf{V} \rangle \\
  &=   \lim
		_{\epsilon\to0}\,
		\big\langle 0,\emptyset \big|
		\ssp
		\mathfrak{Q}\left( (xq^{\frac12},\mathbf{u});(q^{-\frac{1}{2}},\mathbf{s}^{(h)})
		; \mathbf{v};\mathbf{s}^{(v)}\right)
		\ssp
		\mathfrak{S}(-1,\mathbf{u};(\epsilon,
		\ldots,\epsilon ))^{\lfloor t/\epsilon \rfloor}
		\ssp
		|0, \mathbf{V}\rangle \\
		&=
		\lim
		_{\epsilon\to0}\,
		\big\langle 0,\emptyset \big|
		\ssp
		\mathfrak{Q}\left( (xq^{\frac12},\mathbf{u});(q^{-\frac{1}{2}},\mathbf{s}^{(h)}); \mathbf{v};\mathbf{s}^{(v)}\right)
		|0, \mathbf{V}\rangle
		\\
&= \langle \emptyset |\ssp \mathfrak{Q}^{\mathrm{qBos}}
		(\mathbf{u};\mathbf{y};  \mathbf{s}^{(v)})
		| \mathbf{V} \rangle .
\end{align*}

\noindent Here, the first equality holds by \Cref{prop:limit_to_cqboson} and equation \eqref{eq:going_from_Nplus1_to_N_in_the_model}; the second holds by \Cref{thm:straight_st}; and the third holds again by \eqref{eq:going_from_Nplus1_to_N_in_the_model}.

This completes the proof.
\end{proof}

\begin{remark}
While the quantities in the right-hand side of
\eqref{eq:qBoson_stationary_proposition}
seem to depend on~$\mathbf{y}$ and~$\mathbf{s}^{(v)}$,
\Cref{prop:qBos_stationary}
implies that they are independent of these extra parameters.
This observation is parallel to the mASEP situation
(see \Cref{rmk:queue_mASEP_stationary_remark_independent_of_parameters}).
\end{remark}

\subsection{Matching to multiline queues}
\label{sub:qBoson_multiline}

In \cite[Section~8]{ayyer2022modified}, a multiline queue model for the stationary distribution of the colored $q$-Boson process is presented. Let us match this model to our queue vertex model $\mathfrak{Q}^{\mathrm{qBos}}$ on the cylinder, specialized to $s_m^{(v)}=0$, $1\le m\le n$ (that is, with the simpler product-form weights~\eqref{eq:q_boson_weights_degeneration_vertex}).

As in \cite{ayyer2022modified},
we restrict our attention to the simpler \emph{strict} case when, by definition, there is at
most one particle of each color.
First, we recall the definition of a $q$-Boson multiline queue and its weight.
We replace the parameter $t$ from \cite{ayyer2022modified} by our
$q$, and adjust the notation of integer indices, spectral parameters, and the direction of the ring
to match the conventions
used throughout our paper.

\begin{figure}[htpb]
	\centering
	\includegraphics[width=.4\textwidth]{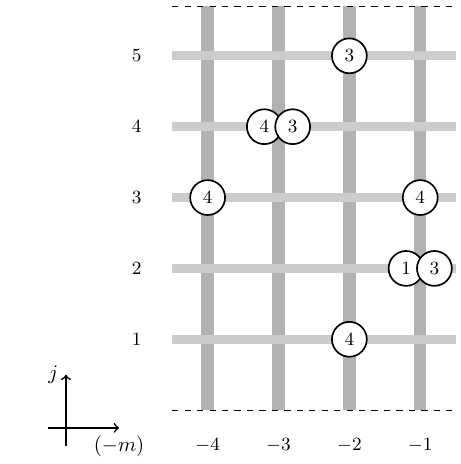}
	\caption{A multiline diagram (\Cref{def:mqueue_for_qBos}) with weight
	$q^{3}u_1u_2^2u_3^2u_4^2u_5$.
	Here the refusal statistic $3$ combines
	$R_{3}=1$ (label $3$ in column~$-2$ is ``between'' the positions of labels $4$ in columns~$-3$ and~$-2$, in the sense described after Equation~\eqref{eq:refusal_statistic_m}) and $R_{2}=2$
	(labels $1$ and $3$ in column~$-1$ are between the labels $4$ in columns~$-2$ and~$-1$).
	This is the same diagram as in examples in \cite[Section~8]{ayyer2022modified}, but
	rotated by $90^\circ$ and with the direction of the ring reversed (to match our vertex model).
	Here the size of the ring is $N=5$, and the number of colors is $n=4$.}
	\label{fig:mqueue_qboson}
\end{figure}

\begin{definition}[\cite{ayyer2022modified}]
	\label{def:mqueue_for_qBos}
	A multiline diagram is an assignment of
	the labels from $\{1,\ldots,n \}$ to
	the vertices of a cylinder
	$\{-n,\ldots,-1 \}\times(\mathbb{Z}/N\mathbb{Z})$,
	satisfying
	\begin{enumerate}[$\bullet$]
		\item Each vertex $(-m,j)$ is assigned a multiset of labels.
		\item In column $(-m)$, all labels are from $\{m,m+1,\ldots,n \}$.
		\item The combined multiset of all labels in column $(-m)$ is obtained from the multiset
			of labels in column $-(m+1)$, together with some new labels of type $m$.

		\item (\emph{strict condition}) Each label $m$, $1\le m\le n$, appears at most once
			in each of the columns
			$-m,-(m-1),\ldots,-1 $.
	\end{enumerate}
	The weight of a multiline diagram is, by definition,
	$q^\mathfrak{R} u_1^{c_1}\ldots u_N^{c_N}$,
	where $c_j$ is the total number of labels assigned to the row $j$, and $\mathfrak{R}$ is
	the \emph{refusal statistic} defined as follows.
	Let
	\begin{equation}
		\label{eq:refusal_statistic_m}
		R_{m}\coloneqq \sum_{m-1\le k<\ell\le n}
		\mathbf{1}_{p_\ell(-(m-1))> p_k(-(m-1))> p_\ell(-m)},
	\end{equation}
	where $p_r(-m)$ is the position
	of the label $r$ in column $-m$, and the event $a>b>c$ means that, reading along the ring
	in the downward direction (corresponding to decreasing positions $j$),
	the label $b$ is strictly between $a$ and $c$.
	This includes the case $a=c\ne b$;
    what this means for a corresponding term in the sum
    is~$p_{\ell}(-m) = p_\ell(-(m-1))$,
    and we think of this as~$p_\ell$ making a
    full loop around the ring to get to its position at~$-(m-1)$.
	Then we set $\mathfrak{R}\coloneqq \sum_{m=2}^{n}R_{m}$.
	See \Cref{fig:mqueue_qboson} for an illustration.
\end{definition}

Given a multiline diagram,
associate to it a path configuration
on the cylinder with vertex weights
$\WQ_{u_j}^{(-m),\mathrm{qBos}}$ at each vertex $(-m,j)$,
and such that the multiset of labels
at $(-m,j)$ is exactly the colors of the paths exiting this vertex.
Recall that we usually denote the latter multiset of colors by $\mathbf{D}\in \mathbb{Z}_{\ge0}^{n}$.
Knowing $\mathbf{D}$ at each vertex is enough to reconstruct the
whole path configuration on the cylinder,
up to unknown windings of paths around the cylinder.
In this way, one multiline diagram corresponds to
many
configurations of the queue vertex model
$\mathfrak{Q}^{\mathrm{qBos}}$ on the cylinder.

\begin{proposition}
	\label{prop:mqueue_to_qboson}
	Let there be exactly one particle of each color $m$, $m=1,\ldots,n$.
	Then
	the mapping between multiline diagrams and configurations of the queue vertex model
	$\mathfrak{Q}^{\mathrm{qBos}}(\mathbf{u};\mathbf{1};\mathbf{0})$ (that is, $y_i=1$ and
	$s_i^{(v)}=0$ for all~$i$)
	described before the proposition is weight-preserving. That is,
	the sum of weights of all vertex model configurations over the
	winding of the paths around the cylinder is proportional to
	the weight of the corresponding multiline diagram.
	The proportionality constant depends on the parameters of the model,
	but not on the particle configuration.
\end{proposition}
If there are no particles of some color, then the sum of the vertex model
weights might diverge, cf.~\Cref{rmk:queue_vertex_not_well_defined_for_not_full}.
On the other hand, we consider the multiline queues for
at most one particle of each color. This leads to the
restriction in \Cref{prop:mqueue_to_qboson}.

\begin{proof}[Proof of \Cref{prop:mqueue_to_qboson}]
	It suffices to fix $m$ and consider
	the behavior in the column $(-m)$.
	For a configuration of the queue vertex model in this column
	let the arrow configurations at each vertex $(-m,k)$ be
	$\mathbf{A}^{(k)},\mathbf{B}^{(k)},\mathbf{C}^{(k)},\mathbf{D}^{(k)}$.
	The corresponding multiline diagram contains information about
	$\mathbf{B}^{(k)},\mathbf{D}^{(k)}$, but not about
	$\mathbf{A}^{(k)},\mathbf{C}^{(k)}$. Let us
	fix $\mathbf{B}^{(k)},\mathbf{D}^{(k)}$ for all $k=1,\ldots,N$,
	and
	sum over
	$\mathbf{A}^{(k)},\mathbf{C}^{(k)}$,
	$k=1,\ldots,N$. The resulting sum must be equal to the weight of column~$-m$ in the
	corresponding multiline diagram.

	With this data fixed, out
	of all allowed configurations of the vertices in column~$(-m)$,
	there is one in which~$C_j^{(N)}$ is minimal for
	each~$j > m$. Fixing $\mathbf{C}^{(N)}$ allows one to reconstruct
	the whole vertex model configuration in column~$(-m)$ in a unique way.
	Denote this minimal configuration by~$\mathbf{C}^{(k), \min}$, and
	let~$E_j^{(k)} \coloneqq C_j^{(k), \min} -
	B_j^{(k)}$,~$j = m+1,\dots, n$.

	The product of the vertex weights
	\eqref{eq:q_boson_weights_degeneration_vertex}
	in column~$(-m)$, summed over all allowed configurations, is proportional to (using the fact that~$D_j^{(k)} \in \{0,1\}$ for all~$j, k$)
	\begin{equation}
		\label{eq:mqueue_qboson_proof}
		\begin{split}
			&
			\sum_{a_{m+1}=0}^\infty \cdots \sum_{a_n = 0}^{\infty}
			\prod_{k=1}^N \Biggl(\ssp
			u_k^{|\mathbf{D}^{(k)}|} \ssp
			q^{\sum_{m\le r < s\le n} D_r^{(k)}
			(E_s^{(k)}+a_s)}
			\prod_{j=m+1}^n
			\frac{(q;q)_{a_j+E_j^{(k)}+D^{(k)}_j}}
			{(q;q)_{a_j+E_j^{(k)}}(q;q)_{D^{(k)}_j}}
			\Biggr)
			\\
			&\hspace{5pt}=
			\Biggl(\ssp\prod_{k=1}^N u_k^{|\mathbf{D}^{(k)}|}
			\Biggr)
			\sum_{a_{m+1}=0}^\infty \cdots \sum_{a_n = 0}^{\infty}
			\Biggl(\ssp\prod_{k=1}^N
			q^{\sum_{m\le r < s\le n} D_r^{(k)}
			(E_s^{(k)}+a_s)}\Biggr)
			\Biggl(\ssp
			\prod_{j=m+1}^n
			\prod_{k\colon D_j^{(k)}=1}
			\frac{1-q^{1+a_j+E_j^{(k)}}}{1-q}
			\Biggr).
		\end{split}
	\end{equation}
	Observe that $B_j^{(k)}\le C_j^{(k)}$ for all $j$ and $k$
	(see \eqref{eq:q_boson_weights_degeneration_vertex}).
	This implies (by arrow conservation, since~$A_j^{(k)}, B_j^{(k)}, C_j^{(k)}, D_j^{(k)} \in \{0, 1\}$) that if $D_j^{(k)}=1$, then $E_j^{(k)}=0$.
	For each $j=m+1,\ldots,n $, we either have
	$D_j^{(k)}=0$ or $D_j^{(k)}=1$,
    and there exists exactly one~$k = k_j$ for which~$D_j^{(k)}=1$.
    Thus, inside the summations we have
    $$\Biggl(\ssp\prod_{k=1}^N
			\prod_{j=m+1}^n q^{\sum_{m\le r < j} D_r^{(k)}
			(E_j^{(k)}+a_j)}\Biggr)
			\Biggl(\ssp
			\prod_{j=m+1}^n
			\frac{1-q^{1+a_j}}{1-q}
			\Biggr).$$
	As a result, the sum over $a_j$ becomes
	\begin{equation*}
		q^{ \sum_{k=1}^N( D_m^{(k)}+\ldots+D_{j-1}^{(k)} )\ssp E_j^{(k)}}
		\sum_{a_j=0}^{\infty}
		q^{a_j \sum_{k=1}^N( D_m^{(k)}+\ldots+D_{j-1}^{(k)} )}
		\ssp
		\frac{1-q^{1+a_j}}{1-q}
		=
		C^{[j]}_{N_1,\ldots,N_n }\ssp
		q^{ \sum_{k=1}^N( D_m^{(k)}+\ldots+D_{j-1}^{(k)} )\ssp E_j^{(k)}}
		,
	\end{equation*}
	where $C^{[j]}_{N_1,\ldots,N_n }$
	does not depend on the particular multiline diagram
	but only on the type counts $(N_1,\ldots,N_n )$.
	Indeed, $\sum_{k=1}^N( D_m^{(k)}+\ldots+D_{j-1}^{(k)} )$
	is the total number of colors $i$, $m\le i\le j-1$,
	leaving column $(-m)$.
	% Similarly, when $D_j^{(k)}=1$, the sum over $a_j$
	% becomes
	% \begin{equation*}
	% 	\sum_{a_j=0}^{\infty}
	% 	\frac{1-q^{1+a_j}}{1-q}
	% 	\ssp
	% 	q^{a_j \sum_{k=1}^N( D_m^{(k)}+\ldots+D_{j-1}^{(k)} )}
	% 	=
	% 	C^{(2),j}_{N_1,\ldots,N_n }.
	% \end{equation*}
	Thus, we can continue
	\begin{equation}
		\label{eq:mqueue_qboson_proof_2}
		\eqref{eq:mqueue_qboson_proof}=
		C_{N_1,\ldots,N_n}
		\Biggl(\ssp\prod_{k=1}^N u_k^{|\mathbf{D}^{(k)}|}
		\Biggr)
		\prod_{k=1}^{N}
		q^{\sum_{m\le i<j\le n}D_i^{(k)}E_j^{(k)}},
	\end{equation}
	where $C_{N_1,\ldots,N_n}$ also depends only on the type counts.
	Note that
	$D_i^{(k)},
	E_j^{(k)}\in\left\{ 0,1 \right\}$.
	One can readily verify that
	each pair $m\le i<j\le n$ such that
	$D_i^{(k)}=
	E_j^{(k)}=1$
	corresponds to an indicator equal to one
	in the definition of $R_m$~\eqref{eq:refusal_statistic_m}.
	In particular, note that the indicator~$\mathbf{1}_{B_j \leq C_j}$ in the weights~\eqref{eq:q_boson_weights_degeneration_vertex} prevents a path from passing straight through without any winding. This behavior
	is accounted for in~\eqref{eq:mqueue_qboson_proof}, and corresponds to the fact that the case~$a = c \neq b$ counts towards the refusal statistic~$\mathfrak{R}$ (see the discussion after its definition~\eqref{eq:refusal_statistic_m}).
	Thus, the power of $q$ in
	\eqref{eq:mqueue_qboson_proof_2}
	is exactly the same as the component
	$R_m$ of the refusal statistic $\mathfrak{R}$.
	The powers of the $u_j$'s also match the ones for the multiline diagrams.
	This completes the proof.
\end{proof}

Let us make two final remarks in this section. First,
\cite{ayyer2022modified}
worked with the so-called tableau process instead of general
(not necessarily strict) multiline queue diagrams; as such,
they did not give explicit weights in these cases.
The general tableau process should correspond
to our vertex model as follows:

\begin{conjecture}
	\label{conj:multiline_qboson_AMM}
	In the general case with no restrictions on the number of particles of each color, consider the vertex model
	$\mathfrak{Q}^{\mathrm{qBos}}(\mathbf{u};\mathbf{1};\mathbf{0})$
	on the cylinder. Each vertex model configuration
	corresponds to a multiline diagram, though this mapping is
	many-to-one. The weight of a multiline diagram is defined
	as the sum of the weights of all corresponding vertex model
	configurations over the winding of paths around the cylinder.
	We conjecture that this weight is proportional to the one derived from the
	tableau process of
	\cite[Section~4]{ayyer2022modified}.
\end{conjecture}

Second, the Yang--Baxter equation for the queue vertex model (\Cref{prop:queue_YBE}) should allow us to directly show the symmetry of the stationary distribution in the parameters $u_j$. More precisely \cite[Proposition~7.2]{ayyer2022modified}, for any $K$, the distribution of the configuration at sites $\{1,\ldots,K \}$ of the ring is symmetric in the parameters $u_{K+1},\ldots,u_N$. Moreover, using the Yang--Baxter equation and couplings similarly to \cite{petrov2022rewriting}, it should be possible to establish the stronger symmetry of the distributions of the whole trajectories of the colored $q$-Boson system. This stronger property is proven only for $q=0$ \cite[Theorem~7.14]{ayyer2022modified}. We leave these two questions for future work.

\section{Colored $q$-PushTASEP from straight cylinder}
\label{sec:qPush}

This section considers another specialization of the straight cylinder Markov transition operator leading to the colored $q$-PushTASEP. We also present a vertex model on the cylinder producing its stationary distribution. Our argument here is very similar to \Cref{sec:ASEP_matrix_products,sec:qBoson} above.
The colored $q$-PushTASEP is a degeneration of the colored stochastic higher spin six-vertex model and was introduced in \cite[Section~12.5]{borodin_wheeler2018coloured}.

Throughout the section, we assume that $q\in (0,1)$ and fix a positive integer $\mathsf{P}$. As usual, let $N$ be the size of the ring, and $n$ be the number of colors. The colored $q$-PushTASEP depends on positive parameters
$\mathbf{u} = (u_1,\dots, u_N)$.

\begin{definition}
	\label{def:qPushTASEP}
	The state space of the colored $q$-PushTASEP is the set of particle configurations on the ring. At any site, there can be at most $\mathsf{P}$ particles. Particles of the same color are indistinguishable.
	Let $V_{\mathsf{P}}$ be the vector space with the basis
	$| \mathbf{V}\rangle$, where $\mathbf{V}\in \mathbb{Z}_{\ge0}^n$
	with $|\mathbf{V}|\le \mathsf{P}$. The states of the
	colored $q$-PushTASEP can be identified with the
	basis vectors of $V_{\mathsf{P}}^{\otimes N}$.
	The $q$-PushTASEP evolves in continuous time as follows.
	Let $\mathbf{A}\in \mathbb{Z}_{\ge0}^{n}$ be the
	configuration of particles at a site $k$.
	For each $j=1,\ldots,n $,
	a
	particle of type~$j$~\emph{activates} and instantaneously
	leaves the
	site~$k$ (moving toward~$k+1$) with the
	rate~$u_k^{-1}(q^{-A_j} - 1)\ssp
	q^{\mathsf{P}-A_{[j+1, n]} }$.

	The active particle triggers other instantaneous
	updates of the configuration according to the following
	rules.
	Let $\mathbf{B}\in \mathbb{Z}_{\ge0}^n$
	be the configuration of particles at a site~$k'$.
	Suppose that an activated particle of type~$c$
	arrives at~$k'$.
	Then the following happens:
		\begin{enumerate}[$\bullet$]
			\item It deactivates and stays at~$k'$
				with probability~$1- q^{\mathsf{P}-|\mathbf{B}|}$,
				then the update ends.
			\item
				It deactivates and stays at~$k'$, but causes the
				activation of exactly one particle from~$k'$ (which
				then moves towards~$k'+1$) of some type~$d < c$. The
				type $d$ of the activated particle is chosen with
				probability~$(q^{-B_d}-1) \ssp
				q^{\mathsf{P}-B_{[d+1, n]} }$ (the particles of the
				same type are indistinguishable).
				After the activated particle leaves the site $k'$,
				the update continues.
			\item
				It remains active and moves on to site~$k' + 1$ with probability~$q^{\mathsf{P}- B_{[c,n]} }$, and the update continues.
		\end{enumerate}
		Note that all probabilities in the three cases above sum to one.
		All particle moves from $j$ to $j+1$ are
		considered cyclically mod $N$.
		Denote the Markov semigroup of the colored $q$-PushTASEP by
		$\mathfrak{P}_{\mathrm{qPush}}(t)$, $t\in \mathbb{R}_{\ge0}$.
\end{definition}

Let us describe a limit of this model as $q\to 0$
(which corresponds to setting $q=\infty$ in \cite[Section~12.5]{borodin_wheeler2018coloured}).
There are several changes to \Cref{def:qPushTASEP} for $q=0$:
\begin{enumerate}[$\bullet$]
	\item The activation of a particle is only possible when the site $k$ contains the full
		number $\mathsf{P}$ of particles. Then only the particle with the minimal type
		is activated, at rate $u_k^{-1}$.
	\item If a particle of type $c$ arrives at a site $k'$ with $<\mathsf{P}$ particles, then it deactivates and stays
		there with probability $1$.
	\item
		If the site is full, then the particle of type $c$ deactivates and stays at $k'$,
		but activates the particle of the minimal available type $d$ at site $k'$.
		The activated particle moves to site $k'+1$, and the update continues.
\end{enumerate}

In particular, for $\mathsf{P}=1$, we get the following process which
appeared in another context:
\begin{remark}[Frog model]
	Set $q=0$, $\mathsf{P}=1$, and $u_k=1$ for all $k$ (the most
	simplified version of the colored $q$-PushTASEP).
	In this case, each particle at any site $k$ can be activated at rate $1$, and moves from $k$ to~$k+1$. Then the instantaneous update proceeds as follows:
	\begin{enumerate}[$\bullet$]
		\item
			If an active particle arrives at an empty site, it deactivates and stays there, and the update ends.
		\item
			If a type $c$ particle arrives at a site~$k'$ with an existing particle of smaller type $d$, then the
			type $c$ particle stops at $k'$, and displaces the type $d$ particle, which now becomes active.
		\item
			Finally, if a particle arrives at a site~$k'$ with an existing particle of larger or equal type, then
			it just moves through to the next site~$k'+1$ and stays active, so the update continues.
	\end{enumerate}
	This process is a particular case of the \emph{frog model}
	\cite{bukh2019periodic}
	related to the problem of the longest common subsequence of a
	random and a periodic word. Our particular case corresponds to the periodic word with all letters distinct.
	More general periodic words lead to the simultaneous activation of particles at several sites.
	The stationary distribution of the frog model was constructed (in the particular case of distinct letters)
	in \cite[Section~4]{bukh2019periodic}.
\end{remark}

As usual, by $(N_1,\ldots, N_n)$ we denote the type counts in the configuration, which are preserved by the $q$-PushTASEP dynamics.
For $q\ne 0$,
when restricted to a sector determined by
$(N_1,\ldots,N_n)$, the colored $q$-PushTASEP is an irreducible continuous-time Markov chain on a finite state space. Therefore, it admits a unique stationary distribution which we denote by
$\mathop{\mathrm{Prob}}^{\mathrm{qPush}}\nolimits_{N_1,\ldots,N_n }(\mathbf{V})$.

\begin{remark}
	\label{rmk:not_irreducible}
	When $q=0$, the colored $q$-PushTASEP
	on the ring
	is not irreducible. Indeed, one can
	check that on
	the ring with $N=2$ sites, $\mathsf{P}=2$, and
	three particles of the types $1,2,3$,
	not all states communicate. Indeed, there are no transitions
	into the states $(1\mid 23)$ and $(12\mid 3)$
	(in this notation, we view the particles as jumping to the right).
\end{remark}

The colored $q$-PushTASEP is a degeneration of the straight cylinder formal Markov operator. Thus, its stationary distribution is accessible through the corresponding limit transition from the queue vertex model on the cylinder. These limits are very similar to the $q$-Boson case (\Cref{sub:qBoson_process,sub:qBoson_stationary}), so we will only provide pictorial illustrations and brief explanations.

A queue vertex model leading to the $q$-PushTASEP stationary distribution must have finite spin rows (with the horizontal spin parameters $q^{-\mathsf{P}/2}$). We reverse the direction of the straight cylinder operator to match the direction of the particle jumps from $k$ to $k+1$ (opposite from the $q$-Boson case). That is, consider the queue vertex model
\begin{equation}
	\label{eq:qpush_stat_vertex_model_Q}
	\mathfrak{Q}\left( (u_1,\ldots,u_N,u_0); (q^{-\mathsf{P}/2},\ldots,q^{-\mathsf{P}/2},q^{-1/2} ); (v_1,\ldots,v_m );(s_1^{(v)},\ldots,
	s_n^{(v)} )\right)
\end{equation}
with $N+1$ sites on the ring indexed by $j=1,\ldots,N,0$. Let the distinguished auxiliary line with $j=0$ be at the bottom; see \Cref{fig:qpush_stat}. In \eqref{eq:qpush_stat_vertex_model_Q}, the vertex weights at the sites $(-m,j)$, $1\le j\le N$, and at $(-m,0)$ are, respectively,
\begin{equation}
	\label{eq:qpush_stat_vertex_model_Q_weights_2_cases}
	\WQ^{(-m)}_{q^{-\mathsf{P}/2},s_m^{(v)},u_j/v_m}
	\quad\textnormal{and}\quad
	\WQ^{(-m)}_{q^{-1/2},s_m^{(v)},u_0/v_m}.
\end{equation}
By the Yang--Baxter equation for the queue vertex model (\Cref{prop:queue_YBE}), the vertex weights of the straight cylinder Markov operator must be the fused stochastic weights $W_{q^{1/2-\mathsf{P}/2}u_j/u_0,\mathsf{P}, 1}$ from \Cref{sub:fusion_text}. They are given in \Cref{fig:L_horizontal_weights}.

\begin{remark}
	The weights
	in
	\Cref{fig:L_horizontal_weights}
	are matched to transition probabilities
	of a discrete time particle system on the ring as
	$$
	\begin{tikzpicture}[baseline=(current bounding box.center),scale=0.6]
		\draw[white!45!black,line width=1.5pt,->] (-5.2,-1) --++ (.4,2);
		\draw[white!45!black,line width=4pt,->] (-6,0) -- (-4,0);
		\node[left] at (-6,0) {\small $\mathbf{A}$};
		\node[right] at (-4,0) {\small $\mathbf{C}$};
		\node[above] at (-4.8,1) {\small $k$};
		\node[below] at (-5.2,-1) {\small $\ell$};
		\node[left] at (-2,0) {\small $=$};
		\draw[white!45!black,line width=1.5pt,->] (1,0) -- (-1,0);
		\draw[white!45!black,line width=4pt,->] (0,-1) -- (0,1);
		\node[left] at (-1,0) {\small $k$};
		\node[right] at (1,0) {\small $\ell$};
		\node[below] at (0,-1) {\small $\mathbf{A}$};
		\node[above] at (0,1) {\small $\mathbf{C}$};
		\node[above] at (2.2,0) { \phantom{.} };
	\end{tikzpicture}.
	$$
	The picture in the left-hand side represents vertices in \Cref{fig:qpush_stat}. In the right-hand side, the vertical direction corresponds to time, and the states $\mathbf{A},\mathbf{C}$ encode particle configurations at a given site $j \in \left\{ 1,\ldots,N  \right\}$ on the ring.

	On the right, the horizontal arrow points left because after rotating \Cref{fig:qpush_stat} by $90^\circ$ counterclockwise, the sites on the ring are cyclically ordered as $(N,N-1,\ldots,1 )$. Recall that under the $q$-PushTASEP, particles move in the direction of increasing $j$. This direction of the particle motion is opposite to the $q$-Boson situation, cf. the proof of \Cref{prop:limit_to_cqboson}.
\end{remark}

\begin{figure}[htpb]
	\centering
	\includegraphics[width=.75\textwidth]{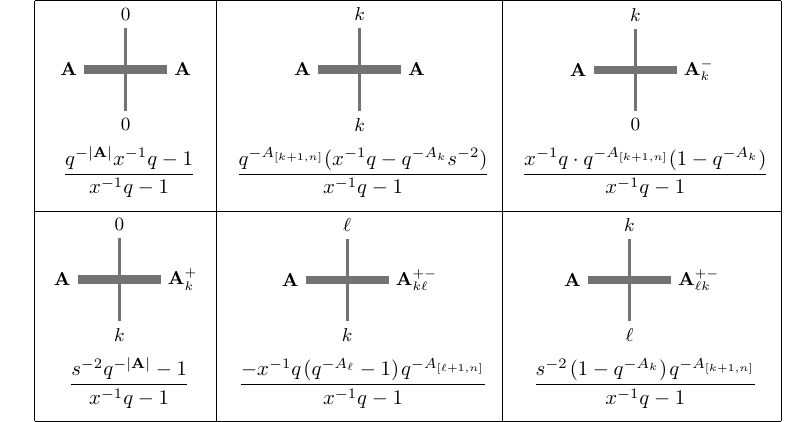}
	\caption{The vertex weights $W_{x,\mathsf{P}, 1}( \mathbf{e}, \mathbf{A}; \mathbf{e}', \mathbf{A}')|_{q^{\mathsf{P}} = s^{-2}}$. Here  $\mathbf{e},\mathbf{e}'$ are basis vectors corresponding to empty or one-particle configurations in $\mathbb{Z}_{\ge0}^n$, and $1\le k<\ell \le n$. Note that these weights can be obtained from the stochastic~$L$ weights (\Cref{fig:L_weights}) by reflecting the picture about the diagonal and setting~$s^2 \rightarrow s^{-2}, q \rightarrow q^{-1}, s x \rightarrow x^{-1} q$.}
	\label{fig:L_horizontal_weights}
\end{figure}

\begin{figure}[htbp]
	\centering
	\includegraphics[width=.75\textwidth]{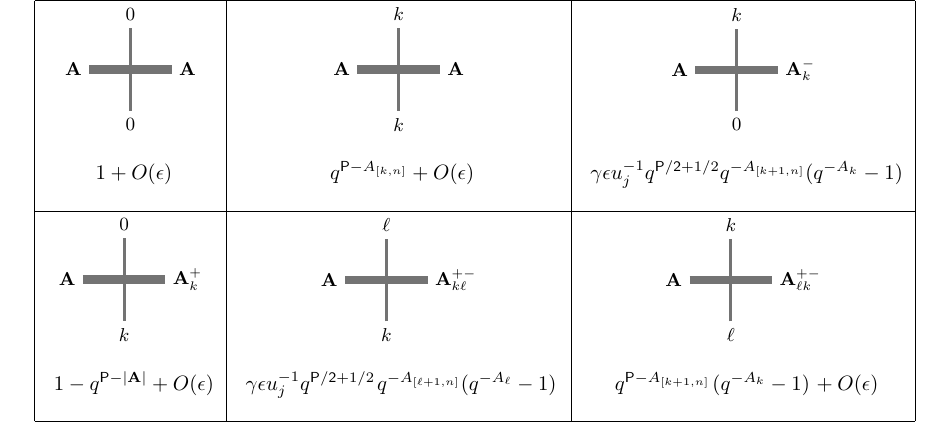}
	\caption{Small~$\epsilon$ expansion of the vertex weights
	$W_{q^{1/2-\mathsf{P}/2}u_j/u_0,\mathsf{P}, 1}$ with
	$u_0=\gamma \epsilon $.}
	\label{fig:hor_small_eps}
\end{figure}

Now let us pass to a Poisson-type continuous-time limit of
the straight cylinder Markov operator to get the continuous
time Markov semigroup of the colored $q$-PushTASEP. Set
\begin{equation}
	\label{eq:qpush_cont_lim_parameters}
	u_0\coloneqq\gamma \epsilon >0,\qquad
	j=1,\ldots,N;
	\qquad
	\gamma\coloneqq q^{\mathsf{P}/2-1/2},
\end{equation}
where $\epsilon>0$ is small. The $\epsilon\to0$ expansions of the vertex weights from \Cref{fig:L_horizontal_weights} are given in \Cref{fig:hor_small_eps}.

\medskip
These expansions imply the convergence as $\epsilon\to0$ of the straight cylinder Markov operators
$\mathfrak{S}(\mathbf{u}, \gamma\epsilon; (q^{-\mathsf{P}/2},\ldots,q^{-\mathsf{P}/2},q^{-1/2} ))^{\lfloor t/\epsilon \rfloor }$
to the $q$-PushTASEP semigroup
$\mathfrak{P}_{\mathrm{qPush}}(t)$, in the same way as for the $q$-Boson process
(\Cref{prop:limit_to_cqboson}). Indeed, the auxiliary spin~$1/2$ line becomes occupied at a given instant in time with probability~$O(\epsilon)$. Then, with high probability it becomes unoccupied within a finite number of discrete time steps, which corresponds to it becoming unoccupied instantaneously with respect to the macroscopic continuous time $t$.

The convergence of the straight cylinder Markov operators to the colored $q$-PushTASEP implies that the stationary distribution of the latter process can be represented as the partition function of a queue vertex model on the cylinder. More precisely, we have the following result:

\begin{proposition}
	\label{prop:qPushTASEP_stationary}
	Let $q\in[0,1)$, $\mathsf{P}\in \mathbb{Z}_{\ge1}$ and $u_1,\ldots,u_N>0 $.
	Fix the type counts $(N_1,\ldots,N_n)$ with $N_i\ge1$ for all $i$.
	For any $v_1,\ldots,v_n$ and
	$s_1^{(v)},\ldots,s_n^{(v)}$,
	the stationary measure of the
	colored $q$-PushTASEP process on the ring has the form
	\begin{equation}
		\label{eq:qPushTASEP_stationary_distribution_result}
		\mathop{\mathrm{Prob}}^{\mathrm{qPush}}\nolimits_{N_1,\ldots,N_n }(\mathbf{V})
		=
		\frac{\langle\emptyset | \ssp
			\mathfrak{Q}
			\bigl(
				 (u_1,\ldots,u_N); (q^{-\mathsf{P}/2},\ldots,q^{-\mathsf{P}/2} ); (v_1,\ldots,v_m );
				 (s_1^{(v)},\ldots,s_n^{(v)} )
			\bigr)
		\ssp | \mathbf{V}\rangle}{Z_{N_1,\ldots,N_n }^{\mathrm{qPush}}}.
	\end{equation}
\end{proposition}
\begin{proof}[Proof outline]
	This is proven in the same way as
	\Cref{prop:qBos_stationary}.
	The queue vertex model
	\eqref{eq:qPushTASEP_stationary_distribution_result}
	on the cylinder
	has the weights
	of the first type in~\eqref{eq:qpush_stat_vertex_model_Q_weights_2_cases}.
	The weights of the second type have the parameter
	$u_0 = \gamma \epsilon$.
	One can check that as $\epsilon\to0$, we have
	\begin{equation*}
		\WQ^{(-m)}_{q^{-1/2},s_m^{(v)},\gamma \epsilon}
		(\mathbf{A},0;\mathbf{A},0)=
		1+O(\epsilon),
		\qquad
		\WQ^{(-m)}_{q^{-1/2},s_m^{(v)},\gamma \epsilon}
		(\mathbf{A},0 ;\mathbf{A}_k^{-},k)=
		O(\epsilon).
	\end{equation*}
	This means that in the limit $\epsilon\to0$, the auxiliary line is unoccupied with probability going to $1$. Thus, the weights of the second type in \eqref{eq:qpush_stat_vertex_model_Q_weights_2_cases} do not contribute to the queue vertex model, and we may pass from the model
	\eqref{eq:qpush_stat_vertex_model_Q} on the cylinder with $N+1$ rows to
	\eqref{eq:qPushTASEP_stationary_distribution_result}
	with $N$ rows, in the same way as in \eqref{eq:going_from_Nplus1_to_N_in_the_model}.
	This completes the proof.
\end{proof}

\begin{figure}
	\centering
	\includegraphics[width=\textwidth]{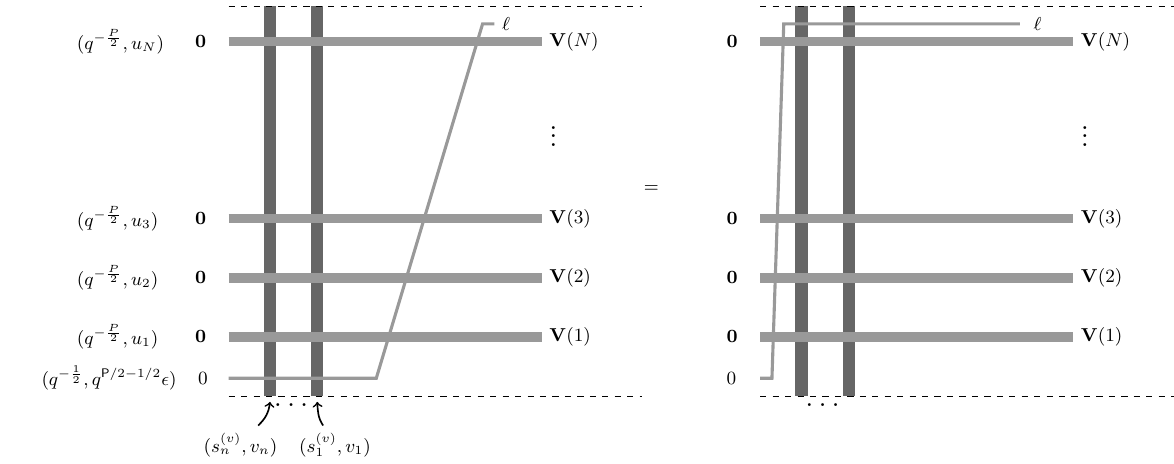}
	\caption{Illustration of the
	$q$-PushTASEP stationarity for finite~$\epsilon$.}
	\label{fig:qpush_stat}
\end{figure}

Let us now discuss the nonnegativity of the individual vertex weights in the queue vertex model in
\eqref{eq:qPushTASEP_stationary_distribution_result}.
Note that the normalized partition functions are positive as
components of the
Perron--Frobenius eigenvector of
$\mathfrak{P}_{\mathrm{qPush}}(t)$.

Define
\begin{equation}
	\label{eq:qPushTASEP_stationary_distribution_positive_weights}
	\WQ^{(-m),\mathrm{qPush}(\mathsf{P})_+}_{s,u}(\mathbf{A},\mathbf{B};\mathbf{C},\mathbf{D})
	\coloneqq
	(-1/s)^{|\mathbf{D}|}\ssp
	\WQ^{(-m)}_{q^{-\mathsf{P}/2},s,u}(\mathbf{A},\mathbf{B};\mathbf{C},\mathbf{D}),
	\qquad
	\mathbf{A},\mathbf{B},\mathbf{C},\mathbf{D}\in \mathbb{Z}_{\ge0}^n.
\end{equation}
Note that $|\mathbf{B}|,|\mathbf{D}|\le \mathsf{P}$
due to the finite-spin reduction
(see \Cref{rmk:finite_spin,rmk:finite_spin_reduction_fully_fused}).
The multiplication by
$(-1/s_m^{(v)})^{|\mathbf{D}|}$
in each column of the queue vertex model
\eqref{eq:qPushTASEP_stationary_distribution_result}
can be absorbed into the normalizing constant, and thus
does not affect the normalized partition functions.
In other words, we can use the weights
\eqref{eq:qPushTASEP_stationary_distribution_positive_weights}
to represent the stationary distribution of the
colored $q$-PushTASEP.

The weights
\eqref{eq:qPushTASEP_stationary_distribution_positive_weights}
arise from the mASEP queue weights
\eqref{eq:ASEP_gauge}
by \emph{fusion}. That is, each
weight
\eqref{eq:qPushTASEP_stationary_distribution_positive_weights}
is a certain sum of $\mathsf{P}$-fold
products of the weights
$\WQ_{q^{-1/2},s,uq^i}^{(-m),\mathrm{mASEP}_+}$,
where $i=0,1,\ldots,\mathsf{P}-1 $.
We refer to
\cite[Appendix~B]{borodin_wheeler2018coloured}
and
\cite[Theorem 8.5]{borodin2019shift}
for details.
This implies the following nonnegativity of
\eqref{eq:qPushTASEP_stationary_distribution_positive_weights}:
\begin{proposition}
	\label{prop:qpush_weights_nonnegative}
	Let
	\begin{equation*}
		0\le s_m^{(v)}< \frac{u_j}{v_m}\ssp q^{\mathsf{P}-1/2}
		<
		\frac{u_j}{v_m}\ssp q^{-1/2}
		\le 1
	\end{equation*}
	for
	all $1\le m\le n$, $1\le j\le N$.
	Then the vertex weights
	$\WQ^{(-m),\mathrm{qPush}(\mathsf{P})_+}_{s_m^{(v)},u_j/v_m}$
	\eqref{eq:qPushTASEP_stationary_distribution_positive_weights}
	(entering the queue vertex model on the cylinder
	representing the stationary distribution of the
	$q$-PushTASEP) are nonnegative.
\end{proposition}
\begin{proof}
	Under the hypotheses, the weights $\WQ_{q^{-1/2},s_m^{(v)},q^i u_j/v_m }^{(-m),\mathrm{mASEP}_+}$, where $i$ runs from $0$ to $\mathsf{P}-1$, are all nonnegative; see \Cref{prop:positivity_parameter_dependence}. Together with fusion, this implies
	the desired nonnegativity of
 	the weights
	$\WQ^{(-m),\mathrm{qPush}(\mathsf{P})_+}_{s_m^{(v)},u_j/v_m}$.
\end{proof}

\begin{remark}
	\label{rmk:ASEP_vs_qPush_stationarity}
	The stationary distribution for the colored $q$-PushTASEP with equal parameters $u_j = u$, $1\le j\le N$, and with $\mathsf{P}=1$, is the same as for the mASEP. Indeed, this follows by matching the vertex weights (see the discussion before \Cref{prop:qpush_weights_nonnegative}). On the other hand, the proofs of the stationarity for mASEP and for the $q$-PushTASEP require different Markov operators on the cylinder (the twisted and the straight ones, respectively).
\end{remark}

\section{Stationarity in the quarter plane and on the line}
\label{sec:quarter_plane}

Here we explain how the queue vertex models on the cylinder
from \Cref{sec:ASEP_matrix_products,sec:qBoson,sec:qPush}
can be used to construct the stationary distributions for mASEP, the colored $q$-Boson and the colored $q$-PushTASEP
on the line $\mathbb{Z}$. Instead of passing to the limit as
the size of the ring goes to infinity (as in, e.g.,
\cite[Section~5]{martin2020stationary}), our proof of the stationarity on
the line passes through applying the Yang--Baxter equation
in the quarter plane, which may be viewed as a \emph{colored
generalization of Burke's theorem} for stochastic vertex models.
Applications of the
latter to single-color stochastic integrable systems were
the subject of, e.g., \cite{OConnellYor2001} (semi-discrete
Brownian polymer), \cite{Seppalainen2012} (log-gamma
polymer).
Particular cases of the colored Burke's theorem
(in the language of queues) appeared previously
in
\cite{FerrariMartin2005},
\cite{ferrari2009multiclass}.

\begin{remark}
	\label{rmk:only_homogeneous_processes}
	We only consider \emph{space-homogeneous} systems on the
	line ($u_j=u$ for all $j$ for the $q$-Boson and the
	$q$-PushTASEP; there are no known space-inhomogeneous
	integrable deformations of the ASEP). In contrast with the ring,
	stationarity of space-inhomogeneous systems on $\mathbb{Z}$
	is much more delicate. If the inhomogeneity is smooth in
	space, we may locally model stationary distributions by the
	homogeneous ones. In the non-smooth case, inhomogeneity in
	the $q$-Boson system may lead to infinite stacks of
	particles, separating the whole system into independent
	components. Out-of-equilibrium single-color inhomogeneous
	models (featuring both smooth and non-smooth inhomogeneity)
	were considered in, e.g., \cite{BorodinPetrov2016Exp},
	\cite{basu2017invariant}, \cite{SaenzKnizelPetrov2018},
	\cite{Petrov2017push}, and we refer to those works for
	further details.
\end{remark}

A special case $\mathsf{P}=1$ of the space-homogeneous colored $q$-PushTASEP is the same as the colored six-vertex model. Therefore, our constructions immediately produce stationary measures for the colored stochastic six-vertex model on the line. Furthermore, taking the Poisson type limit of the latter model along the diagonal (as described on the ring in \Cref{prop:mASEP_from_R}), we get the mASEP on the line. This implies that the stationary measures for the $q$-PushTASEP with $\mathsf{P}=1$ coincide with those of the mASEP. In \Cref{rmk:ASEP_vs_qPush_stationarity} we already noticed this connection on the ring. Therefore, in describing stationary measures on the line, we may restrict attention to the $q$-Boson and the $q$-PushTASEP.

\subsection{Queue steady state}
\label{sub:stationary_queue_regime}

Let $n\ge1$ be the number of colors,
and fix $1\le m\le n$.
Fix parameters $\alpha,\nu$
such that
\begin{equation}
	\label{eq:alpha_nu_nonneg}
	0\le \nu\le \alpha.
\end{equation}
Consider the
queue vertex weights
\begin{equation}
	\label{eq:WZ_queue_weights}
	\WQ^{(-m),\mathrm{line}}_{\alpha,\nu}
	(\mathbf{A},k;\mathbf{B},\ell)
	\coloneqq
	\WQ^{(-m)}_{q^{-1/2},s,z}
	(\mathbf{A},k;\mathbf{B},\ell),
	\qquad
	\alpha=-szq^{-1/2},\qquad \nu=-s^2.
\end{equation}
These weights are given in
\Cref{fig:ASEP_weights} with $z=uq^{1/2}$, and for convenience we
reproduce them with the parameters $\alpha,\nu$ in
\Cref{fig:steady_state_weights}.
The next statement is straightforward
from these expressions and
the sum-to-one property
\eqref{eq:fully_fused_stochastic_weights_sum_to_one}:
\begin{lemma}
	\label{lemma:weights_positive}
	The weights
	\eqref{eq:WZ_queue_weights}
	sum to one over $(\mathbf{B},\ell)$ for any fixed $(\mathbf{A},k)$.
	Moreover, under \eqref{eq:alpha_nu_nonneg}, we have
	$\WQ^{(-m),\mathrm{line}}_{\alpha,\nu}(\mathbf{A},k;\mathbf{B},\ell)\ge0$
	for all $\mathbf{A},\mathbf{B}\in \mathbb{Z}_{\ge0}^n$ and
	$k,\ell\in \left\{ 0,1,\ldots,n  \right\}$.
\end{lemma}

\begin{figure}[htpb]
	\centering
	\includegraphics[width=.8\textwidth]{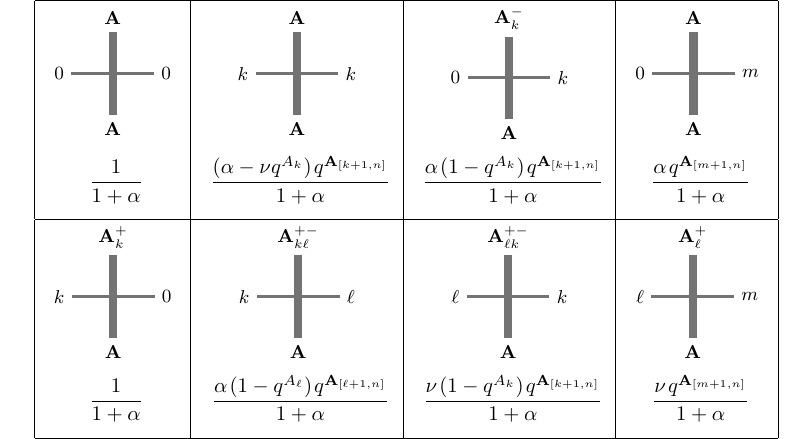}
	\caption{The weights $\WQ^{(-m),\mathrm{line}}_{\alpha,\nu}$
		\eqref{eq:WZ_queue_weights},
	where $m<k<\ell\le n$, and $A_m=+\infty$.}
	\label{fig:steady_state_weights}
\end{figure}

\begin{remark}
	On the ring the stochasticity of the queue vertex weights is not essential, and we multiplied them by $(-1/s)^{\mathbf{1}_{\ell >0}}$ to be nonnegative for $s\ge0$; see \eqref{eq:ASEP_gauge}
	and \Cref{prop:positivity_parameter_dependence}. The factors $(-1/s)^{\mathbf{1}_{\ell >0}}$ were absorbed into the normalizing constant of the stationary distribution on the ring. On the line, this absorption is not possible, so we need to deal with a different range of the parameters $s,z$
	as in \eqref{eq:alpha_nu_nonneg} and \eqref{eq:WZ_queue_weights}.
\end{remark}

Fix parameters
\begin{equation}
	\label{eq:alpha_nu_parameters_for_steady_queue}
	\alpha_1>\alpha_2> \ldots> \alpha_n> 0,\qquad
	\nu_i\in[0,\alpha_i],\quad
	i=1,\ldots,n .
\end{equation}
Our first step is to construct a certain
queue steady state vertex model. Fix $K\in \mathbb{Z}_{\ge1}$ and
consider the
rectangle
\begin{equation}
	\label{eq:rectangle_RK}
	\mathrm{R}_{K}\coloneqq
	\{-n,-n+1,\ldots,-1 \}\times \left\{ -K,-K+1,\ldots,0  \right\}
\end{equation}
(formed by the bottom $K+1$ rows and the left $n$
columns, see \Cref{fig:qp_stationary}, left).
In this rectangle, define a stochastic vertex model
with empty inputs from the bottom and the left, and
vertex weights $\WQ^{(-m),\mathrm{line}}_{\alpha_m,\nu_m}$
at each $(-m,-j)\in \mathrm{R}_K$.
Denote the outgoing arrow configuration
at the top by $\mathbf{M}_K=(\mathbf{M}_K(-n),\ldots, \mathbf{M}_K(-1))$,
where $\mathbf{M}_K(-m)\in \mathbb{Z}_{\ge0}^n$,
and the outgoing arrow configuration on the right by
$\mathbf{d}_K=(d_K(0),d_K(-1),\ldots,d_K(-K) )$,
where $d_K(-j)\in \left\{ 0,1,\ldots,n  \right\}$.

\begin{proposition}
	\label{prop:steady_state_queue}
	Fix arbitrary $c\in \mathbb{Z}_{\ge0}$.
	As $K\to+\infty$, the random tuples
	\begin{equation*}
		\mathbf{M}_K\quad \textnormal{and}
		\quad
		(d_K(0),d_K(-1),
		\ldots,d_K(-c))
	\end{equation*}
	converge
	in joint distribution to random tuples
	$\mathbf{M}$ and $\mathbf{d}^{[-c]}$.
\end{proposition}
We refer to the $K\to+\infty$ limit of the model in $\mathrm{R}_{K}$
as to the \emph{queue vertex model in steady state}.
The limiting tuples $\mathbf{d}^{[-c]}$ are compatible for $c\ge0$.
Let us denote the corresponding infinite tuple by $\mathbf{d}= (d(0),d(-1),d(-2),\ldots )$.

\begin{proof}[Proof of \Cref{prop:steady_state_queue}]
	View the vertical coordinate in $\mathrm{R}_K$ as
	discrete time $\mathrm{t}\in\{-K,-K+1,\ldots,0\}$.
	Observe that
	\begin{equation*}
		\WQ^{(-m),\mathrm{line}}_{\alpha_m,\nu_m}
		(\mathbf{A},k;\mathbf{A}_k^+,0)=
		\WQ^{(-m),\mathrm{line}}_{\alpha_m,\nu_m}
		(\mathbf{A},0;\mathbf{A},0)=
		\frac{1}{1+\alpha_m}.
	\end{equation*}
	Therefore, if we do not distinguish the colors $\ge m$, then the
	arrows leaving the column
	$(-m)$ form a Bernoulli process (in discrete time $\mathrm{t}$)
	with probability of success $\alpha_m/(1+\alpha_m)$.
	This implies that
	the combined number of arrows of colors $m+1,\ldots,n $
	in the column $(-m)$
	evolves as a birth and death Markov chain
	on $\mathbb{Z}_{\ge0}$
	starting from $0$, which makes jumps by $-1$, $0$, and $+1$.
	Denote this chain by
	$A_{[m+1,n]}(\mathrm{t})$.
	The
	jumps by $+1$ and $-1$
	have probabilities, respectively,
	\begin{equation*}
		\frac{\alpha_{m+1}}{1+\alpha_{m+1}}\cdot
		\frac{1+\nu_m q^{A_{[m+1,n]}(\mathrm{t})}}{1+\alpha_m}
		\quad\textnormal{and}
		\quad
		\frac{1}{1+\alpha_{m+1}}\cdot
		\frac{\alpha_m(1- q^{A_{[m+1,n]}(\mathrm{t})})}{1+\alpha_m}.
	\end{equation*}
	The jump by $0$ occurs with the complementary probability.

	Since $\alpha_{m+1}<\alpha_m$, for large $A_{[m+1,n]}(\mathrm{t})$
	the probability to go down is strictly larger, which implies that the
	birth and death chain on $\mathbb{Z}_{\ge0}$ is recurrent.
	Thus, in each column $(-m)$,
	$1\le m\le n$, the
	number of arrows of color $>m$ does not
	grow to infinity.

	We conclude that the (colored)
	configurations of arrows in all columns $(-m)$,
	$1\le m\le n$,
	jointly form a recurrent Markov chain~\cite[Chapter 5]{durrett2019probability} ---
	a system of $n$ queues
	in tandem.
	The limiting random tuple $\mathbf{M}$ is its steady state.
	The limiting configuration $\mathbf{d}=(d(0),d(-1),\dots )$
	is the steady state (colored) departure process, with time running from $-\infty$ to $0$.
	This completes the proof.
\end{proof}

\begin{remark}
	For $\nu_m=0$, $1\le m\le n$, the system of $n$ queues in tandem
	in the proof of \Cref{prop:steady_state_queue}
	was considered
	in \cite[Sections~3.4 and~5]{martin2020stationary}.
\end{remark}

\subsection{Colored Burke's theorem via Yang--Baxter equation}
\label{sub:YBE_quarter_plane}

Let us first discuss the general application of the Yang--Baxter equation in the quarter plane without specifying the weights leading to the concrete model.
We discuss specializations to our colored stochastic particle systems in
\Cref{sub:qboson_and_density_app,sub:qpush_and_density_app} below.
Assume that there exist stochastic, nonnegative vertex weights
\begin{equation}
	\label{eq:abstract_weights_for_qp_stat}
	\WQ^{(-m),\mathrm{queue}}_{\xi,\alpha,\nu}(\mathbf{A},\mathbf{B};\mathbf{C},\mathbf{D})
	\quad\textnormal{and}\quad
	W^{\mathrm{qp}}_{\xi}(k,\mathbf{B};\ell,\mathbf{D})
\end{equation}
which
together with the weights
$\WQ^{(-m),\mathrm{line}}_{\alpha,\nu}(\mathbf{A},k;\mathbf{B},\ell)$
\eqref{eq:alpha_nu_nonneg}
satisfy the Yang--Baxter equation given in \Cref{fig:abstract_YBE}.

\begin{figure}[htpb]
	\centering
	\includegraphics[width=\textwidth]{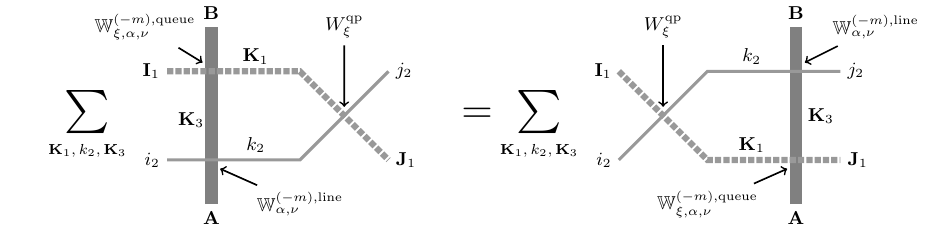}
	\caption{Elementary Yang--Baxter equation for the quarter plane.
	Here $\mathbf{I}_1,\mathbf{A},\mathbf{J}_1,\mathbf{B}\in \mathbb{Z}_{\ge0}^n$
	and $i_2,j_2\in \left\{ 0,1,\ldots,n  \right\}$ are fixed,
	and the sums in both sides are over the internal arrow configurations
	$\mathbf{K}_1,\mathbf{K}_3\in \mathbb{Z}_{\ge0}^n$ and $k_2\in \left\{ 0,1,\ldots,n  \right\}$.}
	\label{fig:abstract_YBE}
\end{figure}

In \Cref{sub:stationary_queue_regime}
above we
used the weights $\WQ^{(-m),\mathrm{line}}_{\alpha_m,\nu_m}$
to
construct the
queue steady state vertex model in $\left\{ -n,\ldots,-1
\right\}\times \mathbb{Z}_{\le 0}$
depending on the parameters
\eqref{eq:alpha_nu_parameters_for_steady_queue}.
As the output, the steady state model produces
the random state $\mathbf{M}=(\mathbf{M}(-n),\ldots,\mathbf{M}(-1))$
at the top, and
the random departure process $\mathbf{d}=(d(0),d(-1),\ldots )$ on the right.
Let us define two more stochastic vertex models
(see the left side of \Cref{fig:qp_stationary}, for an illustration):
\begin{enumerate}[$\bullet$]
	\item A queue vertex model in
		$\left\{ -n,\ldots,1  \right\}\times \left\{ 1,2,\ldots  \right\}$,
		with the weight
		$\WQ^{(-m),\mathrm{queue}}_{\xi,\alpha_m,\nu_m}$
		at each vertex $(-m,j)$.
		This model has
		no incoming arrows from the left, and the
		incoming arrow configuration $\mathbf{M}$
		from the bottom.
		Denote by $\mathbf{V}(j)$, $j\in \mathbb{Z}_{\ge1}$,
		the outgoing arrow configuration from the
		$j$-th horizontal line of this model.
	\item A vertex model in the quadrant
		$\mathbb{Z}_{\ge0}\times \mathbb{Z}_{\ge1}$.
		Let each vertex $(i,j)$ in the quadrant have weight
		$W^{\mathrm{qp}}_{\xi}$.
		Let the incoming
		arrow configurations for this model be
		$\mathbf{V}(1),\mathbf{V}(2),\ldots $
		from the left, and
		$d(0),d(-1),\ldots $
		from the bottom.
\end{enumerate}

Let $(i,j)\in \mathbb{Z}_{\ge0}\times \mathbb{Z}_{\ge1}$. Denote by $\mathbf{V}'(j)\in \mathbb{Z}_{\ge0}^{n}$ the arrow configuration at the horizontal edge $(0,j)-(1,j)$, and by $d'(-i)\in\left\{ 0,1,\ldots,n  \right\}$ the color of the vertical edge $(i,1)-(i,2)$.

\begin{theorem}[Colored Burke's theorem]
	\label{thm:colored_Burke}
	We have equalities of joint distributions:
	\begin{equation*}
		\Bigl\{
			\bigl( d(-i) \bigr)_{i\ge0},
			\bigl( \mathbf{V}(j) \bigr)_{j \ge1}
		\Bigr\}
		\stackrel{d}{=}
		\Bigl\{
			\bigl( d'(-i) \bigr)_{i\ge1},
			\bigl( \mathbf{V}(j) \bigr)_{j \ge1}
		\Bigr\}
		\stackrel{d}{=}
		\Bigl\{
			\bigl( d(-i) \bigr)_{i\ge0},
			\bigl( \mathbf{V}'(j) \bigr)_{j \ge2}
		\Bigr\}.
	\end{equation*}
\end{theorem}
In words, the joint distribution of the horizontal and the vertical arrow configurations along the boundary of an arbitrarily shifted quadrant $\mathbb{Z}_{\ge I}\times \mathbb{Z}_{\ge J+1} $ (where $I,J\ge0$) is the same as for the original quadrant $\mathbb{Z}_{\ge0}\times \mathbb{Z}_{\ge1}$.
Since by our assumption the weights $W^{\mathrm{qp}}_{\xi}$ are stochastic, \Cref{thm:colored_Burke} can be viewed as the statement that the boundary data given by $(\mathbf{V}, \mathbf{d})$ is stationary for the stochastic vertex model in the quadrant with the weights
$W^{\mathrm{qp}}_{\xi}$.
\begin{proof}[Proof of \Cref{thm:colored_Burke}]
	The result follows by repeatedly applying the Yang--Baxter equation from \Cref{fig:abstract_YBE} to the combination of the three vertex models in the left side of \Cref{fig:qp_stationary}.
	Indeed, to shift the index in
	$\bigl(\mathbf{V}(j)\bigr)_{j\ge1}$ up by one, one needs to
	drag the crosses from the right to the left, and move the
	dotted line given by~$\{y=1\}$ down to minus infinity. See the top two
	pictures in the right
    side of \Cref{fig:qp_stationary}. The cross
	vertices on the left boundary are empty, have probability
	weight~$1$, and thus can be removed. In the limit as the
	dotted line goes down to minus infinity, the
	distribution of the outputs $\mathbf{M}$ and $\mathbf{d}$
	of the queue steady state model becomes the same; in fact, the transition matrix for the system of queues in tandem commutes with the transition matrix for the queues obtained from the stochastic vertex weights on the dotted line, and thus, the dotted line preserves~$(\mathbf{M},\mathbf{d})$ if~$(\mathbf{M},\mathbf{d})$ is stationary.
	The index shift in
	$\bigl( d(-i) \bigr)_{i\ge0}$
	is performed similarly, but now we need to move the turning line
	which carried $d(0)$ up to positive infinity (see the bottom picture in the right side of \Cref{fig:qp_stationary}
 	for an illustration). This completes the proof.
\end{proof}

\begin{figure}[p]
	\centering
	\includegraphics[width=.94\textwidth]{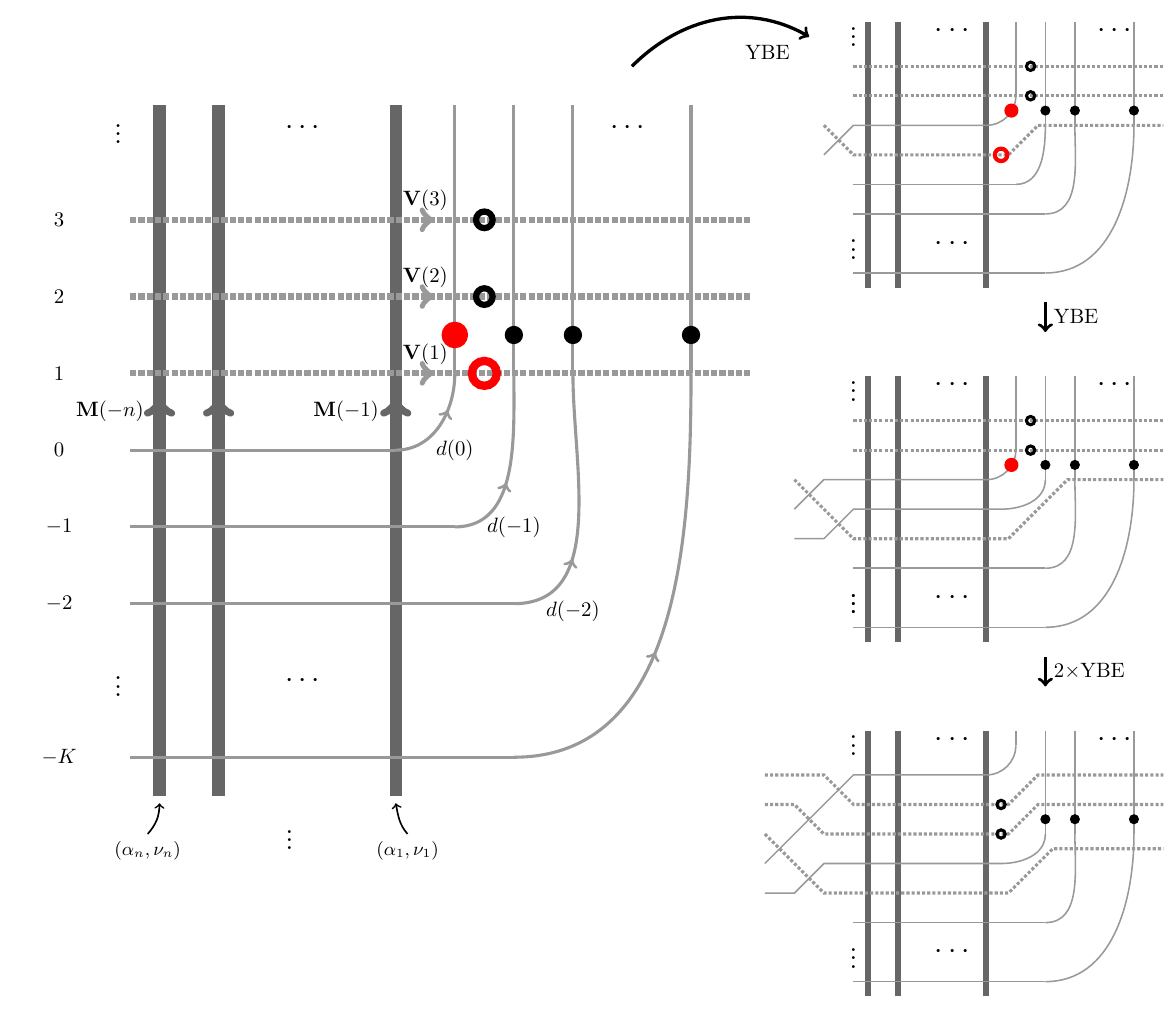}
	\caption{Left: The queue steady state vertex model in
	$\left\{ -n,\ldots,-1  \right\}\times \mathbb{Z}_{\le0}$
	produces the random state
	$\mathbf{M}(-n),\ldots,\mathbf{M}(-1) $ and the
	departure process $\mathbf{d}=(d(0),d(-1),\ldots )$ (see
		\Cref{prop:steady_state_queue} which involves the region $\mathrm{R}_K$
	illustrated in the figure). The queue vertex model has the empty
	incoming configuration from the left. On top of it, we
	put a queue vertex model in $\left\{ -n,\ldots,-1
	\right\}\times \mathbb{Z}_{\ge1}$ (also with no incoming
	arrows from the left; the choice of the queue weights
	depends on whether we work with the colored $q$-Boson or $q$-PushTASEP). Denote
	the random output of this model by
	$\mathbf{V}=(\mathbf{V}(1),\mathbf{V}(2),\ldots )$.
	\medskip\\
	The outputs $\mathbf{V}$ and $\mathbf{d}$ form the left
	and bottom inputs into a third stochastic vertex model in
	the quadrant.
	In a
	limit when the one of the lattice directions turns into the
	continuous time,
	the model in the quadrant converges
	to either the colored $q$-Boson or the colored
	$q$-PushTASEP stochastic particle system.
	\medskip\\
	Right:
	Consecutive applications of the Yang--Baxter equation to the
	vertex models on the left which lead to the shift of the
	quadrant by $(1,1)$. The cross vertices on the left boundary
	are empty and can be removed. The black dots on the edges in all pictures
	represent arrow configurations $d'(-i)$ (solid) and
	$\mathbf{V}'(j)$ (circled), whose joint distributions are
	the same in all pictures. Throughout applying the
	Yang--Baxter equation, the larger (also marked red in the colored version) dots representing $d'(0)$ and $\mathbf{V}'(1)$ disappear,
	which corresponds to shifting of the indices in the formulation of \Cref{thm:colored_Burke}.
	\medskip\\
	Note that topologically, the configuration in the left is the same as in
	\Cref{fig:qp_stationary_intro} from the Introduction, but
	this figure contains more of the essential details.}
	\label{fig:qp_stationary}
\end{figure}

\begin{remark}
	\label{rmk:arbitrary_shifted_xi}
	The stochastic vertex model in the quadrant
	can be made inhomogeneous by
	letting the parameters $\alpha_m$
	and $\xi$ in the queue vertex model part
	depend on the vertical coordinate $j\in \mathbb{Z}$.
	This would lead to the weights $W^{\mathrm{qp}}_{v_i\xi_j}$
	at each vertex $(i,j)\in\mathbb{Z}_{\ge0}\times \mathbb{Z}_{\ge1}$.
	One can readily formulate an extension of
	\Cref{thm:colored_Burke}
	for this situation, but for simplicity we
	only discussed the homogeneous setting.
	See also
	\Cref{rmk:only_homogeneous_processes}
	on stationarity in the presence of
	space inhomogeneity.
\end{remark}

Let us record a property of the steady state
$\mathbf{M}$ which will be useful in \Cref{sec:merging_of_colors} below:
\begin{lemma}
	\label{lemma:steady_state_property}
	Let
	$\mathbf{M}=(\mathbf{M}(-n),\ldots,\mathbf{M}(-1))$
	be the steady state of the $n$-column queue vertex model
	with the weights
	$\WQ^{(-m),\mathrm{line}}_{\alpha_m,\nu_m}$ (the bottom $n$ columns in
	\Cref{fig:qp_stationary}, left). Take a
	one-row vertex model
	in $\{-n,\ldots,-1\}\times \{0\}$ with the weight
	$\WQ^{(-m),\mathrm{queue}}_{\xi,\alpha_m,\nu_m}$ at
	each vertex $(-m,0)$.
	If there are no arrows incoming from the left, and the configuration
	$\mathbf{M}$ of arrows is incoming from below into this one-row model, then the
	distribution of the outgoing arrows from the top
	is the same as that of $\mathbf{M}$.

	In short, the distribution of $\mathbf{M}$ is preserved by the horizontal
	action of the $n$-column queue vertex model with the weights
	$\WQ^{(-m),\mathrm{queue}}_{\xi,\alpha_m,\nu_m}$.
\end{lemma}
\begin{proof}
	Consider the rectangle $\mathrm{R}_K$ \eqref{eq:rectangle_RK},
	see \Cref{fig:qp_stationary}, left.
	Put
	a single $n$-column horizontal layer with the weights
	$\WQ^{(-m),\mathrm{queue}}_{\xi,\alpha_m,\nu_m}$
	at the horizontal coordinate $0$.
	Below it, let us put
	$K$ layers of the
	weights $\WQ^{(-m),\mathrm{line}}_{\alpha_m,\nu_m}$.
	In this way, we obtain the right (straight lattice) part of the
	configuration in \Cref{fig:M_is_stationary}, left.

	Assume that the incoming arrow configuration at the left boundary
	is empty, while at the bottom
	let the input be the random steady state
	$\mathbf{M}=(\mathbf{M}(-n),\ldots,\mathbf{M}(-1))$.
	Denote the random output at the top by
	$\mathbf{M}'=(\mathbf{M}'(-n),\ldots,\mathbf{M}'(-1))$.
	See \Cref{fig:M_is_stationary}, left, for an illustration.
	Since the solid horizontal lines preserve the distribution of $\mathbf{M}$
	(this is the steady state property),
	it remains to show that $\mathbf{M}'$ and $\mathbf{M}$ have the same distribution.

	Attach $K$ cross vertices
	with the stochastic weights
	$W^{\mathrm{qp}}_{\xi}$
	to the rectangle on the right.
	This does not change the distribution of $\mathbf{M}'$.
	Then apply the Yang--Baxter equation
	(as in \Cref{fig:abstract_YBE})
	to move these cross vertices to the left.
	On the left, there are no incoming arrows, so the cross vertices
	in \Cref{fig:M_is_stationary}, right,
	can be removed.

	Denote by $\mathbf{M}''$ the random configuration which
	arises from $\mathbf{M}$ after the single dashed line,
	that is, the line with the weights
	$\WQ^{(-m),\mathrm{queue}}_{\xi,\alpha_m,\nu_m}$. The
	Yang--Baxter equation implies that $\mathbf{M}'$
	on the left and on the right in \Cref{fig:M_is_stationary}
	have the same distribution. However, we now can take the limit
	as $K\to+\infty$ on the right. By the steady state property,
	this limit forgets the distribution
	of $\mathbf{M}''$, and thus $\mathbf{M}'$ would converge
	to the steady state $\mathbf{M}$.
	Note that on the other hand, for any $K$, the distribution
	of $\mathbf{M}'$ is the same. Therefore,
	$\mathbf{M}'$ and $\mathbf{M}$ have the same distribution, as desired.
\end{proof}
\begin{figure}[htpb]
	\centering
	\includegraphics[width=.9\textwidth]{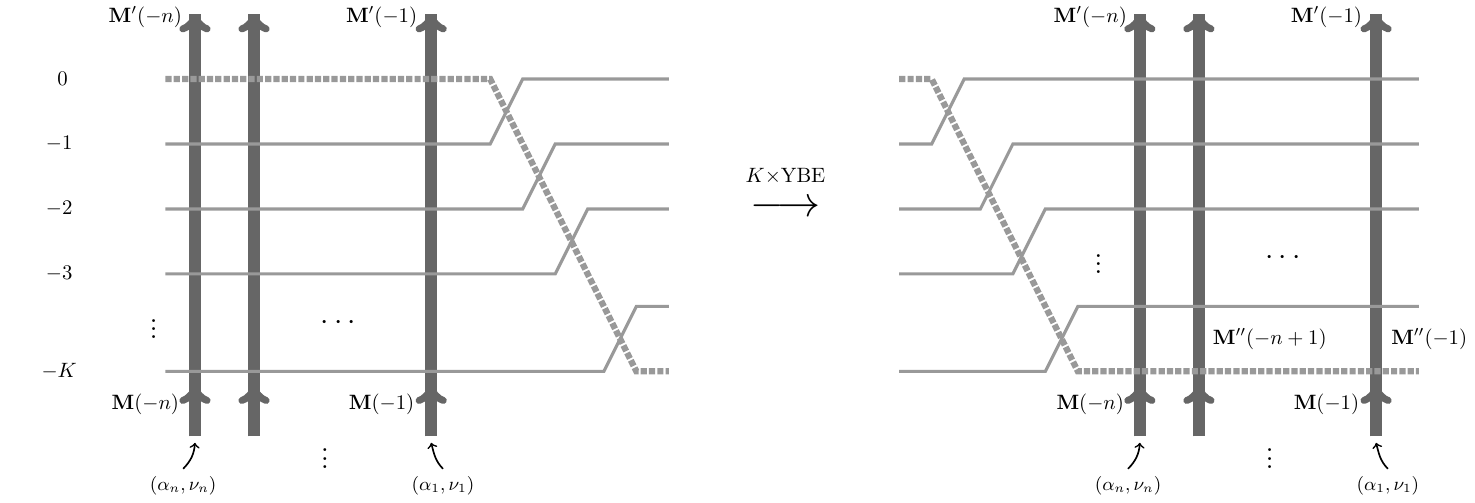}
	\caption{Application of the Yang--Baxter equation in the proof of \Cref{lemma:steady_state_property}.
	The dashed and the solid horizontal lines correspond to the vertex weights
	$\WQ^{(-m),\mathrm{queue}}_{\xi,\alpha_m,\nu_m}$
	and
	$\WQ^{(-m),\mathrm{line}}_{\alpha_m,\nu_m}$, respectively.}
	\label{fig:M_is_stationary}
\end{figure}

From \Cref{thm:colored_Burke}, we can construct a stationary random configuration of the vertex model $W^{\mathrm{qp}}_{\xi}$ in the whole plane $\mathbb{Z}^{2}$. The model in the plane is characterized as follows (a similar single-color construction appeared in \cite[Section~A.2]{Amol2016Stationary}). For any $J\ge 1$, define the map
\begin{equation*}
	\tau_J\colon\mathbb{Z}_{\ge 0}\times \mathbb{Z}_{\ge 1}\to
	\mathbb{Z}^{2},\qquad
	\tau_J\colon (i,j)\mapsto (i-J,j-J).
\end{equation*}
By \Cref{thm:colored_Burke}, the random configurations of
colored paths in parts of $\mathbb{Z}^2$ coming from
shifting the configuration in~$\mathbb{Z}_{\geq 0} \times
\mathbb{Z}_{\geq 1}$ by $\tau_J$, $J\in \mathbb{Z}_{\ge1}$,
are compatible for different $J$ (that is, the restriction
of the shifted configuration to the
original subset has the same distribution as the unshifted configuration).
Therefore, by the Kolmogorov
extension theorem, there exists a distribution on path
configurations in the whole plane $\mathbb{Z}^{2}$, which is
translation invariant (that is, stationary under the
stochastic vertex model $W^{\mathrm{qp}}_{\xi}$).

The parameters $(\alpha_m,\nu_m)$ of the queue vertex models attached to the quadrant determine the densities of various types of colors (the ``colored slope'') of the resulting translation invariant model in $\mathbb{Z}^{2}$. We explore the exact connection between these parameters and the densities of various colors for $q$-Boson and $q$-PushTASEP in \Cref{sub:qboson_and_density_app,sub:qpush_and_density_app} below.

\subsection{Specialization to stochastic six-vertex model and $q$-PushTASEP}
\label{sub:qpush_and_density_app}

Let us specialize \Cref{thm:colored_Burke} to the colored $q$-PushTASEP (defined on the line in the same way as in \Cref{def:qPushTASEP}, but with the homogeneous parameters $u_k= u$, $k\in \mathbb{Z}$). Take the weights
\eqref{eq:WZ_queue_weights}, \eqref{eq:abstract_weights_for_qp_stat} to be
\begin{equation}
	\label{eq:qPushTASEP_weights_line_steady}
	\begin{split}
		\WQ^{(-m),\mathrm{line}}_{\alpha_m,\nu_m}
		(\mathbf{A},k;\mathbf{B},\ell)
		&=
		\WQ^{(-m)}_{q^{-1/2},s_m^{(v)},z/v_m}
		(\mathbf{A},k;\mathbf{B},\ell),
		\\
		\WQ^{(-m),\mathrm{queue}}_{\xi,\alpha_m,\nu_m}
		(\mathbf{A},\mathbf{B};\mathbf{C},\mathbf{D})
		&=
		\WQ^{(-m)}_{q^{-\mathsf{P}/2},s_m^{(v)},u/v_m}
		(\mathbf{A},\mathbf{B};\mathbf{C},\mathbf{D})
		,
		\\
		W^{\mathrm{qp}}_{\xi}(k,\mathbf{B};\ell,\mathbf{D})
		&=
		W_{q^{1/2-\mathsf{P}/2}u/z,\mathsf{P},1}
		( \mathbf{e}_k\mathbf{1}_{k\ge1}
			,\mathbf{B};
			\mathbf{e}_\ell\mathbf{1}_{\ell\ge1},\mathbf{D}),
		\\
		\alpha_m\coloneqq -s_m^{(v)}q^{-1/2}
		\ssp\frac{z}{v_m},
		\qquad
		\nu_m &\coloneqq -(s_m^{(v)})^2,
		\qquad
		\xi\coloneqq u/z.
	\end{split}
\end{equation}
Here $\mathsf{P}\in \mathbb{Z}_{\ge1}$, and
$\alpha_m,\nu_m$
must satisfy \eqref{eq:alpha_nu_parameters_for_steady_queue}.
The weights in the right-hand sides of \eqref{eq:qPushTASEP_weights_line_steady}
are given, respectively, in \Cref{fig:ASEP_weights},
formula \eqref{eq:queue_spec_fully_fused}, and
\Cref{fig:L_horizontal_weights}.

One can check that the weights
\eqref{eq:qPushTASEP_weights_line_steady}
are all nonnegative if
\begin{equation}
	\label{eq:qPushTASEP_steady_state_parameters_nonnegativity}
	\nu_i\in[0,\alpha_i q^{\mathsf{P}-1}],\quad i=1,\ldots,n,
	\qquad
	\xi\ge q^{1/2-\mathsf{P}/2}.
\end{equation}
Indeed, the first condition corresponds to the fact that the
weights $\WQ^{(-m),\mathrm{queue}}_{\xi,\alpha_m,\nu_m}$
come from the fusion of
$\WQ^{(-m),\mathrm{line}}_{\alpha_m,\nu_m}$ (see
\Cref{prop:qpush_weights_nonnegative} for a related
nonnegativity property), and the second condition involving
the weights $W^{\mathrm{qp}}_{\xi}$ is read off from their
explicit form given in \Cref{fig:L_horizontal_weights}.

\medskip

Let us first consider the case $\mathsf{P}=1$.
Then the vertex model in the quadrant with
the weights $W^{\mathrm{qp}}_{\xi}$ becomes the colored
stochastic six-vertex model (see
\cite[Figure~1]{borodin_wheeler2018coloured} for a
simulation with non-stationary boundary conditions).
\Cref{thm:colored_Burke} and the shifting argument after it
allows to construct a translation invariant (stationary)
colored stochastic six-vertex model in the full plane
$\mathbb{Z}^2$.
By analogy with
\cite{Neergard1995},
\cite{aggarwal2020nonexistence},
let us call this path configuration in $\mathbb{Z}^2$
the \emph{colored KPZ pure phase} of the stochastic six-vertex model.
The colored KPZ pure phase has a finite number
$n$ of colors.

From the weights in
\Cref{fig:steady_state_weights}
and the matching of the parameters
\eqref{eq:qPushTASEP_weights_line_steady},
we see that the
probability that no paths leave column $(-m)$ under
the stochastic weights
$\WQ^{(-m),\mathrm{queue}}_{\xi,\alpha_m,\nu_m}$
and
$\WQ^{(-m),\mathrm{line}}_{\alpha_m,\nu_m}$,
respectively, is equal to
\begin{equation}
	\label{eq:density_computation_simplest_part}
	\frac{1}{1+\alpha_m\xi q^{-1/2}}
	\quad\textnormal{and} \quad
	\frac{1}{1+\alpha_m q^{-1/2}}.
\end{equation}
By \Cref{prop:steady_state_queue},
the arrow configurations in the columns
$-n,\ldots,-1 $
are in steady state. This means that the
number of colors $i>m$ in the each column $(-m)$
does not grow to infinity. Thus, once a path
of color $m$ originates at the column $(-m)$,
it must exit the column $(-1)$
at a bounded (random) distance from where it originated.
This means that the combined density of
paths of colors $\ge m$ exiting the column
$(-1)$ (in either the bottom or the
top part of the configuration of $n$ vertical columns,
see \Cref{fig:qp_stationary}, left) is equal to the
complementary probability
to \eqref{eq:density_computation_simplest_part}.

That is, define
$\rho_m^{(h)}$ (respectively, $\rho_m^{(v)}$)
to be the probability at a horizontal
(respectively, vertical)
edge
contains an arrow of color $m$. In the stationary translation
invariant regime, it does not matter which edge we consider. By
ergodicity, the quantities $\rho_m^{(h)}$ and $\rho_m^{(v)}$ also represent the densities of colors $m$ in the horizontal and vertical directions, respectively.

We conclude that
in the colored KPZ pure phase determined by the parameters
$\{\alpha_m\}_{m=1}^n$,
the
quantities $\rho_m^{(h)}$ and $\rho_m^{(v)}$
are given by
\begin{equation}
	\label{eq:rho_h_v_for_colored_s6v}
	\rho_m^{(h)}=
	\frac{\alpha_m\xi q^{-1/2}}{1+\alpha_m\xi q^{-1/2}}-
	\frac{\alpha_{m+1}\xi q^{-1/2}}{1+\alpha_{m+1}\xi q^{-1/2}}
	,
	\qquad
	\rho_m^{(v)}=
	\frac{\alpha_m q^{-1/2}}{1+\alpha_m q^{-1/2}}-
	\frac{\alpha_{m+1} q^{-1/2}}{1+\alpha_{m+1} q^{-1/2}}
	.
\end{equation}
Here $\alpha_{n+1}=0$, by agreement.
When $n=1$, we can solve for $\alpha_1$, and get
\begin{equation*}
	\rho_1^{(h)}
	=
	\frac{\rho_1^{(v)} \xi}{1 +(\xi-1) \rho_1^{(v)} },
\end{equation*}
which agrees with the slope relation $\rho_1^{(h)}=\varphi( \rho_1^{(v)} )$ in the single-color KPZ phase (for example, see \cite[(2.6)]{aggarwal2020nonexistence}). For general $n$,
solving for the $\alpha_i$'s in
\eqref{eq:rho_h_v_for_colored_s6v} yields the following \emph{colored slope relations}:
\begin{equation}
	\label{eq:colored_KPZ_phase_densities}
	\rho_m^{(h)}
	=
	\frac{\rho_m^{(v)}\xi}
	{
		\bigl( 1+(\xi-1)\rho^{(v)}_{[m,n]}   \bigr)
		\bigl( 1+(\xi-1)\rho^{(v)}_{[m+1,n]} \bigr)
	},
	\qquad m=1,\ldots,n,
\end{equation}
where $\rho^{(v)}_{[a,b]}=\rho^{(v)}_a+
\rho^{(v)}_{a+1}+\ldots+\rho^{(v)}_b $.

Under the colored KPZ pure phase, the
colors occupying the vertical edges
along a given horizontal line
induce a
random configuration of colors on $\mathbb{Z}$.
This random configuration is a stationary
distribution for the mASEP.

\begin{remark}
	\label{rmk:uniq_ASEP}
	As shown in
	\cite{ferrari1991microscopic},
	for given color densities, a translation invariant
	stationary distribution for mASEP on $\mathbb{Z}$ is unique.
	We believe that a similar uniqueness holds for the colored stochastic six-vertex model, but this statement does not seem to be present in the existing literature.
\end{remark}

Let us now return to the case of general $\mathsf{P}$, and
take a continuous-time limit to the colored $q$-PushTASEP.
The limit is achieved by setting
$z=q^{\mathsf{P}/2-1/2}\epsilon$, and letting
$\epsilon\to0$.
In this limit, the horizontal direction in the quarter plane, scaled by $1/\epsilon$, turns into
continuous time $t \in \mathbb{R}_{\ge0}$.
See \Cref{fig:hor_small_eps} for the $\epsilon\to0$
expansions of the weights $W^{\mathrm{qp}}_{\xi}$.
The
continuous-time Markov process coming from
$W^{\mathrm{qp}}_{\xi}$ lives on configurations on
$\mathbb{Z}_{\ge1}$ (where at each site there
are at most $\mathsf{P}$ particles), and coincides with the
colored $q$-PushTASEP (\Cref{def:qPushTASEP}) with
homogeneous parameters $u_k\equiv u$.

After the rescaling, let us further set $\nu_m=0$,
which implies nonnegativity of the
remaining weights
$\WQ^{(-m),\mathrm{queue}}_{\xi,\alpha_m,0}$
and
$\WQ^{(-m),\mathrm{line}}_{\alpha_m,0}$.
Indeed, note that as $\alpha_m$ (containing $z$ as a factor)
goes to zero,
conditions \eqref{eq:qPushTASEP_steady_state_parameters_nonnegativity}
cannot hold unless $\nu_m=0$.

\begin{remark}
	\label{rmk:qpush_line_nonnegativity}
	In the constructions on the line, we need to set $\nu_m=0$, $m=1,\ldots,n$, from the beginning, to ensure the nonnegativity of the queue vertex weights and the corresponding jump rates
	in the queue columns (occurring as $\epsilon\to0$).
	The nonnegativity is required to ensure the existence of the steady state in the queue columns (see \Cref{prop:steady_state_queue}).

	This should be contrasted to the ring case, where we could initially work with negative probabilities and jump rates formally. Then, when the commutation relation
	between the queue vertex model transfer matrix and the straight cylinder
	transfer matrix is established, we can renormalize the queue vertex model
	on the cylinder to get nonnegative probabilities under the stationary distribution.
	Thus, we have a whole family of vertex models on the cylinder
	(depending on the $\nu_m$'s) which produce the same stationary measure,
	and on the line we have to set $\nu_m=0$ for all $m$.
\end{remark}

In the remaining vertex models in
\Cref{fig:qp_stationary}, the weights
$\WQ^{(-m),\mathrm{queue}}_{\xi,\alpha_m,0}$
in the top part
do not depend on $z$ (see \eqref{eq:qPushTASEP_weights_line_steady}) and thus do not change in the limit.
In the weights
$\WQ^{(-m),\mathrm{line}}_{\alpha_m,0}$ in the
bottom part (given in \Cref{fig:steady_state_weights}),
we have
\begin{equation}
	\label{eq:definition_of_y_m_for_qpushTASEP}
	\alpha_m=
	-\frac{s_m^{(v)}}{v_m}
	q^{\mathsf{P}/2-1}\epsilon
	\eqqcolon y_m \epsilon\to0.
\end{equation}
Here $y_1> \ldots > y_n>0$
are the new parameters of the continuous-time colored
$q$-PushTASEP.
The fact that the $\alpha_m$'s are proportional to $\epsilon$ corresponds to the scaling of the
bottom columns
$\{-n,\ldots,-1 \}\times \mathbb{Z}_{\le 0}$
in \Cref{fig:qp_stationary}
to continuous ones,
$\{-n,\ldots,-1 \}\times \mathbb{R}_{\le 0}$.
Note that this scaling does not affect the weights
$\WQ^{(-m),\mathrm{queue}}_{\xi,\alpha_m,\nu_m}$ since they
do not depend on $z$.
The resulting scaled queue steady state model in
$\{-n,\ldots,-1 \}\times \mathbb{R}_{\le 0}$
runs in continuous time.
Let
$\mathbf{M}=(\mathbf{M}(-n),\ldots,\mathbf{M}(-1) )$
be the steady state of these continuous-time
queues in tandem, and let
$\mathbf{d}(t)$, $t\le 0$,
be the (continuous-time) departure process.
Using $\mathbf{d}(t)$ and the output $\mathbf{V}$
of the top columns
$\{-n,\ldots,-1 \}\times \mathbb{Z}_{\ge1}$,
one can use the
Burke's theorem (\Cref{thm:colored_Burke}) to
construct
a stationary version of the colored $q$-PushTASEP on the
whole line.

Let us compute the
densities of the colors under the stationary measure
for the colored $q$-PushTASEP.
Since more than one arrow may leave the column $(-m)$,
we need to take the expectation of the number of arrows.
For this expectation, we do not need to distinguish the colors.
Employing the color merging discussed in
\Cref{sec:merging_of_colors},
we may assume that $m=n=1$.
By a specialization of \eqref{eq:queue_spec_fully_fused},
one can check that the weights have the form
(with $a=c=\infty$):
\begin{equation*}
	\begin{split}
		\WQ^{(-1),\mathrm{queue}}_{\xi,\alpha_m,0}
		(a,b;c,d)&=
		\frac{\mathbf{1}_{a+b=c+d}}{(q^{-\mathsf{P}/2}s_m^{(v)}u/v_m;q)_{\mathsf{P}}}
		\ssp
		\frac{(q^{\mathsf{P}/2}s_m^{(v)}u/v_m)^d \ssp (q^{-\mathsf{P}};q)_d}{(q;q)_d}
		\\&=
		\frac{\mathbf{1}_{a+b=c+d}}{(-q^{1-\mathsf{P}}u y_m;q)_{\mathsf{P}}}
		\ssp
		\frac{(-qu y_m)^d \ssp (q^{-\mathsf{P}};q)_d}{(q;q)_d}
		.
	\end{split}
\end{equation*}
Note that these weights sum to $1$ over $0\le d\le \mathsf{P}$ by the $q$-binomial theorem
\cite[(1.3.2)]{GasperRahman}.
The expected number of arrows is expressed through
the function
\begin{equation}
	\label{eq:q_digamma}
	\phi(\zeta)\coloneqq
	\sum_{k=0}^{\infty} \frac{\zeta q^{k}}{1-\zeta q^{k}}
\end{equation}
(up to a change of variables and a linear transformation,
this is the $q$-digamma function).
We have
\begin{equation*}
	\begin{split}
		\sum_{d=0}^{\mathsf{P}}\frac{\zeta^d(q^{-\mathsf{P}};q)_d}{(q;q)_d}&=
		\frac{(q^{-\mathsf{P}}\zeta;q)_{\infty}}{(\zeta;q)_{\infty}}=
		(q^{-\mathsf{P}}\zeta;q)_{\mathsf{P}},
		\\
		\frac{1}{(q^{-\mathsf{P}}\zeta;q)_{\mathsf{P}}}
		\sum_{d=0}^{\mathsf{P}}
		d\cdot
		\frac{\zeta^d(q^{-\mathsf{P}};q)_d}{(q;q)_d}&=
		\zeta\frac{\partial}{\partial \zeta}\log (q^{-\mathsf{P}}\zeta;q)_{\mathsf{P}}
		=-\sum_{i=0}^{\mathsf{P}-1}
		\frac{q^{i-\mathsf{P}}\zeta}{1-q^{i-\mathsf{P}}\zeta},
	\end{split}
\end{equation*}
and the latter sum is a difference of two functions of the form \eqref{eq:q_digamma} with the arguments differing by the factor $q^{-\mathsf{P}}$.
Therefore,
\begin{equation}
	\label{eq:qpush_qp_expectation}
	\sum_{d=0}^{\mathsf{P}}
	d\cdot
	\frac{1}{(-q^{1-\mathsf{P}}uy_m;q)_{\mathsf{P}}}
	\ssp
	\frac{(-qu y_m)^d \ssp (q^{-\mathsf{P}};q)_d}{(q;q)_d}
	=
	\phi(-qu y_m)-
	\phi(-q^{1-\mathsf{P}}u y_m).
\end{equation}
Thus, the horizontal density of the $m$-th color is the difference
of the above expressions:
\begin{equation}
	\label{eq:h_current_qpush}
	\rho_m^{(h)}=
	\phi(-qu y_m)-
	\phi(-q^{1-\mathsf{P}}u y_m)-
	\phi(-qu y_{m+1})+
	\phi(-q^{1-\mathsf{P}}u y_{m+1})
	,
\end{equation}
where, by agreement, $y_{n+1}=0$.
This follows similarly to the case of the colored
six-vertex model: since the queues are in steady state,
a color $m$ cannot accummulate in any column except $(-m)$.
This implies that the expectation \eqref{eq:qpush_qp_expectation}
is equal to
$\rho^{(h)}_{[m,n]}=\rho^{(h)}_{m}+\ldots+\rho^{(h)}_{n}$,
which yields
\eqref{eq:h_current_qpush}.

We can also compute the currents of the colored
$q$-PushTASEP in stationarity (that is, the vertical densities
of the colors), using
\begin{equation*}
	\WQ^{(-m),\mathrm{line}}_{\alpha_m,0}
	(\mathbf{A},k;\mathbf{B},0)=
	\frac{1}{1+\alpha_m}=
	1-y_m \epsilon+O(\epsilon^2),
\end{equation*}
where $\mathbf{B}=\mathbf{A}+\mathbf{e}_k\mathbf{1}_{k\ge1}$.
Therefore, in the continuous-time limit, we have
\begin{equation}
	\label{eq:v_current_qpush}
	\rho_m^{(v)}=
	y_{m} - y_{m+1}.
\end{equation}
Expressing the colored currents
$(\rho_1^{(v)}, \ldots, \rho_n^{(v)})$ in terms of
the colored densities
$(\rho_1^{(h)}, \ldots, \rho_n^{(h)})$
for general $\mathsf{P}$ would require finding the
$y_m$'s from \eqref{eq:h_current_qpush},
which is not explicit for general $\mathsf{P}$.
However, a reverse expression is essentially given by
\eqref{eq:h_current_qpush}--\eqref{eq:v_current_qpush}:
\begin{equation*}
	\rho_m^{(h)}=
	\phi\bigl(-qu \rho^{(v)}_{[m,n]}\bigr)-
	\phi\bigl(-q^{1-\mathsf{P}}u \rho^{(v)}_{[m,n]}\bigr)-
	\phi\bigl(-qu \rho^{(v)}_{[m+1,n]}\bigr)+
	\phi\bigl(-q^{1-\mathsf{P}}u \rho^{(v)}_{[m+1,n]}\bigr)
	.
\end{equation*}

\begin{remark}
	\label{rmk:uniq_push}
	We believe that for any $\mathsf{P}$,
	a translation invariant stationary distribution for the
	$q$-PushTASEP on $\mathbb{Z}$ with parameter $\mathsf{P}$ and with given
	densities of the colors is unique.
	However,
	this statement does not seem to be present in the existing literature.
\end{remark}

\subsection{Specialization to $q$-Boson}
\label{sub:qboson_and_density_app}

Let us specialize \Cref{thm:colored_Burke} to the stochastic colored $q$-Boson process. It is defined the same way on the line as on the ring (\Cref{def:cqBoson}), but we reverse the direction of the particle movement. That is, a particle of color $i$ jumps from $k$ to $k+1$, $k \in \mathbb{Z}$,
at the homogeneous rate
$u^{-1}(1-q^{\mathbf{V}(k)_i})\ssp q^{\mathbf{V}(k)_{[i+1,n]}}$, where $\mathbf{V}(k)\in \mathbb{Z}_{\ge0}^n$ is the arrow configuration at site $k$.

To obtain the $q$-Boson process together with its stationary measure from the vertex models in \Cref{fig:qp_stationary}, is it convenient to take horizontally fused weights (meaning multiple paths can occupy a horizontal edge, as compared to \Cref{sub:qpush_and_density_app}, when~$\WQ^{(-m),\mathrm{line}}$ were weights for which at most one path can occupy a horizontal edge) in the bottom part $\{-n,\ldots,-1 \}\times \mathbb{Z}_{\le 0}$. That is, let us take the following pre-limit weights depending
on $\epsilon>0$:
\begin{equation}
	\label{eq:qBoson_weights_line_steady}
	\begin{split}
		\WQ^{(-m),\mathrm{line}}_{\xi,\alpha_m,\nu_m}
		(\mathbf{A},\mathbf{B};\mathbf{C},\mathbf{D})
		&=
		\WQ^{(-m)}_{\epsilon,s,u \epsilon y_m/s}
		(\mathbf{A},\mathbf{B};\mathbf{C},\mathbf{D})
		,
		\\
		\WQ^{(-m),\mathrm{queue}}_{\alpha_m,\nu_m}
		(\mathbf{A},k;\mathbf{B},\ell)
		&=
		\WQ^{(-m)}_{q^{-1/2},s,-\epsilon q^{1/2}y_m/s}
		(\mathbf{A},k;\mathbf{B},\ell)
		,
		\\
		W^{\mathrm{qp}}_{\xi,\epsilon}
		(\mathbf{A},k;\mathbf{B},\ell)
		&=
		W_{-1 /(u \epsilon),1,\mathsf{N}}
		(\mathbf{A},\mathbf{e}_k\mathbf{1}_{k\ge1};\mathbf{B},\mathbf{e}_{\ell}
		\mathbf{1}_{\ell\ge1})
		,
		\\
		\alpha_m
		\coloneqq
		\epsilon
		y_m/s
		,
		\qquad
		\nu_m &\coloneqq
		-s^2
		,
		\qquad
		\xi\coloneqq u,
		\qquad
		q^{-\mathsf{N}^\epsilon} \coloneqq \epsilon^2
		.
	\end{split}
\end{equation}
Note that here we placed the $\xi$-dependence into the bottom
part of
the left $n$ columns in
\Cref{fig:qp_stationary}, left, instead of the top one.

Let us take $\epsilon\to0$ and simultaneously rescale the
vertical coordinate of the quadrant by $1/\epsilon$.
This turns the vertical coordinate into the continuous time $t\in \mathbb{R}_{\ge0}$.
After that, set $s=0$, which would imply the nonnegativity of the jump rates
(the restriction $s=0$ is parallel to the case of the
$q$-PushTASEP; see \Cref{rmk:qpush_line_nonnegativity}). The results of \Cref{sub:qBoson_process,sub:qBoson_stationary}
imply that as $\epsilon\to0$ and
$s=0$,
the weights
\eqref{eq:qBoson_weights_line_steady} become
\begin{equation}
	\label{eq:qBoson_weights_line_steady_limiting}
	\begin{split}
		\WQ^{(-m),\mathrm{line}}_{\xi,\alpha_m,\nu_m}
		&\to
		(uy_m;q)_\infty\ssp
		\WQ^{(-m),\mathrm{qBos}}_{0,uy_m}
		,
		\\
		\WQ^{(-m),\mathrm{queue}}_{\alpha_m,\nu_m}
		&\to
		\textnormal{jump rates in a continuous-time
		queue vertex model (\Cref{fig:qboson_cont_queue_model_s7})}
		,
		\\
		W^{\mathrm{qp}}_{\xi,\epsilon}
		&\to
		\textnormal{colored $q$-Boson jump rates in
		\Cref{fig:qboson_specialization}},
	\end{split}
\end{equation}
where $\WQ^{(-m),\mathrm{qBos}}_{0,uy_m}$
is given by the right-hand side of
\eqref{eq:q_boson_weights_degeneration_vertex},
and the
limits of
$\WQ^{(-m),\mathrm{queue}}_{\alpha_m,\nu_m}$
to the continuous vertical direction
are read off from
\Cref{fig:ASEP_weights}.
These limits are given in \Cref{fig:qboson_cont_queue_model_s7}.
\begin{figure}[htpb]
	\centering
	\includegraphics[width=.8\textwidth]{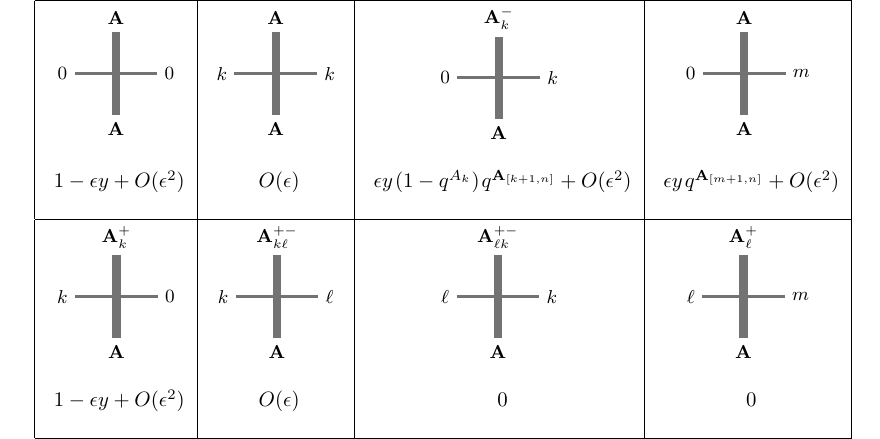}
	\caption{Expansions
	of the weights
	$\WQ^{(-m)}_{q^{-1/2},s,-\epsilon q^{1/2}y/s}\Big\vert_{s=0}$
	as $\epsilon\to0$. Here
	$m < k < \ell \le n$.
	The vertices of the types $(\mathbf{A},k;\mathbf{A},k)$
	and $(\mathbf{A},k;\mathbf{A}^{+-}_{k\ell},\ell)$
	with probabilities of order $\epsilon$ do not occur in the queue vertex model.
	Indeed, to get a nonempty
	input from the left, another event of probability $O(\epsilon)$
	should have occurred at the same instance of the continuous time.}
	\label{fig:qboson_cont_queue_model_s7}
\end{figure}

One readily sees that the vertex weights
\eqref{eq:qBoson_weights_line_steady_limiting}
define
Markov processes in discrete and continuous time with nonnegative transitions.
Similarly to the proof of
\Cref{prop:steady_state_queue},
for
\begin{equation*}
	y_1 > y_2 > \ldots > y_n > 0
\end{equation*}
one can verify that the
queue vertex model
with the weights
$(uy_m;q)_\infty\ssp
\WQ^{(-m),\mathrm{qBos}}_{0,uy_m}$
produces a steady state
(that is, it does not run off to infinity).
Then, following the proof of \Cref{thm:colored_Burke} and the shifting argument after it, one can construct a stationary version of the colored $q$-Boson process on the whole line. We omit the details of the construction as they are similar to the ones in \Cref{sub:qpush_and_density_app}.

\begin{remark}
	\label{rmk:uniq_boson}
	The results of \cite{amir2021tazrp} imply that for $q=0$, a
	translation invariant stationary distribution for the
	$q$-Boson process on the whole line with given densities of
	the colors is unique. We believe that a similar result
	should hold for general $q\in (0,1)$, but its proof does not
	seem present in the existing literature.
\end{remark}

Arguing similarly to the case of the $q$-PushTASEP
(\Cref{sub:qpush_and_density_app}),
we
can compute the colored densities
and currents in terms of the parameters
$y_1>\ldots>y_n>0$.
The colored densities
are the vertical densities $\rho_m^{(v)}$ under the steady state
vertex model
with probability weights $(uy_m;q)_\infty\ssp
\WQ^{(-m),\mathrm{qBos}}_{0,uy_m}$
\eqref{eq:q_boson_weights_degeneration_vertex}.
Let us compute the expected number of arrows leaving the column
$(-m)$.
By color merging described in \Cref{sec:merging_of_colors},
we do not need to distinguish the colors and may assume that $m=n=1$. We have
\begin{equation*}
	(uy_m;q)_\infty \WQ^{(-m),\mathrm{qBos}}_{0,uy_m}
	(a,b;c,d)=
	(uy_m;q)_\infty
	\cdot\frac{(uy_m)^d}{(q;q)_d}
	,
	\qquad
	d\in \mathbb{Z}_{\ge0}.
\end{equation*}
The expected number of arrows is
expressed through the function
\eqref{eq:q_digamma}:
\begin{equation*}
	\sum_{d=0}^{\infty}
	d\cdot
	(uy_m;q)_\infty
	\cdot\frac{(uy_m)^d}{(q;q)_d}
	=
	\phi(uy_m).
\end{equation*}
Thus, we have
\begin{equation}
	\label{eq:q_boson_rho_v}
	\rho^{(v)}_m=
	\phi(uy_m)-\phi(uy_{m+1}).
\end{equation}
Here $y_{n+1}=0$, by agreement.
The colored currents are determined from the
continuous-time queue
vertex model in \Cref{fig:qboson_cont_queue_model_s7},
and have the form
\begin{equation*}
	\rho_m^{(h)}=y_m-y_{m+1}.
\end{equation*}
Similarly to the $q$-PushTASEP,
expressing the currents in terms of the densities is not
an explicit operation.
On the other hand,
a reverse expression is
readily
available. It is obtained from
\eqref{eq:q_boson_rho_v}
by replacing each~$y_j$ with~$\rho^{(h)}_{[j,n]}$.

\appendix

\section{Merging of colors in stationary measures on the line}
\label{sec:merging_of_colors}

Take a stochastic
particle system (mASEP, the $q$-Boson process or the
$q$-PushTASEP) with $n$ colors on the line $\mathbb{Z}$.
By agreement, holes are viewed as particles of color $0$.
If we declare that for some $m=0,1,\ldots,n $,
all particles of colors $m$ and $m+1$
are identified, then we get a stochastic process with $n-1$
colors.
This operation of identifying two colors is a particular case of \emph{color merging}
(see \Cref{def:color_merge_def} below).

A similar color merging can be performed for a stationary distribution of an $n$-colored
particle system, and as a result, we should get a stationary distribution
of a system with $n-1$ colors and modified densities of the colors.
We call this the \emph{color merging property of the stationary distributions}.
Here we explain how one can get this property
on the line $\mathbb{Z}$ using our queue vertex model constructions
from \Cref{sec:quarter_plane}.

\begin{remark}
	\label{rmk:uniqueness_and_ring_for_color_merging}
	For some (but not all) of our particle systems, it is
	proven that a translation invariant stationary
	distribution with given densities of colors is unique
	(see~Remarks~\ref{rmk:uniq_ASEP}, \ref{rmk:uniq_push}, and~\ref{rmk:uniq_boson}).  When this uniqueness is available,
	it implies the color merging property of the stationary
	distributions directly, without reference to queue vertex
	models.
\end{remark}
\begin{remark}[Color merging on the ring]
	\label{rmk:color_merging_ring}
	For all colored particle systems on the ring (mASEP, the
	$q$-Boson process, and the $q$-PushTASEP), we readily have
	uniqueness of the stationary distribution in any given sector
	determined by the number of particles of each color.
	Thus, on the ring the color merging property
	holds automatically.

	However, it is not clear if this color
	merging can be seen at the level of
	queue vertex models on the cylinder. One reason for this is that queue vertex models
	on a cylinder are \emph{not}
	stochastic because they involve summing over input and also output path
	configurations at vertices (see~\Cref{sub:mqueue_states} for more discussion).
	At the same time, color merging involves summing over outputs only, and the
	two summations are not readily compatible.
	In the remainder of \Cref{sec:merging_of_colors},
	we focus only on systems on the line.
\end{remark}

\begin{definition}[Color merging]
	\label{def:color_merge_def}
    Suppose we have a partition of~$\{0,1,\dots,n\}$ into~$k$ disjoint intervals~$I_0,\dots, I_k \subset \{0,1,\dots, n\}$ which are contiguous (that is,~$\max(I_j) = \min(I_{j+1})-1$ for all~$j$).
	The \emph{color merging projection}
	$\pi = \pi_{I_0,\dots, I_k}$
	applied to a state $i\in \{0,1,\ldots,n\}$
	or $\mathbf{A}\in \mathbb{Z}_{\ge0}^n$
	maps it into
	$\pi(i)\in\{ 0,1,\ldots,k\}$ or
	$\pi(\mathbf{A})\in \mathbb{Z}_{\ge0}^k$,
	respectively,
	by assigning to all particles (or~arrows) with colors in each interval $I_j$
	the new
	color $j$.

	For a probability measure $\mu$ on
	$n$-color configurations on $\mathbb{Z}$
	(where the maximal number of particles at a site
	is $1$, $\mathsf{P}$, or $\infty$, depending on the particle system),
	denote by
	$\pi_* \mu$ the pushforward of
	$\mu$ under the color merging projection $\pi$.
\end{definition}

Fix a partition $\left\{ 0,1,\ldots,n  \right\}= I_0\sqcup \ldots \sqcup I_k$
as in \Cref{def:color_merge_def}. For the densities
$\rho_1,\ldots,\rho_n$ of the old colors, denote
by
\begin{equation}
	\label{eq:rho_j_prime_new_densities}
	\rho_j'\coloneqq \sum_{i\in I_j}\rho_i, \qquad  j=1,\ldots,k,
\end{equation}
the densities of the new colors.
In \Cref{sub:qpush_and_density_app,sub:qboson_and_density_app}, we showed that the densities $(\rho_1,\ldots,\rho_n )$ are in one-to-one correspondence with ordered $n$-tuples $y_1>\ldots>y_n>0 $.
More precisely, $y_m$ parametrizes
$\sum_{i\ge m}\rho_i$
(the exact parametrization is different for the $q$-PushTASEP and the $q$-Boson process, but for color merging we do not need these exact formulas).
Therefore,
the new densities
$(\rho_1',\ldots,\rho_k' )$
correspond to the ordered $k$-tuple
\begin{equation}
	\label{eq:y_j_prime_new_densities}
	y'=(y_{\min(I_1)},\dots, y_{\min(I_k)}).
\end{equation}
Here, by agreement, if $I_1$ contains the color $0$ (corresponding to the hole), then we need to remove the parameter $y_{\min(I_1)}=y_0$ from the tuple $y'$.

Denote by $\mu_y$ and $\mu_{y'}$ the stationary measures, respectively, for the $n$- and the $k$-colored particle systems on the whole line $\mathbb{Z}$.

\medskip

Let us record the color merging properties of the
vertex weights
which appeared in \Cref{sec:colored_YBE}. We have:
\begin{align}
	\sum_{i_2, j_2 \colon \pi(i_2) = i_2', \pi(j_2) = j_2'} R_z(i_1, j_1; i_2, j_2) &= R_z(\pi(i_1), \pi(j_1); i_2', j_2') \label{eq:R_merge} ,\\
	\sum_{\mathbf{B}, \ell \colon \pi(\mathbf{B}) = \mathbf{B}', \pi(\ell) = \ell'} L_{s,x}(\mathbf{A}, k; \mathbf{B}, \ell) &= L_{s,x}(\pi(\mathbf{A}), \pi(k); \mathbf{B}', \ell') \label{eq:L_merge} ,\\
	\sum_{\mathbf{C}, \mathbf{D} \colon \pi(\mathbf{C})= \mathbf{C}', \pi(\mathbf{D}) = \mathbf{D}'} W_{x,\mathsf{L},\mathsf{M}}(\mathbf{A},\mathbf{B};\mathbf{C}, \mathbf{D})
	&= W_{x,\mathsf{L},\mathsf{M}}(\pi(\mathbf{A}),\pi(\mathbf{B});\mathbf{C}',\mathbf{D}') \label{eq:W_merge} ,\\
	\sum_{\mathbf{C}, \mathbf{D} \colon \pi(\mathbf{C})= \mathbf{C}', \pi(\mathbf{D}) = \mathbf{D}'} \WQ_{s_1,s_2,u}^{(-m)}(\mathbf{A},\mathbf{B};\mathbf{C}, \mathbf{D})
	&= \WQ_{s_1,s_2,u}^{(-\pi(m))}(\pi(\mathbf{A}),\pi(\mathbf{B});\mathbf{C}',\mathbf{D}'). \label{eq:WQ_merge}
\end{align}
Identity \eqref{eq:R_merge} is immediate. Vertical fusion or a direct verification leads to \eqref{eq:L_merge}. Then by horizontal fusion, \eqref{eq:L_merge} leads to \eqref{eq:W_merge}. Finally, we get \eqref{eq:WQ_merge} from \eqref{eq:W_merge} by the queue specialization defined in \Cref{sub:mqueue_spec}. Note that in \eqref{eq:WQ_merge}, both $m$ and $\pi(m)$ must be strictly positive.

In probabilistic language, identity, say, \eqref{eq:WQ_merge},
can be interpreted as follows. Starting from $(\pi(\mathbf{A}),\pi(\mathbf{B}))$,
to sample $(\mathbf{C}',\mathbf{D}')$
under the $k$-color stochastic weight $\WQ_{s_1,s_2,u}^{(-\pi(m))}$, we may choose any representatives
$(\mathbf{A},\mathbf{B})$ for the input, sample $(\mathbf{C},\mathbf{D})$
under the $n$-color stochastic weight $\WQ_{s_1,s_2,u}^{(-m)}$, and then
project the output $(\mathbf{C},\mathbf{D})$ back to $k$ colors using $\pi$.
The projection $\pi$ ``forgets'' some of the information about the colors, and this
operation is the same as the summation in the left-hand side of \eqref{eq:WQ_merge}.
The other identities \eqref{eq:R_merge}--\eqref{eq:W_merge} have a similar
probabilistic interpretation.

\medskip

Stacking vertices vertically or horizontally results in a Markov mapping which commutes with the projection~$\pi$ in the same way as described in \eqref{eq:R_merge}--\eqref{eq:WQ_merge}.
We need an instance of this stacking for queue vertex models on the whole line.
Let us take a queue vertex model
in $\left\{ -n,\ldots,-1  \right\}\times \mathbb{Z}$
with
empty input from the left and
vertex weights
$\WQ^{(-m)}_{s_1,s_2,u}$ in the column $-m$, $m=1,\ldots,n $.
Assume that the parameters $s_1,s_2,u$ make these weights nonnegative.
More precisely, we assume that the weights are of the form
$\WQ^{(-m)}_{q^{-\mathsf{P}/2},s_m^{(v)},u/v_m}$ for the
$q$-PushTASEP \eqref{eq:qPushTASEP_weights_line_steady},
or $(uy_m;q)_\infty\ssp
\WQ^{(-m),\mathrm{qBos}}_{0,uy_m}$
for the $q$-Boson process \eqref{eq:qBoson_weights_line_steady_limiting}.
In both cases, we can define the steady state distribution of this model
as in \Cref{prop:steady_state_queue}. Denote the corresponding random
tuple by~$\mathbf{M}=(\mathbf{M}(-n),\ldots,\mathbf{M}(-1) )$.

\begin{lemma}
	\label{lemma:quad_merge}
	Fix a partition $\left\{ 0,1,\ldots,n  \right\}= I_0\sqcup \ldots \sqcup I_k$
	as in \Cref{def:color_merge_def}, and let $\pi$ be the corresponding projection.

	Let the random configurations~$\mathbf{V} = (\mathbf{V}(1),\mathbf{V}(2),\dots)$ and~$\tilde{\mathbf{V}} =  (\tilde{\mathbf{V}}(1),\tilde{\mathbf{V}}(2),\dots)$ be sampled from the vertex models shown
	on the left and right in
	\Cref{fig:stacking_for_merge}, respectively.
	That is, $\mathbf{V}$ is the output of the original $n$-color queue vertex model
	run in steady state. The model for
	$\tilde{\mathbf{V}}$ has input $\pi(\mathbf{M})$,
	and $k$-color queue vertex weights $\WQ^{(-\pi(m))}_{s_1,s_2,u}$
	in the column $-m$, $m=1,\ldots,n $.
	Then
	\begin{equation*}
		\pi(\mathbf{V}) \stackrel{d}{=} \tilde{\mathbf{V}} .
	\end{equation*}
	Moreover, the distribution of the $n$-tuple $\pi(\mathbf{M})$
	is the steady state for the system of $n$
	$k$-colored queues in tandem in the right side of
	\Cref{fig:stacking_for_merge}.
\end{lemma}
\begin{proof}
	Both claims follow from an inductive application of the color
	merging property \eqref{eq:WQ_merge}.
\end{proof}

\begin{figure}[htpb]
	\centering
	\includegraphics[scale=.9]{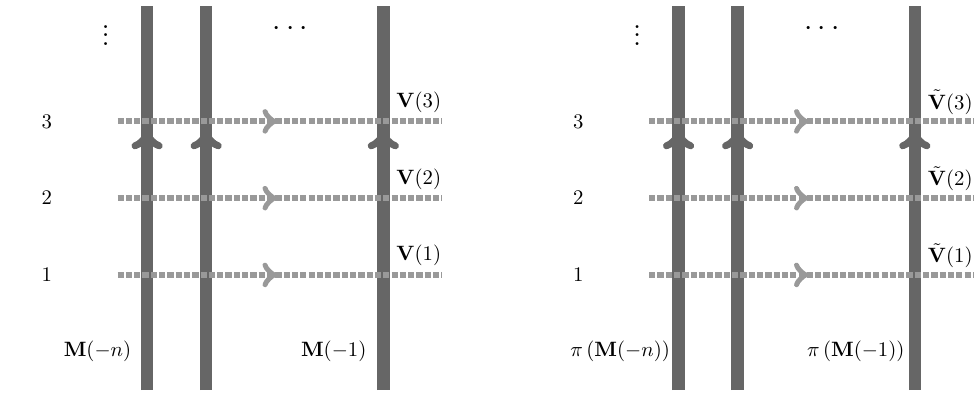}
	\caption{Color merging applied to queue vertex models; see
	\Cref{lemma:quad_merge}. There are $n$ columns in both
	figures.}
	\label{fig:stacking_for_merge}
\end{figure}

\begin{lemma}
	\label{lemma:merging_1_and_0_in_first_column}
	Let $\pi$ be a
	projection
	which merges colors $1$ and $0$.
	Under $\pi$,
	the $n$-color
	queue vertex weights
	$\WQ^{(-1)}_{s_1,s_2,u}(\mathbf{A},\mathbf{B};\mathbf{C},\mathbf{D})$
	in the rightmost column turn into
	the $(n-1)$-color fused stochastic weights
	\eqref{eq:fully_fused_stochastic_weights}
	\begin{equation}
		\label{eq:prelimit_from_merging}
		W
		_{s_1s_2^{-1}u,\mathsf{L},\mathsf{M}}
		(\mathbf{A}',\mathbf{B}';\mathbf{C}',\mathbf{D}')
		\big\vert_{q^{-\mathsf{L}}=s_1^2,\ q^{-\mathsf{M}}=s_2^{2}}
		,
	\end{equation}
	where
	$\mathbf{A}'=(A_2,\ldots,A_n )\in \mathbb{Z}_{\ge0}^{n-1}$,
	and similarly $\mathbf{B}',\mathbf{C}',\mathbf{D}'$,
	and
	\begin{equation*}
		\mathbf{A}=(\infty,\mathbf{A}'),
		\qquad
		\mathbf{B}=(0,\mathbf{B}'),
		\qquad
		\mathbf{C}=(\infty,\mathbf{C}'),
		\qquad
		\mathbf{D}=(D_1,\mathbf{D}').
	\end{equation*}
	Here
	$B_1=0$ because in a queue vertex model, no arrows of color $1$ can enter the column $(-1)$.
\end{lemma}
\begin{proof}
	We use the explicit expression
	\eqref{eq:queue_spec_fully_fused}
	for
	$\WQ^{(-1)}_{s_1,s_2,u}$.
	Let
	$\mathbf{P}'=(P_2,\ldots,P_n )$,
	which is a part of the summation index
	in $\WQ^{(-1)}_{s_1,s_2,u}$. We have
	$P_1=0$ because $P_1\le B_1$.
	The projection $\pi$ involves the summation over
	$D_1$ from $0$ to $\infty$. The latter reduces to the $q$-binomial theorem
	\cite[(1.3.2)]{GasperRahman}:
	\begin{align*}
			&\sum_{D_1=0}^{\infty}
			\WQ^{(-1)}_{s_1,s_2,u}(\mathbf{A},\mathbf{B};\mathbf{C},\mathbf{D})
			=
			\frac{(s_1^{-1}s_2 u; q)_{\infty}}{(s_1s_2u ; q)_{\infty}}
		\sum_{\mathbf{P}'}
		\frac{(s_1s_2/u ; q)_{|\mathbf{P}'|}
		(s_1u/s_2 ; q)_{|\mathbf{B}'-\mathbf{P}'|}}
		{(s_1^2 ; q)_{|\mathbf{B}'|}}\ssp
		\\&\hspace{10pt}\times
		q^{\sum_{2\le i<j\le n}\left(B_i-P_i\right) P_j}
		(s_1s_2u^{-1})^{|\mathbf{B}'|-|\mathbf{P}'|}
		(s_1^{-1}s_2u)^{|\mathbf{D}'|}
		\prod_{i=2}^n
		\frac{(q;q)_{B_i}}{(q;q)_{P_i}(q;q)_{B_i-P_i}}
		\frac{(q;q)_{C_i-P_i+D_i}}{(q;q)_{C_i-P_i}(q;q)_{D_i}}
		\\&\hspace{10pt}\times
		\ssp
		q^{\sum_{2\le i<j\le n} D_i (C_j-P_j)}
		(s_1^2 ; q)_{|\mathbf{D}'|}
		\sum_{D_1=0}^{\infty}
		q^{D_1 (|\mathbf{C}'|-|\mathbf{P}'|)}
		(us_2/s_1)^{D_1}
		\frac{
		(s_1^2 q^{|\mathbf{D}'|} ; q)_{D_1}}
		{(q;q)_{D_1}}
		\\&
		=
		\frac{(s_1^{-1}s_2 u; q)_{\infty}}{(s_1s_2u ; q)_{\infty}}
		\sum_{\mathbf{P}'}
		\frac{(s_1s_2/u ; q)_{|\mathbf{P}'|}
		(s_1u/s_2 ; q)_{|\mathbf{B}'-\mathbf{P}'|}}
		{(s_1^2 ; q)_{|\mathbf{B}'|}}\ssp
		\\&\hspace{10pt}\times
		q^{\sum_{2\le i<j\le n}\left(B_i-P_i\right) P_j}
		(s_1s_2u^{-1})^{|\mathbf{B}'|-|\mathbf{P}'|}
		(s_1^{-1}s_2u)^{|\mathbf{D}'|}
		\prod_{i=2}^n
		\frac{(q;q)_{B_i}}{(q;q)_{P_i}(q;q)_{B_i-P_i}}
		\frac{(q;q)_{C_i-P_i+D_i}}{(q;q)_{C_i-P_i}(q;q)_{D_i}}
		\\&\hspace{10pt}\times
		\ssp
		q^{\sum_{2\le i<j\le n} D_i (C_j-P_j)}
		(s_1^2 ; q)_{|\mathbf{D}'|}
		\frac{(s_1s_2u q^{|\mathbf{C}'|-|\mathbf{P}'|+|\mathbf{D}'|};
		q)_{\infty}}
		{(s_1^{-1}s_2u q^{|\mathbf{C}'|-|\mathbf{P}'|};q)_{\infty}}.
	\end{align*}
	Canceling out the infinite $q$-Pochhammer symbols, we
	continue the computation as follows:
	\begin{align*}
		&=
		\sum_{\mathbf{P}'}
		\frac{(s_1^{-1}s_2 u; q)_{|\mathbf{C}'|-|\mathbf{P}'|}
		(s_1^2 ; q)_{|\mathbf{D}'|}}
		{(s_1s_2u ; q)_{|\mathbf{C}'|-|\mathbf{P}'|+|\mathbf{D}'|}}
		\frac{(s_1s_2/u ; q)_{|\mathbf{P}'|}
		(s_1u/s_2 ; q)_{|\mathbf{B}'-\mathbf{P}'|}}
		{(s_1^2 ; q)_{|\mathbf{B}'|}}
		\ssp
		(s_1s_2u^{-1})^{|\mathbf{B}'|-|\mathbf{P}'|}
		(s_1^{-1}s_2u)^{|\mathbf{D}'|}
		\ssp
		\\&\hspace{40pt}\times
		q^{\sum_{2\le i<j\le n}\left(B_i-P_i\right) P_j}
		q^{\sum_{2\le i<j\le n} D_i (C_j-P_j)}
		\prod_{i=2}^n
		\frac{(q;q)_{B_i}}{(q;q)_{P_i}(q;q)_{B_i-P_i}}
		\frac{(q;q)_{C_i-P_i+D_i}}{(q;q)_{C_i-P_i}(q;q)_{D_i}}
		.
	\end{align*}
	From \eqref{eq:fully_fused_stochastic_weights},
	one readily sees that the resulting expression matches
	\eqref{eq:prelimit_from_merging}, and we are done.
\end{proof}

Now, we can formulate and prove our main statement about the color merging property of the stationary distributions.
\begin{proposition}
	\label{prop:color_merging_property}
	We have
	$\pi_* \mu_y = \mu_{y'}$.
	Here $\mu_y$ is the stationary distribution of the $n$-colored $q$-PushTASEP or the $q$-Boson process on $\mathbb{Z}$
	with the densities of the colors depending
	on the parameters $y_1>\ldots>y_n>0$
	via
	\eqref{eq:h_current_qpush} or \eqref{eq:q_boson_rho_v},
	respectively. The distribution
	$\mu_{y'}$ is stationary for the corresponding $k$-colored
	system, and the parameters $y'$ are obtained from $y$ by merging
	as in \eqref{eq:y_j_prime_new_densities}.
\end{proposition}
\begin{proof}
	Arguing inductively, it suffices to consider the merging any two consecutive colors
	$m$ and $m+1$. Case (1) with $m=0$ is special, we treat it separately first.
	For $m\ge1$, we need to show that the output of the column $-(m+2)$ (distributed as $\mu_{(y_{m+2},\ldots,y_n)}$), passed through two consecutive columns with infinitely many arrows of color $m$ and parameters $y_{m+1}$ and $y_m$, is distributed as $\mu_{(y_{m},\ldots,y_n)}$. After that, we can pass the output of the column $(-m)$ into further columns, and the final output will be distributed as $\mu_{y'}$ by the very definition. Thus, in Case (2) it suffices to take $m=1$ and merge the colors $1$ and $2$.

	\medskip\noindent
	\textbf{Case (1). Step 1.}
	Consider the system of $n$ columns of queue vertex models (``queues in tandem'') which produces the $n$-color stationary distribution $\mu_y$. Let us denote the output of this system by $\mathbf{V}=(\mathbf{V}(1),\mathbf{V}(2),\ldots )$. We know that the color merged output $\pi(\mathbf{V})$ is the same as the output of $n$ queues with $n-1$ colors, where the columns $-2,\ldots,-n $ do not change, and $\pi$ is applied to the rightmost column $(-1)$. Indeed, $\pi$ affects only the column $(-1)$, and erases the color $1$ which does not appear in columns $(-m)$, $m\ge2$. It thus suffices to show that $\pi(\mathbf{V})$ is distributed as $\mu_{y'}$, that is, it is stationary for the $(n-1)$-color interacting particle system.

	\smallskip\noindent
	\textbf{Case (1). Step 2.}
	Applying \Cref{lemma:merging_1_and_0_in_first_column}, we see that the weights in the rightmost column $(-1)$ become the $q$-PushTASEP or $q$-Boson specializations of the $(n-1)$-color weights \eqref{eq:prelimit_from_merging}. More precisely, for the $q$-PushTASEP, we have the following weights in the column $(-2)$ and the column $(-1)$:
	\begin{equation}
		\label{eq:qpush_line_merge_2_weights}
		\WQ^{(-2)}_{q^{-\mathsf{P}/2},s,-q^{1-\mathsf{P}/2}s^{-1} uy_2}\Big\vert_{s=0}
		\quad
		\textnormal{and}
		\quad
		W_{-q^{1-\mathsf{P}}s^{-2}uy_1,\mathsf{P},\mathsf{M}}
		\Big\vert_{s=0},
	\end{equation}
	where $q^{-\mathsf{M}}=s^2$ (see \Cref{sub:qpush_and_density_app}).
	Note in particular that for the
	$q$-PushTASEP, the queue weights
	$\WQ^{(-m),\mathrm{queue}}_{\xi,\alpha_m,\nu_m}$ (see \eqref{eq:qPushTASEP_weights_line_steady})
	do not scale with
	$\epsilon$ which entered the weights $\WQ^{(-m),\mathrm{line}}_{\alpha_m,\nu_m}$ in the quadrant through the scaling \eqref{eq:definition_of_y_m_for_qpushTASEP}.

	For the $q$-Boson process, these weights are, respectively,
	\begin{equation}
		\label{eq:qBos_line_merge_2_weights}
		\WQ^{(-2)}_{\epsilon,s,\epsilon s^{-1}uy_2}\Big\vert_{\epsilon\to0 \text{ then } s=0}
		\quad
		\textnormal{and}
		\quad
		W_{(\epsilon/s)^2 uy_1,\mathsf{L},\mathsf{M}}\Big\vert_{\epsilon\to0 \text{ then } s=0},
	\end{equation}
	where $q^{-\mathsf{L}}=\epsilon^2$
	and $q^{-\mathsf{M}}=s^2$ (see \Cref{sub:qboson_and_density_app}).

	Let us now choose the auxiliary weights with which \eqref{eq:qpush_line_merge_2_weights} or \eqref{eq:qBos_line_merge_2_weights}, respectively, satisfy the Yang--Baxter equation. They are found from \Cref{prop:queue_YBE}. The auxiliary weights for the $q$-PushTASEP and the $q$-Boson are exactly the same, and they have the form
	\begin{equation}
		\label{eq:aux_weights_qpush_qboson}
		\WQ^{(-2)}_{s,s,y_2/y_1}\Big\vert_{s=0}.
	\end{equation}
	One can check that the weights
	\eqref{eq:qpush_line_merge_2_weights},
	\eqref{eq:qBos_line_merge_2_weights}, and
	\eqref{eq:aux_weights_qpush_qboson}
	are nonnegative under our restrictions on the parameters (in particular, recall that $y_1>y_2>0$).

	\begin{figure}
	\centering
	\includegraphics[width=.85\textwidth]{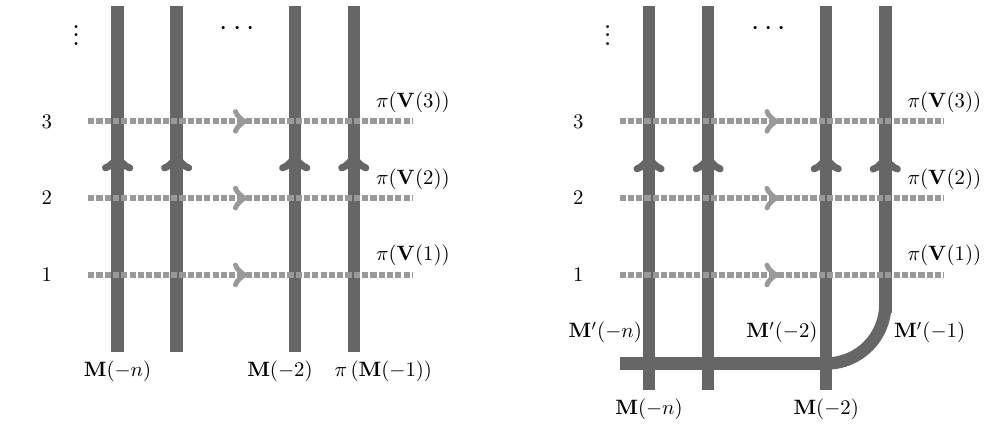}
	\caption{Two ways to sample the random configuration $\pi(\mathbf{V})$ in Case (1) in the proof of \Cref{prop:color_merging_property}.
	Left: $\pi(\mathbf{V})$ is the output of a system of $n$ queues with $(n-1)$ colors in the steady state, where $\pi$ is applied only in the rightmost column. Right: We add auxiliary vertices with the weights $\WQ^{(-2)}_{0,0,y_m/y_1}$ at the bottom of each column $(-m)$, $m=2,\ldots,n $. The incoming arrow configurations are empty on the left, and the $(n-1)$-color steady state $(\mathbf{M}(-n),\ldots,\mathbf{M}(-2))$ at the bottom. The partition function on the right satisfies the Yang--Baxter equation at the triangular intersection of the lines.}
	\label{fig:stat_2}
	\end{figure}

	\smallskip\noindent
	\textbf{Case (1). Step 3.}
	Let us show that the random output $\pi(\mathbf{V})$ may be sampled using another vertex model which is displayed in \Cref{fig:stat_2}, right. We claim that
	\begin{equation}
		\label{eq:M_pi_M_stationarity}
		\bigl(\mathbf{M}(-n),\ldots,\mathbf{M}(-2),\pi(\mathbf{M}(-1)) \bigr)
		\stackrel{d}{=}
		\bigl(\mathbf{M}'(-n),\ldots,\mathbf{M}'(-2),\mathbf{M}'(-1) \bigr).
	\end{equation}
	By \Cref{lemma:steady_state_property}, the distribution of $(\mathbf{M}(-n),\ldots,\mathbf{M}(-2))$ in \Cref{fig:stat_2}, right, can be generated by infinitely many dashed horizontal lines below the picture (carrying the corresponding queue vertex weights). Using the Yang--Baxter equation, we may bring an arbitrary number, say, $K$, of these dashed horizontal lines into the space between the auxiliary line (solid horizontal line) and the output $(\mathbf{M}'(-n),\ldots,\mathbf{M}'(-2),\mathbf{M}'(-1) )$. See
	\Cref{fig:step3_merging} for an illustration.
	Taking the limit as $K\to\infty$, we get \eqref{eq:M_pi_M_stationarity} because its left-hand side is the steady state of the system of $n$ queues with $(n-1)$ colors.

	\begin{figure}[htpb]
		\centering
		\includegraphics[width=.7\textwidth]{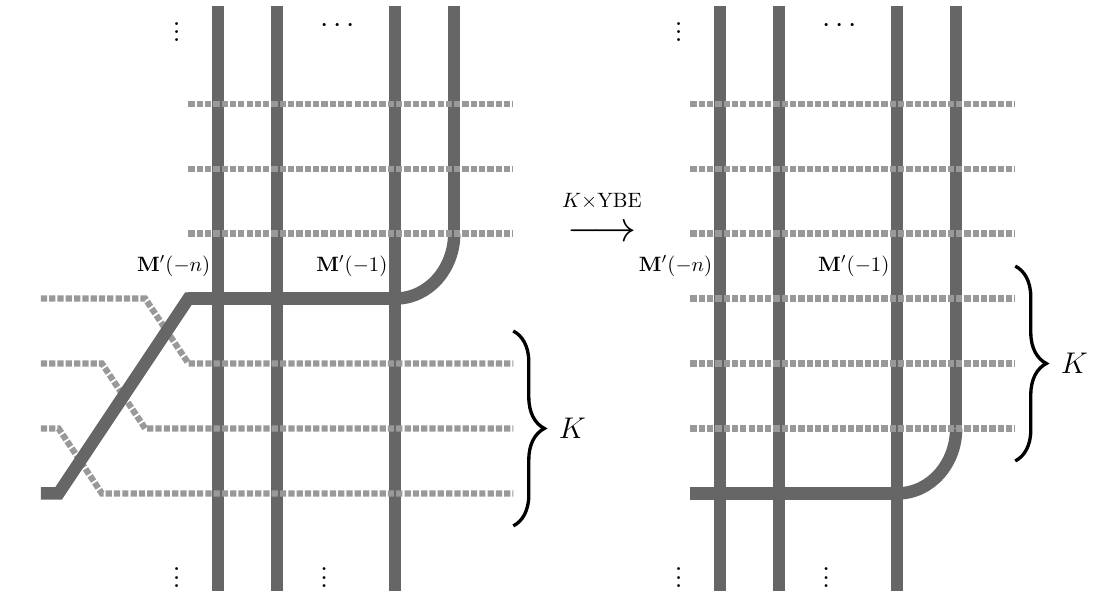}
		\caption{Moving of $K$ dashed horizontal lines in step 3
		of the proof of \Cref{prop:color_merging_property}. In the picture, $K=3$. We add $K$ empty cross
		vertices to the left, and drag them through to the right. In the limit as $K\to+\infty$,
		the distribution of $\mathbf{M}'$ becomes the same as of $\mathbf{M}$.}
		\label{fig:step3_merging}
	\end{figure}

	\smallskip\noindent
	\textbf{Case (1). Step 4.}
	Now, using the Yang--Baxter equation, we may move the rightmost vertical line in \Cref{fig:stat_2}, right, all the way to the left of the column $-n$. There, this vertical line can be removed because it carries no arrows.
	This application of the Yang--Baxter equation transforms the lattice
	from the right panel of \Cref{fig:step3_merging}
	to the left one.
	The resulting output configuration from the $(n-1)$-column system
	is distributed as $\pi(\mathbf{V})$, and we are done.
	This shows that $\pi(\mathbf{V})$ has the same distribution as the output of the system of $n-1$ queues with $(n-1)$ colors and parameters $y_2> \ldots > y_n>0$. Thus, $\pi(\mathbf{V})$ is distributed as
	$\mu_{(y_2,\ldots,y_n)}=\mu_{y'}$, as desired.

	The applications of the Yang--Baxter equation in Steps \textbf{3} and \textbf{4} are
	similar to what we used in the colored Burke's theorem
	in \Cref{sub:YBE_quarter_plane}.

	\medskip\noindent
	\textbf{Case (2).} Now let us consider the merging of colors $1$ and $2$. Consider the
	system of $n-1$ queues in tandem, which have $n-1$ colors, and parameters
	$y_2>\ldots>y_n $. Denote its output by $\mathbf{V}'=(\mathbf{V}'(1),\mathbf{V}'(2),\ldots )$;
	it is distributed as the $(n-1)$-colored stationary distribution $\mu_{(y_2,\ldots,y_n )}$.
	We need to replace the color $2$ by $1$ in $\mathbf{V}'$ and pass it as an input into the column $(-1)$ with parameter $y_1$.
	To complete the proof, it suffices to show that the output $\mathbf{V}$
	of the column $(-1)$ in this scenario has the
	distribution $\mu_{(y_1,y_3,\ldots,y_n )}$, see \eqref{eq:y_j_prime_new_densities}.

	Notice that by \Cref{lemma:queue_indep_of_Bm},
	the queue vertex weights in the column $(-1)$ do not depend on the
	arrows of color $1$ incoming from the left. Therefore, instead of
	replacing the color $2$ by $1$ in~$\mathbf{V}'$, we may replace
	the color $2$ by $0$, and pass the result into the column $(-1)$.
	Denote by $\pi^{\circ}$
	the projection which merges the colors $2$ and $0$.
	By Case (1), $\pi^{\circ}(\mathbf{V}')$ has
	$n-2$ colors $3,\ldots,n$ and is distributed as
	$\mu_{(y_3,\ldots,y_n )}$.
	Passing $\pi^{\circ}(\mathbf{V}')$ through the column $(-1)$ with the parameter $y_1$
	adds the color~$1$ and,
	by the very definition of the queue vertex model,
	produces $\mathbf{V}$ with the distribution $\mu_{(y_1,y_3,\ldots,y_n )}$.
	This completes the proof.
\end{proof}

\bibliographystyle{alpha}
\bibliography{bib}

\medskip

\textsc{A. Aggarwal, Columbia University, New York, NY, USA and Clay Mathematics Institute, Denver, CO, USA}

E-mail: \texttt{amolagga@gmail.com}

\medskip

\textsc{M. Nicoletti, Massachusetts Institute of Technology, Cambridge, MA}

E-mail: \texttt{mnicolet@mit.edu}

\medskip

\textsc{L. Petrov, University of Virginia, Charlottesville, VA}

E-mail: \texttt{lenia.petrov@gmail.com}

\end{document}